\documentclass[10pt]{amsart}
\usepackage{euscript}
\usepackage{amssymb}
\usepackage{amsmath}
\usepackage{enumitem}
\usepackage{epic}
\usepackage{graphics}
\usepackage{epsfig}
\usepackage{color}
\usepackage{stmaryrd}
\usepackage{amscd,euscript}
\usepackage[frame,cmtip,curve,arrow,matrix,line,graph]{xy}

\usepackage[colorlinks]{hyperref}

\usepackage{setspace}

\usepackage{tikz}
\usepackage{tikz-cd}
\usepackage{pgfplots}
\tikzset{
	cross/.pic = {
		\draw[rotate = 45] (-#1,0) -- (#1,0);
		\draw[rotate = 45] (0,-#1) -- (0, #1);
	}
}

\usepackage[
margin=9mm,
marginparwidth=17mm,     
marginparsep=3mm,       
]{geometry}

\numberwithin{equation}{section}
\theoremstyle{plain}
\newtheorem{theorem}{Theorem}[section]
\newtheorem{lemma}[theorem]{Lemma}
\newtheorem{example}[theorem]{Example}
\newtheorem{proposition}[theorem]{Proposition}
\newtheorem{prop}[theorem]{Proposition}

\newtheorem{corollary}[theorem]{Corollary}
\newtheorem{cor}[theorem]{Corollary}

\theoremstyle{definition}

\newtheorem{defn}[theorem]{Definition}

\newtheorem{remark}[theorem]{Remark}
\newtheorem{rem}[theorem]{Remark}

\newtheorem{Example}[theorem]{Example}
\theoremstyle{remark}

\numberwithin{equation}{section}

\newcommand{\Acknowledgements}{{\em Acknowledgements.} }

\newcommand{\scrH}{\EuScript{H}}

\newcommand{\scrT}{\EuScript{T}}

\newcommand{\calM}{\mathcal{M}}

\newcommand{\calL}{\mathcal{L}}

\newcommand{\calT}{\mathcal{T}}

\newcommand{\scrC}{\EuScript{C}}

\newcommand{\scrL}{\EuScript{L}}

\newcommand{\calO}{\mathcal{O}}

\newcommand{\bE}{\mathbb{E}}
\newcommand{\bW}{\mathbb{W}}

\newcommand{\bA}{\mathbb{A}}
\newcommand{\bG}{\mathbb{G}}

\newcommand{\fg}{\mathfrak{g}}
\newcommand{\fF}{\mathfrak{F}}
\newcommand{\fC}{\mathfrak{C}}

\newcommand{\cT}{\mathcal{T}}

\newcommand{\bF}{\mathbb{F}}
\newcommand{\bR}{\mathbb{R}}
\newcommand{\bZ}{\mathbb{Z}}
\newcommand{\bQ}{\mathbb{Q}}
\newcommand{\bC}{\mathbb{C}}
\newcommand{\bN}{\mathbb{N}}
\newcommand{\bP}{\mathbb{P}}

\newcommand{\cdbar}{\mathrm{\overline{\partial}}}
\newcommand{\Sym}{\mathrm{Sym}}
\newcommand{\Pic}{\mathrm{Pic}}
\newcommand{\id}{\mathrm{id}}

\renewcommand{\ker}{\mathrm{ker}}

\newcommand{\Glo}{\mathsf{Glo}}

\def\lra#1{\overset{#1}{\longrightarrow}}

\renewcommand{\hbar}{\overline{\frak{h}}}

\newcommand{\Aut}{\mathrm{Aut}}

\newcommand{\Symp}{\mathrm{Symp}}
\newcommand{\Diff}{\mathrm{Diff}}

\newcommand{\scrE}{\EuScript{E}}
\newcommand{\scrF}{\EuScript{F}}

\newcommand{\scrD}{\EuScript{D}}

\newcommand{\ccM}{\mathcal M}
\newcommand{\ccC}{\mathcal C}

\newcommand{\bL}{\mathbb{L}}

\newcommand{\bK}{\mathbb{K}}

\numberwithin{equation}{section}
\renewcommand{\leq}{\leqslant}
\renewcommand{\geq}{\geqslant}

\newcommand{\C}{\mathbb C}

\newcommand{\Z}{\mathbb{Z}}

\newcommand{\bT}{\mathbb{T}}

\newcommand{\ev}{\operatorname{ev}}

\newcommand{\ver}{\mathrm{ver}}
\newcommand{\pre}{\mathrm{pre}}
\newcommand{\reg}{\mathrm{reg}}

\newcommand{\st}{\operatorname{st}}

\newcommand{\rk}{\operatorname{rk}}

\newcommand{\ccMbar}{\overline{\mathcal{M}}}
\newcommand{\ccCbar}{\overline{\mathcal{C}}}

\title[Gromov-Witten invariants in complex and Morava-local $K$-theories]{Gromov-Witten invariants in \\ complex and Morava-local $K$-theories}
\author{Mohammed Abouzaid, Mark McLean, Ivan Smith}
\address{Mohammed Abouzaid, Stanford}
\email{abouzaid@stanford.edu}
\address{Mark McLean, Stony Brook}
\email{mark.mclean@stonybrook.edu}
\address{Ivan Smith, Cambridge}
\email{is200@cam.ac.uk}
\date{July 2024}

\begin{document}
	\maketitle
	
	\begin{abstract}

Given a closed symplectic manifold $X$, we construct Gromov-Witten-type invariants valued both in (complex) $K$-theory and in any complex-oriented cohomology theory $\bK$ which is $K_p(n)$-local for some Morava $K$-theory $K_p(n)$. We show that these invariants satisfy a version of the Kontsevich-Manin axioms, extending Givental and Lee's work for the quantum $K$-theory of complex projective algebraic varieties. In particular, we prove a Gromov-Witten type splitting axiom, and hence define quantum $K$-theory and quantum $\bK$-theory as commutative deformations of the corresponding (generalised) cohomology rings of $X$; the definition of the quantum product  involves the formal group of the underlying cohomology theory. The key geometric input of these results is a construction of global Kuranishi charts for moduli spaces of stable maps of arbitrary genus to $X$. On the algebraic side, in order to establish a common framework covering both ordinary $K$-theory and $K_p(n)$-local theories, we introduce a formalism of `counting theories' for enumerative invariants on a category of global Kuranishi charts.  
	\end{abstract}
	{ \setcounter{tocdepth}{1}
		\hypersetup{linkcolor=black}
		\tableofcontents
	}

	\section{Introduction}
	
	Cohomological Gromov-Witten invariants for smooth algebraic varieties were introduced in \cite{LiTian} and developed in \cite{Behrend}, based on the theory of virtual fundamental classes for moduli problems with perfect obstruction theories.  The theory of Kuranishi atlases for moduli spaces of pseudo-holomorphic curves led to a generalisation of Gromov-Witten theory to arbitrary compact symplectic manifolds, cf.  \cite{Fukaya-Ono, Li-Tian, Pardon2016algebraic}.  Corresponding $K$-theoretic Gromov-Witten invariants were introduced in the algebraic setting in \cite{Givental, Lee-QKFoundations}, based on the virtual structure sheaf, but their symplectic counterparts were never developed;  the local-to-global construction of the symplectic virtual fundamental class did not seem well-suited to constructing a replacement for the virtual structure sheaf.  
	
	Fix an arbitrary compact symplectic manifold $(X,\omega)$. 
	\begin{itemize}
		\item  In Section \ref{sec:counting-theories} we set up a general formalism of `counting theories' on a category of global Kuranishi charts, which may be useful elsewhere.  We explain how to obtain such theories from both Morava-local cohomology theories and from complex $K$-theory.
		
		\item In Section \ref{sec:glob-kuran-charts}, we extend the theory of  `global Kuranishi charts'  developed in \cite{AMS-Hamiltonian} for moduli spaces of rational stable maps with $X$ to the case of stable maps of arbitrary genus\footnote{Global charts for higher genus curves have been independently constructed, by a slightly different method, by Amanda Hirschi and Mohan Swaminathan \cite{HS}.}.  The basic construction is implemented in Definition \ref{defn Kurnishi chart}, while a variant that we use to prove independence of choices (in an appropriate sense) is given in Section \ref{sectiongeneralconstruction}. 
		
		\item In Section \ref{sec:grom-witt-invar} we define Gromov-Witten invariants
		\[
		I_{g,n}(\beta) \in H_*(\ccMbar_{g,n} \times X^n;\bK)
		\] 
		when $\bK$ is either (orbifold) complex $K$-theory or a Morava-local complex-oriented theory. Taking $\bK$ to be the Morava $K$-theory $K_p(n)$ at large height $n\gg0$  approximates having a theory of Gromov-Witten invariants over a finite field $\bF_p$.\footnote{Following ideas of Fukaya and Ono, Bai and Xu \cite{BaiXu} introduced  integral-valued Gromov-Witten invariants in genus $0$. We expect that their work can be extended to higher genus, but, even in genus $0$, the comparison with the $\bF_p$ invariants that we produce is not currently understood.}
		
		\item  In Section \ref{sec:quant-gener-cohom} we obtain associative quantum ordinary and Morava $K$-theory rings (in the Morava case at primes $p>2$) for general symplectic manifolds. A fundamental insight of Givental is that these algebraic structures require the use of corrected fundamental classes\footnote{More precisely, in quantum $K$-theory Givental \cite{Givental} and Lee \cite{Lee:QKtheory} correct the metric, rather than the fundamental class;  correcting the latter seems essential in the  general case.}  on the moduli spaces of curves, as discussed in Section \ref{Sec:mesosphere}.
		
		\item  In Section \ref{Sec:operadic} we establish a splitting axiom (Theorem \ref{thm:operad} and Proposition \ref{prop:splitting}) for an `operational suboperad' of invariants $I_{g,n}(\beta)$, which involves `tautological integrals' over strata of split curves, organised by the formal group associated to $\bK$.
	\end{itemize}
	
	The main results are summarised by Theorem \ref{thm:basicglobalchart}, Corollary \ref{cor:charts_summary} and Theorems \ref{thm:associative} and \ref{thm:operad}.
	
	Three technical aspects of the construction may be worth highlighting. First, we find it convenient to leverage known results in algebraic cobordism for spaces of stable maps to projective space. Second, the counting theory for complex $K$-theory relies on Atiyah's `transverse-equivariant $K$-theory' \cite{Atiyah}, which is the home for symbols of operators which are elliptic in the directions tranverse to the orbits of a compact Lie group action; in contrast, for Morava-local theories we work Borel-equivariantly. Finally, formal group laws play an essential role in extracting algebraic structures from the virtual fundamental classes of the moduli spaces of maps;  this is consistent with Givental's approach to the subject \cite{CoatesGivental2006}.
	
	\begin{remark}
		For the quantum product / splitting axiom we restrict to Morava-local theories at primes $p>2$, since the ring spectrum $K_2(n)$ is not homotopy commutative.  We prove a splitting but not a genus reduction axiom for Gromov-Witten theory, to avoid the more involved combinatorics of boundary divisors under self-gluing. That combinatorial complexity is largely irrelevant for the simplest (i.e. additive and multiplicative) formal groups. The results presented here yield the `full' Gromov-Witten axioms without descendents for rational cohomology and complex $K$-theory, with the exception of the forgetful axiom which, in complex $K$-theory, we establish only in genus zero, cf. Remark \ref{rmk:forgetful in K-theory}.
	\end{remark}
	
	
	\begin{Example}\label{toyexample}
	Non-vanishing of Gromov-Witten invariants defined over other cohomology theories can be used to constrain those automorphisms of cohomology of $(X,\omega)$ arising from symplectomorphisms,  and hence the image of $\pi_0\Symp(X) \to \pi_0\Diff(X)$. As an example, if $X = Y \times S^2$, any diffeomorphism $f: X \to X$ preserving the class of the $[S^2]$ and its Poincar\'e dual induces an endomorphism of $H^*(X;\bZ) = H^*(Y;\bZ) \oplus H^{*-2}(Y;\bZ)$ which is upper triangular (since preserving cup-product). Now suppose $(X,\omega) = (Y,\omega_Y) \times (S^2,\omega)$ is a symplectic product, equipped with a product almost complex structure, and that $f$ is a symplectomorphism. There is a degree $-2$-endomorphism of any $K_p(n)$-local cohomology of $X$ coming from the sweepout map associated to the obvious moduli space of rational curves $\{pt\} \times \bP^1$: more precisely, one considers the evaluation map $\ccMbar := \ccMbar_{0,2}(X,[S^2]) \to X\times X$ from the moduli space of two-pointed curves in the class $\{pt\}\times [S^2]$, and uses push-pull and the existence of an $\bE$-valued virtual class for the moduli space to define a map
	\[
	\bE^i(X) \to \bE^{i-2}(X), \ \alpha \mapsto \mathrm{PD}\left( (\ev_{\infty})_* (\ev_0^*\alpha \cap [\ccMbar]^{vir}) \right)
	\]
	where $\mathrm{PD}$ denotes Poincar\'e duality in the theory $\bE$.  If a symplectomorphism $f$ preserves $[S^2]$ and its dual, in $\bE$-theory, then it acts on $\bE^*(X)$ intertwining this automorphism.
	
	By varying the prime $p$ of a Morava-local theory and working at large height, one obtains constraints on the cohomology automorphism over $\bZ$: more precisely, lifting to $\bZ$ uses constraints coming from the Morava theories $K_{p^k}(n)$ for all $n,k$ and an application of the Chinese remainder theorem. The upshot is that the cohomology endomorphism $f^*$ associated to a symplectomorphism preserving the $[S^2]$ and its Poincar\'e dual is lower  as well as upper triangular. For instance, if $H^*(Y;\bZ)$ contains $p$-torsion in  degrees $i$ and $i-2$ (say classes $x,y$)  then  one cannot have $f^*(x) = x+u\cdot y$ with $u$ dual to $[S^2]$.  (Such an integral result, rather than a result in $\bE$-theory, could also be derived from \cite{BaiXu}.) This yields constraints on $f^*$ not detected by rational Gromov-Witten theory, and gives a new obstruction for a class in $\pi_0\Diff(X)$ to lift to $\pi_0\Symp(X)$.  
	\end{Example}

	Example \ref{toyexample} aside, this paper is foundational\footnote{La lutte elle-m\^eme vers les sommets suffit \`a remplir un c\oe ur d'homme; il faut imaginer Sisyphe heureux.}; a first application of these ideas was given in \cite{AMS-Hamiltonian}, and further applications will appear elsewhere.  	
	
	\textbf{\Acknowledgements } We are grateful to Shaoyun Bai, Amanda Hirschi, Yuan-Pin Lee, Rahul Pandharipande, Oscar Randal-Williams, Dhruv Ranganathan and Mohan Swaminathan for helpful conversations and correspondence, to Mohan Swaminathan for pointing out an error in our original construction of consistent domain metrics, and to Julius Zhang for comments about Section 6.4.  We are grateful to the anonymous referee for their numerous comments, queries and suggestions which have helped  improve and clarify the exposition.
	\medskip
	
	M.A. was partially supported by MSRI / SLMath,  NSF award DMS-2103805, and the Simons Collaboration on Homological Mirror Symmetry.
	M.M. was partially supported by NSF award DMS-2203308. 
	I.S. was partially supported by MSRI / SLMath, by the Clay Foundation as a Clay Senior Scholar, and by a UKRI Frontier Research grant (in lieu of an ERC advanced grant).

	\section{Background and coefficients}
	
	We briefly recall some results on $K$-theories and on formal groups, simultaneously fixing conventions and notation.

	\subsection{Morava $K$-theories}
	
	Fix a prime $p$ and an integer $n>0$. The Morava $K$-theory $K_p(n)$ is a complex-oriented cohomology theory with coefficients
	\[
	\bK_* := K_p(n)_* = \bF_p[v,v^{-1}] \qquad |v| = 2(p^n-1).
	\]
	When $n=1$ it is a summand of mod $p$ complex $K$-theory, but for larger $n$ it is constructed algebraically. 	The $\{K_p(n)\}_{n\geq 0}$ are the  `primes' in the stable homotopy category; for a given $p$, the thick subcategories of the category of $p$-local finite spectra are exactly given by the kernels of the $K(n)_*$ (see \cite{Balmer, Hopkins-Smith} or \cite[Corollary 9.5]{Balmer-spectra}).  They are the only theories (besides the Eilenberg-Maclane theories $H\frak{k}$ for fields $\frak{k}$) which satisfy K\"unneth isomorphisms for arbitrary finite cell complexes
	\[
	H^*(X\times Y;\bK) = H^*(X;\bK) \otimes_{\bK_*} H^*(Y;\bK).
	\]
	Recall that for any spectrum $X$ we have homotopy orbits and homotopy fixed points
	\[
	X_{hG} = X \wedge_G EG_+ \qquad X^{hG} = F(EG_+, X)^G
	\]
	and a norm map $X_{hG} \to X^{hG}$.  A basic fact  \cite{Ravenel,Greenlees-Sadofsky} is that the map
	\begin{equation} \label{eqn:norm_map}
		X_{h\Gamma} \wedge \bK \to X^{h\Gamma} \wedge \bK
	\end{equation}
	induced by the collapse from $EG_+$ to $S^0$ is an equivalence when $\Gamma$ is finite and $\bK$ is $K_p(n)$-local.  One also has the Adams isomorphism \cite{LMSM} (see also \cite{Cheng})  
	\begin{equation} \label{eqn:Adams_isomorphism}
		Y/G \simeq (\Sigma^{-\frak{g}} \iota_*Y)^G
	\end{equation}
	for $G$-free spectra, with $\iota$ the inclusion of a $G$-fixed universe $\mathcal{U}^G \hookrightarrow \mathcal{U}$. Together with the freeness of the $G$-action on the Borel construction $EG_+ \wedge X$, this underlies a theorem of Cheng \cite{Cheng}:
	
	\begin{prop} Let a compact Lie group $G$ act on a closed smooth manifold $\scrT$ with finite stabilisers. Then there is a Poincar\'e duality isomorphism
		\[
		H_*^G(\scrT;\bK) \simeq \widetilde{H}^{-*}_G(\scrT^{-T\scrT\oplus\frak{g}};\bK).
		\] 
		The groups and the isomorphism depend on the orbifold $\bT := \scrT/G$ and not its presentation as a global quotient.
	\end{prop}
	
	\begin{proof} For $\bK = K_p(n)$ this is proved by Cheng in  \emph{op. cit.} The extension to $K_p(n)$-local theories is well known to experts, and a proof is provided in \cite{AMS-Hamiltonian}. \end{proof}
	
	Stably almost complex vector bundles are oriented for Morava $K$-theories, and so stable almost complex structures on the virtual bundle $T\scrT-\frak{g}$ will play a role in the sequel.  There is also a version of the previous proposition for non-compact manifolds, respectively manifolds with boundary, involving cohomology with compact supports, respectively cohomology of $\scrT / \partial\scrT$, see \cite{AMS-Hamiltonian}.

	\subsection{Equivariant $K$-theory}
	If $X$ is compact and Hausdorff, 
	$K(X)$ will denote the $\bZ/2$-graded group consisting of $K^0(X)$ and $K^1(X)$, with the even group obtained from complex stable vector bundles on $X$, and the odd group from those on $S^1 \times X$ with a trivialisation along the product of a point with $X$.  	If $X$ is locally compact with $X^+$ its $1$-point compactification, we define\footnote{This notation conflicts with that of \cite{Segal:equivtK}, but allows us to have uniform notation with the Borel-equivariant Morava theory later.} $K_c(X) := \tilde{K}(X^+)$ to be $K$-theory with compact support.
	
	Vector bundles with a lifting of the structure group from the orthogonal group to $Spin^c$ (in particular stably almost complex bundles) are $K$-oriented.   If $j: Z \subset X$ is a closed submanifold of a compact manifold $X$ of codimension $k$ and the normal bundle $\nu_Z$ is $K$-oriented, then we have a fundamental class  $[Z]=j_!(1_Z) \in K^k(X)$. If $X$ is a complex algebraic variety and $Z\subset X$ is a closed subvariety, then $[Z]$ is the image of $[\mathcal{O}_Z]$ under the natural map from algebraic $K$-theory to topological $K$-theory. 
	
	Let $G$ be a compact Lie group, and $X$ a compact $G$-CW complex. There is again a $\bZ/2$-graded theory $K_G(X)$ constructed from $G$-equivariant vector bundles. 
	If $W$ is a locally compact space, we set $K_{G,c}(W) := \tilde{K}_{G,c}(W^+)$ to be the $K$-theory with compact supports.  This is contravariantly functorial for \emph{proper} $G$-maps, and covariantly functorial for open embeddings of $G$-invariant open subsets.  The natural map
	\[
	\varinjlim K_{G,c}(U) \to  K_{G,c}(W) 
	\] is an isomorphism (c.f. \cite[Proposition 2.11]{Segal:equivtK}), where $U$ runs through relatively compact open $G$-subspaces of $W$ ordered by inclusion.  
	
	\begin{remark}
		As explained in Section \ref{Sec:notation} below, for a quasi-projective algebraic complex orbifold $Y$ we will sometimes write $K^0(Y)$ for the \emph{orbifold} $K$-theory, defined as the Grothendieck group of orbifold vector bundles on $Y$.   If $Y = \scrT/G$ is presented as a global quotient, this is canonically isomorphic to  $K^0_G(\scrT)$, which in particular does not depend on the presentation, cf. \cite{Pardon:enough}. Defining $K^1(Y) = \mathrm{cok}(K^0(Y) \to K^0(S^1\times Y))$ and $K(Y) = K^0(Y) \oplus K^1(Y)$ again yields a 2-periodic theory.
	\end{remark}

	\subsection{Transverse $K$-theory}\label{sec:transverse_Kthy}
	Suppose a compact Lie group $G$ acts on a manifold $M$; let $T_G^*M$ denote the conical subset of $T^*M$ of covectors which vanish on tangent vectors to $G$-orbits. Let $P: \Gamma(E) \to \Gamma(E)$ be a $G$-invariant pseudo-differential operator which is elliptic transverse to the $G$-orbits, meaning its symbol is invertible on $T_G^*M\backslash \{0\}$. Following Atiyah \cite{Atiyah}, we call such an operator \emph{transversely elliptic}. Such an operator is typically not Fredholm, and has infinite-dimensional kernel, but -- after restricting to a submanifold of $M$ with compact closure, which we shall do without further comment --  each $G$-representation occurs in the kernel with finite multiplicity. We shall only be interested in the multiplicity of the trivial representation, which can be interpreted as a map
	\begin{equation}
		\left(\mathrm{ind}_M\right)^G: K_G(T^*_GM) \to K_*.
	\end{equation}
	This has three properties we recall:
	\begin{enumerate}
		\item If $j: U \hookrightarrow M$ is a $G$-invariant open embedding into a compact manifold, there is a push-forward 
		\[
		j_!: K_G(T_G^*U) \to K_G(T_G^*M)
		\]
		and $\left(\mathrm{ind}_M\right)^G \circ j_!$ is intrinsic to $U$ (independent of choice of $M$); this extends the theory to the non-compact case.
		\item  If $Q \hookrightarrow M$ is a $G$-equivariant embedding of closed manifolds, then the bundle $T(\nu_{Q/M}) \to TQ$ is naturally almost complex. Clifford multiplication $\Lambda^{ev} T^{1,0}\nu \to \Lambda^{odd} T^{1,0}\nu$ defines a  symbol class over the total space of $T\nu$. Given any transversely elliptic symbol on $Q$, i.e. bundles $E, F\to Q$ and $\sigma \in \mathrm{Hom}(\pi_{TQ}^*E,\pi_{TQ}^*F)$ invertible transverse to orbits, we can pull back under $T\nu \to TQ$, multiply by the Clifford complex on $T\nu$, and then extend by zero from $\nu_{Q/M}$ to $M$. This then satisfies
		\[
		\xymatrix{
			K_G(T_G^*Q) \ar[rr]_{j_!} \ar[rd]_{\mathrm{ind}^G_Q} & & K_G(T_G^*M) \ar[ld]^{ \left(\mathrm{ind}_M\right)^G} \\
			& K_* &   
		}
		\]
		\item If $Y = M/G$ is a global quotient presentation of an orbifold, so $G$ acts with finite stabilisers, an orbifold elliptic differential operator $P$ on $Y$ lifts to a transversely elliptic operator $\tilde{P}$ on $M$, and Kawasaki's orbifold index \cite{Kawasaki} $\mathrm{ind}_Y$ applied to the symbol $\sigma(P)$ (in the orbifold $K$-theory of $Y$) is given by
		\[
		\mathrm{ind}_Y(\sigma(P)) = \left(\mathrm{ind}_M\right)^G (\tilde{P}).
		\]
	\end{enumerate}
	
        \begin{rem}
     
	As indicated by our notation, the index $ \left(\mathrm{ind}_M\right)^G $ can be obtained from a more elaborate structure which considers all $G$-representations. This is usually encoded as a distribution on $G$, using the fact that the multiplicities of the representations appearing in the index grow in a sufficiently controlled manner. More precisely, if $\xi_1,\ldots,\xi_r$ are an orthonormal basis of $\frak{g}$ and $Y_i$ denotes the first order differential operator on $E$ given by Lie derivative $\mathcal{L}_{\xi_i}$, then the Casimir
	\[
	\Delta_E = 1 - \sum_j Y_j^2
	\]
	has symbol injective in directions tangent to the orbits. For each $\lambda$, $P$ determines a Fredholm operator $P_{\lambda}$ with domain the $\lambda$-eigenspace of $\Delta_E$, and $\ker(P_\lambda)$ is finite-dimensional (and comprises smooth sections, so is independent of choices of Sobolev exponents in the construction). The formal sum $\sum_{\lambda} \mathrm{ind}(P_\lambda)$ converges to a well-defined distribution, yielding an index map
	\[
	\mathrm{ind}_M: K_G(T^*_GM) \to C^{-\infty}(G)^{\mathrm{Ad}(G)}
	\]
	The index $  \left(\mathrm{ind}_M\right)^G$ can then be recovered by pairing against the identity function on $G$, and the three properties listed above have natural formulations in this context.
        \end{rem}	
	
	\begin{rem}\label{rem:conical_subset}
		In \cite{Atiyah}, Atiyah does not assume that the action of $G$ is locally free. Under that assumption, which we always impose, the orbits of the $G$-action are submanifolds with tangent space $\frak{g}$, and a choice of Riemannian metric induces an isomorphism $ T^*_G \cT \oplus \frak g \cong T \cT$, and hence an isomorphism $K_G(T_G^*\cT) = K_G(\cT^{T\cT-\frak g})$.
	\end{rem}

	\subsection{Formal groups}
	
	Let $L:=L_\bE(u,v) \in \bE_*\llbracket u,v \rrbracket$ denote the formal group associated to a complex oriented cohomology theory $\bE$. We write 
	\[
	L(u,v) = u+_L v; \qquad [n]\cdot_L u = u +_L u +_L  \cdots +_L u
	\]
	and 
	\[
	L^{n_1,\ldots,n_m}(u_1,\ldots,u_m) = [n_1]\cdot_L u_1 +_L \cdots +_L [n_m]\cdot_L u_m
	\]
	and, simplifying notation, 
	\[
	L(u_1,\ldots,u_n) = u_1 +_L u_2 +_L \cdots +_L u_n.
	\]
	We recall the basic connection to line bundles:  $L(u,v) = u+v+\cdots$ gives the first Chern class,  in $\bE$-cohomology, of the tensor product of the line bundles with first Chern classes $u$ and $v$, as a power series in $u$ and $v$. Analogously, if  the line bundle $L_i$ has Chern class $u_i$, then  $L^{n_1,\ldots,n_m}(u_1,\ldots,u_m)$ gives the first Chern class for the line bundle $L_1^{\otimes n_1} \otimes \cdots \otimes L_m^{\otimes n_m}$ (universally, i.e. over a product of $m$ copies of $\bC\bP^{\infty}$).

	For $J = (j_1,\ldots,j_m) \in \{0,1\}^m$ there are uniquely defined power series $L^{n_1,\ldots,n_m}_J$ with
	\[
	L^{n_1,\ldots,n_m}(u_1,\ldots,u_m) = \sum_{J} u^J \cdot L_J^{n_1,\ldots,n_m}(u_1,\ldots,u_m).
	\]
	where $u^J = u_1^{j_1}\cdots u_m^{j_m}$. 
	For instance, 
	\begin{equation} \label{eqn:formal_group}
		L^{1,1}(u,v) = L(u,v) = \sum a_{ij} u^i v^j = u+v+ (u v) \cdot \sum_{i, j\geq 1} a_{ij} u^{i-1}v^{j-1}
	\end{equation}
	shows 
	\[
	L^{1,1}_{0,0} = 0;  \ L^{1,1}_{0,1} = 1;  \ L^{1,1}_{1,0} = 1; \ L^{1,1}_{1,1} = \sum_{i,j \geq 1} a_{ij} u^{i-1}v^{j-1}.
	\]
	
	\begin{Example} $H\bZ$ has the additive formal group $L(x,y) = x+y$.  Complex $K$-theory has the multiplicative formal group $L(x,y) = x+y-xy$. \end{Example}
	
	\begin{Example} Morava $K$-theory $K_p(n)$ has coefficients $K_p(n)_* = \bF_p[v,v^{-1}]$ with $|v| = 2(p^n-1)$. The formal group law  satisfies 
		\[
		x+_{K_p(n)} y \, = \, x+y - v \sum_{i=1}^{p-1} \frac{1}{p} \, {p \choose i} x^{i \, p^{n-1}} y^{(p-i)\, p^{n-1}} \, + O(x^{p^n}, y^{p^n}).
		\]
		and is essentially characterised by the height $n$ property $[p]\cdot_L x = v \,x^{p^n}$.  Since 
		\[
		x+_{K_p(n)} y \, = \, x+y \, \in K_p(n)_* \llbracket x,y\rrbracket / \langle p^{n-1}\mathrm{th} \, \mathrm{powers}\rangle
		\]
		for $n\gg 0$ this behaves somewhat like a finite field version of the additive formal group.
	\end{Example}
	
	We will sometimes write $L^{n_1,\ldots,n_m}_{\bE}(u_1,\ldots,u_m)$
	instead of $L^{n_1,\ldots,n_m}(u_1,\ldots,u_m)$ and analogously
	$L^{n_1,\ldots,n_m;J}_{\bE}(u_1,\ldots,u_m)$
	instead of $L^{n_1,\ldots,n_m}_J(u_1,\ldots,u_m)$
	if we wish to specify which complex oriented cohomology theory we are using.

	\subsection{Algebraic cobordism}
	
	Recall that algebraic cobordism $\Omega^*$ is a universal oriented Borel-Moore cohomology theory on smooth schemes over $\bC$, which has cycles for $\Omega^*(W)$ being tuples
	\begin{equation} \label{eqn:general_cycle}
		[Z,\iota; \scrL_1,\ldots, \scrL_r] 
	\end{equation}
	with $Z$ a smooth scheme, $\iota: Z \to W$ a morphism, and the $\scrL_i$ line bundles on $Z$. In particular, a line bundle $\scrL\to W$ defines a class $[W,\id; \scrL]$, and more generally an operation on cycles 
	\[
	\tilde{c}_1(\scrL):  [Z,\iota; \scrL_1,\ldots, \scrL_r]  \mapsto [Z,\iota; \scrL_1,\ldots, \scrL_r, \iota^*\scrL].
	\] Fundamental relations on cycles  include the `dimension axiom' which is the  vanishing condition 
	\begin{equation} \label{eqn:dimension_axiom}
		r > \dim_{\bC}(Z) \Rightarrow [Z,\iota; \scrL_1,\ldots, \scrL_r] = 0 
	\end{equation}
	(note that this holds without any assumption that the $\scrL_i$ be pairwise distinct), a relation equating 
	\[
	[Z,\iota; \scrL_1,\ldots,\scrL_r, \scrL] \quad \mathrm{and} \quad [D, \iota \circ \mathrm{incl}_D; (\scrL_1)|_D,\ldots, (\scrL_r)|_D]
	\]
	when $s \in H^0(Z;\scrL)$ vanishes transversely along $D$; and (imposed by hand) a formal group law relation asserting that
	\[
	\tilde{c}_1(\scrL\otimes \scrL')[Z,\iota; \scrL_1,\ldots, \scrL_r] = L_{\bL}(\tilde{c}_1(\scrL), \tilde{c}_1(\scrL'))  [Z,\iota; \scrL_1,\ldots, \scrL_r]
	\]
	with $L_{\bL}$ the  Lazard universal formal group law.  
	
	\begin{remark} \label{rmk:comparison} Let $W$ be a smooth quasi-projective variety over $\bC$. There is a comparison map $\eta: \Omega^*(W) \to H^*(W;MU)$ from algebraic to complex cobordism, which yields base-change maps to $H^*(W;\bE)$ for any complex-oriented $\bE$. Since $\Omega^*(W) \otimes_{\bL_*}\bZ = CH^*(W) $, algebraic cobordism  determines the Chow ring of $W$ so $\eta$ is usually not an isomorphism.
	\end{remark}
	
 	The model of equivariant algebraic cobordism we use is defined by  lifting to the stages of an algebraic approximation to the Borel construction, see \cite{Krishna} and \cite{Heller-MalagonLopez}. In particular, while it is equipped with a natural map to Borel equivariant complex cobordism, and hence to Borel equivariant complex oriented cohomology theories, it does not map to (genuine) equivariant $K$-theory. This requires us to give separate arguments, in various places, in order to establish the validity of statements about equivariant $K$-theory.

	If we have a polynomial $P(t_1,\cdots,t_r) = \sum_{i_1,\cdots,i_l} a_{i_1,\cdots,i_l}t^{i_1}\cdots t^{i_l}$ with coefficients in $\bE_*$ for some
	multiplicative complex oriented cohomology theory $\bE$, then
	we define
	\begin{equation}
		[Z,\iota;P(\scrL_1,\cdots,\scrL_r)] := \sum_{i_1,\cdots,i_l} a_{i_1,\cdots,i_l} [Z,\iota;\underbrace{\scrL_1,\cdots,\scrL_1}_{i_1},\underbrace{\scrL_2,\cdots,\scrL_2}_{i_2},\cdots,\underbrace{\scrL_r,\cdots,\scrL_r}_{i_r}]
	\end{equation}
	to be the corresponding element in $H^*(W;\bE)$.

	\subsection{Normal crossing divisors\label{Subsec:divisor}} 
	
	We will say that a normal crossing divisor is \emph{strict} if its irreducible components are smooth and meet pairwise transversely, with all iterated intersections of irreducible components being smooth and locally analytically modelled on $\{z_1\cdots z_r=0\} \subset \bC^n$.  If $\bE$ is any complex-oriented theory, and $Y$ a smooth variety, a divisor $D\subset Y$ defines a line bundle $\mathcal{O}(D)$ and hence a class $[\mathcal{O}_Y(D)] := \tilde{c}_1(\mathcal{O}(D)) \in H^*(Y;\bE)$. More generally, cycles of the form \eqref{eqn:general_cycle} define classes in $\bE$-cohomology.

	\begin{prop}\label{prop:formal group}
		Let $D = \cup_{i=1}^m n_i D_i$ be supported on the strict normal crossing divisor $\cup_i D_i$ on a smooth quasi-projective variety $Y$, and let $\scrL_i = \calO_Y(D_i)$. For $J \in \{0,1\}^m$ let $\iota_J: D_J \hookrightarrow Y$ be the inclusion of the stratum labelled by $J$ into $Y$, and let $\scrL_i^J$ denote $\iota_J^*\scrL_i$.  Then the following identity holds in $H^*(Y;\bE)$ 
		\begin{equation} \label{eqn:snc_identity}
			[\calO_Y(D)] = \sum_J  [D_J, \iota_J;L_\bE^{n_1,\ldots,n_m;J}(\scrL_1^J, \ldots, \scrL_m^J)].
		\end{equation}
		If the  group $GL(r,\bC)$ acts algebraically on $Y$ preserving the $\bE$-orientation and preserving $D$ stratum-wise, the identity \eqref{eqn:snc_identity} lifts to $H_{U(r)}^*(Y;\bE)$. 
	\end{prop}
	
	Note that although the terms in \eqref{eqn:snc_identity} involve power series in the Picard groups of the $D_J$, the `dimension axiom' for cycles \eqref{eqn:dimension_axiom} ensures that the operational first Chern class $\tilde{c}_1(\scrL)$ is always a locally nilpotent operator, so the terms in the sum are actually finite. Note also that \eqref{eqn:snc_identity} is indeed a graded identity, because the coefficients $a_{ij} \in \bE_*$ in the formal group law \eqref{eqn:formal_group} themselves have non-trivial degree.
	
	\begin{proof}  The corresponding identity for algebraic cobordism $\Omega_{alg}^*(Y)$ is proved as Proposition 3.1.9 in \cite{Levine-Morel}. As noted in Remark \ref{rmk:comparison}, for schemes $Y$ over a field $k$ of characteristic zero, an embedding $k\hookrightarrow \bC$ induces a comparison map from algebraic cobordism $\Omega^*(Y)$ to complex cobordism $H^*(Y;MU)$ (working integrally) by \cite{Levine-Morel,Levine-Pandharipande}.  (Note that we are working with schemes over $\bC$, but the characteristic of the coefficients of the cohomology theory are not constrained.) 
		The result for complex oriented $\bE$ then follows from that for $MU$ by base-change.
		
		The lift to equivariant algebraic cobordism  follows from \cite[Section 4.2]{Krishna}. Let $G = U(r)$. 
		Note $BG= \cup_N Gr(r, \bC^N)$ is an increasing union of algebraic $G_{\bC}$-varieties, and the same filtration defines a filtration $\{\widehat{Y}_N\}$ of $Y\times_G EG$, so $\widehat{Y}_N \to Gr(r, \bC^N)$ is the associated $Y$-bundle over the finite-dimensional Grassmannian. Since the $G_\bC$-action preserves $D$, it defines a divisor $D_N \subset \widehat{Y}_N$, to which we can apply the previous result (for fixed $N$). The natural map
		\[
		H^*_G(Y;\bE) = H^*(Y\times_G EG_+; \bE) \to  \lim_N H^*(\widehat{Y}_N;\bE)
		\]
		is an isomorphism.  The inclusion of $D_N \to \widehat{Y}_N$ induces a morphism of inverse limits which is compatible under increasing $N$.
		
		This establishes the result for Borel-equivariant complex oriented cohomology theories. The case of genuine equivariant $K$-theory follows from the Koszul resolution (cf. Equation \eqref{eq:Koszul-resolution} below).  More precisely, note that an analogue of the exact sequence of \eqref{eqn:inclusion_exclusion} holds whenever $D = \cup_{i=1}^k D_i \subset Y$ is any union of codimension one subschemes defined by a regular sequence (in particular they need not be smooth, reduced or meet transversely).  For varieties, this follows from the fact that in any commutative ring $R$ one has $V(I) \cup V(J) = V(I\cap J)$, where  $V(I) \subset \mathrm{Spec}(R)$ is the locus defined by the ideal $I$.  More generally, a union of not necessarily reduced divisors  defines a map $Y \to [\bA^k / \bG_m^k]$, or locally a map $ Y \dashrightarrow \bA^k$. Regularity of the defining sequence implies that this map is flat, which means there are no higher $\mathrm{Tor}$-terms when one pulls back the local resolution of \eqref{eq:Koszul-resolution}.  
	\end{proof}
	
	Even in the reduced case, note that in \eqref{eqn:snc_identity} the `coefficient' of a stratum $D_J \subset Y$ (more precisely, the coefficients of the various line bundles over that stratum which define the actual cycle) depends not just on the underlying scheme, or even its inclusion into $Y$, but on the particular way it is expressed as an intersection of irreducible components in $D$.
	
	\begin{Example} \label{ex:two_breakings} Suppose $D= D_1 \cup D_2 \subset W$ is reduced with two smooth irreducible components which meet cleanly along a smooth $Z = D_1 \cap D_2$.  Then, suppressing the inclusion maps $\iota_J$, we have: 
		\[
		[\mathcal{O}_W(D)] = [D_1]+[D_2] + \sum_{i,j \geq 1} a_{ij}\cdot (D_1)|_Z^{i-1}\cdot(D_2)|_Z^{j-1}
		\]
		with the 3rd term supported on $Z$ and $a_{ij}$ as in \eqref{eqn:formal_group}. This case already contains all of the information of the formal group law $L_{\bE}$; indeed algebraic cobordism can be reconstructed from relations associated to such `double point degenerations', compare to \cite{Levine-Pandharipande}. 
	\end{Example}
	
	If a divisor $D$ has normal crossings but not strict normal crossings, for instance components with non-trivial self-intersection, then there is no elementary analogue of Proposition \ref{prop:formal group}.  On the other hand, by Hironaka's theorem one can find a resolution on which the proper transform of $D$ has strict normal crossings;  \cite[Remark 2.1]{Payne} shows that one can achieve that iterated intersections of components are empty or smooth and irreducible.  Proposition \ref{prop:general_divisor} indicates one way of keeping track of the choices involved in such a procedure. 
	
	For the statement, we find it convenient to extend the notation of Equation \eqref{eqn:general_cycle} to generalised cohomology theory, so that given a pair $(Z,c)$ consisting of a subvariety $Z$ of $Y$ and a class $c$ in the $\bE$-cohomology of $Z$, we denote by $[Z;c]$ the pushforward of $c$ to the  $\bE$-cohomology of $Y$. Recall for instance from \cite{MPS} that a normal crossing divisor has a \emph{cone complex}, which is the dual intersection complex if it has strict normal crossings (or more generally if none of the irreducible components of $D$ self-intersect), but is in general more complicated and defined by a descent argument; see \emph{op. cit.} for the details of the construction. 
	
	\begin{prop} \label{prop:general_divisor}
		Let $D \to Y$ be a normal crossing divisor on a smooth quasi-projective variety $Y$ with smooth irreducible strata $D_J$. There is an identity in $H^*(Y;\bE)$ 
		\begin{equation} \label{eqn:nonstrict_nc_divisor}
			[\calO_Y(D)] = \sum_J [D_J;c_J]
		\end{equation}
		for coefficients $c_J \in \bE_*\llbracket \mathrm{Pic}(D_J)\rrbracket$ determined inductively in $J$ by $L_{\bE}$ and the cone complex for $D$.  If $D$ is $GL(r,\bC)$-equivariant for an algebraic action, then the equality lifts to $H^*_{U(r)}(Y;\bE)$.
	\end{prop}
	
	Given the inductive and implicit nature even of the definitions, the usefulness of this Proposition in practise is not clear: we include it to indicate the ingredients which enter any general formulation of a splitting axiom for the Gromov-Witten invariants defined later.  Because it is a digression, we give only a brief indication of the argument, the details behind which will not appear later in the paper.
	
	\begin{proof}[Sketch of proof]
		It suffices to prove the result in algebraic cobordism, since we can then pushforward to any complex oriented cohomology theory (in particular, the coefficients $c_J$ lie in the image of the maps $MU_*\llbracket \mathrm{Pic}(D_J)\rrbracket \to   \bE_*\llbracket \mathrm{Pic}(D_J)\rrbracket$). For a strict normal crossing divisor, the result holds -- with the strata $J$ labelled by $\{0,1\}^m$ when $D$ has $m$ irreducible components and the coefficients $c_J$ determined universally by the formal group law -- by expanding equation \eqref{eqn:snc_identity}.  The general case is reduced to this by `explosion'.  
		
		We follow the terminology of \cite{MPS}. Denote by $\Delta(D)$ the cone complex of $D$. We can form the blow up of $Y$ along the minimal stratum of $D$, then the strict transforms of the next-to-minimal strata, etc.  Iteratively this is the `explosion' of $D$, described in detail in \cite[Section 5.2, Proof of Theorem 11]{MPS}. They explain that explosion yields a variety $Y' \to Y$ on which the proper transform of $D$ has cone complex given by the barycentric subdivision of $\Delta(D)$. The barycentric subdivision has the feature that the associated divisor has no components which self-intersect, but the intersection of a set of branches may still be globally disconnected. Exploding twice yields a morphism $f: Y'' \to Y$ in which the pullback $f^*D$ has cone complex which is a cone over a simplicial complex, the combinatorial analogue of the divisor having  strict normal crossings. As a general feature of iterated blow-ups, the natural pullback map 
		\[
		\Omega^*(Y) \to \Omega^*(Y'')
		\]
		is a split injection, and the same holds $G$-equivariantly in the presence of an algebraic $G_{\bC}$-action by working on the stages of the Borel construction. The transform $f^*D$ has strata which are iterated projective bundles over bases obtained from strata of $D$ by blowing up along substrata. Using the projective bundle formula for algebraic cobordism and induction, one in principle obtains an implicit formula as in \eqref{eqn:nonstrict_nc_divisor}. Compare to \cite[Proposition 2.5.1]{Levine-Morel}. 	\end{proof}
	
	\begin{Example} \label{ex:HQ_easier}
		For $\bE = H\bQ$ one has the formal group $L(x,y) = x+y$, which gives that $[\calO_Y(D)] = \sum n_i [\calO_Y(D_i)]$.
	\end{Example}
	
	\begin{Example} \label{ex:divisor_Ktheory}
		In complex $K$-theory, with conventions for the Thom isomorphism compatible with \cite{Levine-Morel}, the Euler class $e^{KU}(V) = \sum (-1)^i \Lambda^i V^*$ and the first Chern class $c_1^{KU}(\scrL) = 1 - \scrL^*$. Then
		\[
		c_1^{KU}(\scrL \otimes \scrL') = c_1^{KU}(\scrL) + c_1^{KU}(\scrL') - c_1^{KU}(\scrL) \cdot c_1^{KU}(\scrL')
		\] 
		giving the multiplicative formal group in the normalisation 
		\[
		L(x,y) = x+y-xy.
		\]
		This leads to an `inclusion-exclusion' identity one can obtain more directly: the exact sequence of sheaves 
		\begin{equation}
			\label{eq:Koszul-resolution}
			0 \to \calO/(x_1\cdots x_k) \to \sum_i \calO/(x_i) \to \sum_{i<j} \calO/(x_i,x_j) \to \cdots \to \calO/(x_1,\ldots,x_k) \to 0  
		\end{equation}
		for $\{x_1\cdots x_k = 0\} \subset \bC^k$ shows that a reduced normal crossing divisor $D = \cup_i D_i  \subset M$ of a smooth quasi-projective variety gives an identity
		\begin{equation} \label{eqn:inclusion_exclusion}
			[\calO_D] - \sum_i [\calO_{D_i}] + \sum_{i<j} [\calO_{D_i\cap D_j}] + \cdots + (-1)^k [\calO_{D_1 \cap \ldots \cap D_k}]  = 0 \in K_{alg}(M)
		\end{equation}
		which one can push forward to $K_c(M)$.  	\end{Example}
	
    	Concretely, it is hard to apply Proposition \ref{prop:general_divisor} because it is inexplicit, but Example \ref{ex:divisor_Ktheory} and the final paragraph of the proof of Proposition \ref{prop:formal group} shows that in the special case of complex $K$-theory the fundamental class of a normal crossing divisor is given by the inclusion/exclusion formula, whether or not the normal crossing divisor is strict. This enables us to give slightly stronger results in this case.

	\section{Counting theories}
	\label{sec:counting-theories}
	
	We introduce a category of global (Kuranishi) charts, and `counting theories' on that category, to give a setting for generalised cohomological Gromov-Witten invariants which simultaneously encompasses Morava theories and complex $K$-theory. 
	
	\subsection{Global charts}
	
	\begin{defn}\label{defn:global_charts_category}
		The category $\Glo$ of \emph{stably almost complex smooth global charts} is defined as follows:
		\begin{itemize}
			\item an \textbf{object} is a quadruple $\bT \equiv (G, \cT, V, s)$, where
			(i) $G$ is a compact Lie group, (ii)  $\cT$ is a smooth $G$-manifold with finite stabilisers, (iii) $V$ is  a smooth $G$-vector bundle, and
			(iv) $s$ is a smooth $G$-equivariant section of $V$.    We fix in addition a $G$-equivariant stable almost complex structure on each of $V$ and on the virtual bundle $T \cT - \fg$.  
			
			\item a \textbf{morphism}  $f : \bT' \to \bT$ consists of (i) a map $\rho: G' \to G$ of Lie groups, (ii) a smooth $G'$ equivariant map (also denoted) $f: \cT' \to \cT$, and (iii) a $G'$-equivariant bundle map $V' \to f^* V$. We require that the pullback of $s$ agree with the pushforward of $s'$ under $f$, and that the data be compatible with the homotopy classes of the stable almost complex structures. 
		\end{itemize}
		A morphism is said to be \emph{proper} if the induced map ${s'}^{-1}(0) \to s^{-1}(0)$ on zero-loci is proper. We say that a chart is proper if the projection to the point is such.
	\end{defn}
	
	We recall that a $G$-equivariant stable almost complex structure on the virtual bundle $T\cT - \frak{g}$ comprises a choice of $k \in \bN$, a complex $G$-bundle $E$, a complex $G$-representation $W$, and a $G$-equivariant isomorphism $\psi: T\cT \oplus \bR^k \oplus W \stackrel{\sim}{\longrightarrow} \frak{g} \oplus E$, where we regard the data $(k,W,E,\psi)$ to be equivalent to $(k+2,W,E\oplus \bC, \psi \oplus \id_{\bC})$ as well as to $(k,W \oplus \bC,E\oplus \bC, \psi \oplus \id_{\bC})$.  A stable complex structure on $V$ is just a complex structure on $V \oplus \bR^k$ for some $k$, which we again insist is $G$-equivariant, taken up to the same equivalence relation of adding further trivial factors.  In the preceding definition, compatibility with the stable complex structures amounts to saying that, perhaps after further increasing $k$ and enlarging $E$, the $G$-bundle map $V' \oplus \bR^k \to f^*V \oplus \bR^k$ is equivariantly homotopic to a complex-linear bundle map, and the composite
	\[
	\xymatrix{
	\frak{g}' \oplus E' \ar[rr]_{(\psi')^{-1}} && T\cT' \oplus \bR^k \ar[rr]_{{Df\oplus \id}} && T\cT \oplus \bR^k \ar[rr]_{\psi} && \frak{g}\oplus E
	}
	\]	
	is equivariantly homotopic through bundle maps of the form $D\rho \oplus \chi$ to a map which is complex-linear on the second factor. 
	
	\begin{remark} We note that the section $s$ is assumed to be smooth and not merely continuous; its differential arises in the definition of an equivalence of charts below (cf. the maps in the exact sequence \eqref{eqn:exact_on_Z}). This is a slight variation on the set-up adopted in \cite{AMS-Hamiltonian}.  In the global charts of geometric origin arising later, smoothness of the section is achieved by  a mollification argument, cf. Lemma \ref{lem:mollify}. \end{remark}
	
	\begin{remark} We emphasise that the stable almost complex structure is on the virtual bundle $T\cT - \fg$ and not $T\cT$; it is the orbifold $\scrT/G$ which is naturally almost complex. Since we will often work $G$-equivariantly on $\cT$ itself, rather than passing to the orbifold quotient, this means that we naturally encounter `transverse almost complex structures', cf. Definition \ref{defn:transverse_complex}. \end{remark}
	
	We shall write $Z$ for the zero locus  $s^{-1}(0)$; we do not assume that $Z$ is compact, and this is the reason why it is essential to introduce the class of proper morphisms, since the fundamental constructuctions of Gromov-Witten theory naturally lead to global charts which are proper, or more generally which are proper over a parameter space. Those charts which are not proper arise in an auxiliary though unavoidable way given the approach we take, because the moduli space of \emph{pre-stable} Riemann surfaces is not proper, and we perform many construction parametrically over approximations of these moduli spaces. Non-proper charts also appear when one needs to compare the global approach to local Kuranishi descriptions of moduli spaces of curves, following \cite{Fukaya-Ono} (c.f. Appendix \ref{sec:comp-glob-kuran}).

        We will refer to $\cT$ as the \emph{thickening} of either $Z$ or of $Z/G$, and refer to $V \to \cT$ as the \emph{obstruction bundle}.   The category $\Glo$ is symmetric monoidal, with the product of global charts $\bT_1 =  (G_1, \cT_1, V_1, s_1)$ and   $\bT_2 =  (G_2, \cT_2, V_2, s_2)$ given by
	\begin{equation}
		\bT_1 \times \bT_2 \equiv (G_1 \times G_2, \cT_1 \times \cT_2, V_1 \oplus V_2 , s_1 \oplus s_2),   
	\end{equation}
	where $V_1 \oplus V_2$ denotes the Whitney sum of the pullbacks of $V_1$ and $V_2$ to the product, and $ s_1 \oplus s_2 $ the sum of the sections. The product of morphisms is given by the evident maps, as are the associativity and commutativity structure maps.
	
	We shall later use the fact that the subcategory of proper morphisms is closed under the monoidal structure, as well as the fact that every chart $\bT$ has a canonical diagonal morphism
	\begin{equation}
		\bT \to   \bT \times \bT,
	\end{equation}
	given by the diagonal at the level of groups, spaces, and vector bundles.  The diagonal map is always proper.

	\begin{defn}
		An \emph{equivalence} $f : \bT' \to \bT$  is a morphism with the property that (i) the induced map $(Z',G') \to (Z,G)$ of zero loci is an isomorphism of orbispaces (i.e. an isomorphism on quotients and on stabiliser groups) and (ii) the following sequence of vector bundles on $\cT'$ is exact near the $0$-locus:
		\begin{equation}\label{eqn:exact_on_Z}
			\fg' \to \fg \oplus T \cT' \to T\cT \oplus V' \to V
		\end{equation}
		(where the maps $T\cT' \to V'$ and $T\cT \to V$ are given by $Ds'$ respectively $Ds$).
	\end{defn}
	It is straightforward to define composition of morphisms and to check its associativity.  In practise, we will often have a particular Hausdorff topological orbispace $M$ in mind, and will `present' it via a global chart in the sense of equipping a global chart $(G,\cT,V,s)$ with a homeomorphism $Z/G \to M$ compatible with the data of the pointwise automorphism groups.
	
	We write 
	\begin{equation}
		\dim \bT = \dim \cT - \dim V  - \dim G
	\end{equation}
	for the `virtual dimension'  of $\bT$.  Exactness of \eqref{eqn:exact_on_Z} implies that this is preserved by equivalences.

	For later discussion, it is convenient to decompose equivalences into the following, which include the `germ equivalence', `stabilisation' and `group enlargement' moves of \cite{AMS-Hamiltonian}.
	\begin{lemma}\label{lem:factor_equivalence} 
		Every equivalence factors as a composition of the following maps or their inverses:
		\begin{enumerate}
			\item A free quotient $\bT \to \bT/H$, with $H$ a normal subgroup of $G$.  This replaces $(G,\cT,V,s)$ by $(G/H, \cT/H, V/H,s)$ where $V/H$ is viewed as a $G/H$-bundle on $\cT/H$.
			\item The induced map $\bT \to \bT \times_{G} G'$ associated to an embedding $G \to G'$. This replaces $(G,\cT,V,s)$ by $(G', \cT\times_G G', p^*V, p^*s)$ where $p: \cT\times_G G' \to \cT$ is the natural projection.
			\item The stabilisation map $\bT \to \bN$ associated to a $G$ vector bundle $N$ on $\cT$.  This replaces $(G,\cT,V,s)$ by $(G,\mathrm{Tot}(N), V \oplus p^*W, p^*s \oplus \id_N)$ where $p$ is the projection map $ N \to\cT$.
			\item An open embedding $\bT' \to \bT$ of a neighbourhood of $Z'$. Here we have a diffeomorphism $\phi: U' \to U$ from an open neighbourhood of $Z'$ in $\calT'$  to an open neighbourhood of $Z$ in $\calT$, for which $\phi^*(V|_U) \cong V'|_{U'}$ in a way which entwines the obstruction sections and so takes $Z'$ homeomorphically to $Z$. The move\footnote{When $\phi$ is the inclusion of an open neighbourhood $U \subset \calT$ of $Z$ into $\calT$, this is the `germ equivalence' move of \cite{AMS-Hamiltonian}.} replaces $(G,\cT,V,s)$ by $(G, U', V' |_{U'}, s')$. 
		\end{enumerate} 
	\end{lemma}
	
	\begin{proof} A morphism $\bT' \to \bT$ involves a homomorphism $\rho: G'\to G$ of groups. If the morphism is an equivalence, the resulting isomorphism of orbispaces shows that the kernel of $\rho$ acts freely; dividing this out by the first class of equivalence, one reduces to the case of an injection of groups. The second class of equivalence then reduces us to the case where $G' = G$ and the group homomorphism is the identity. At this point, using the fourth kind of equivalence to replacing $\cT$ by the total space of (an extension of) $V'$, and changing the $G$-equivariant map $\rho : \cT' \to \cT$ to $(\rho,s')$, the exact sequence \eqref{eqn:exact_on_Z} can be assumed to inject $T\cT' \hookrightarrow T\cT$ in a neighbourhood of $Z' \subset \cT'$. Such a map, up to the fourth kind of equivalence is a stabilization by the normal bundle.
	\end{proof}

	\begin{rem}
		It is possible to work with weaker assumptions. For example, one requires only a complex structure on the formal difference $T \cT - V - \frak g$, but we prefer to work with separate data to simplify the discussion. This does not result in any significant loss of generality, since stabilising a chart by an appropriate $G$-representation yields an equivalent chart in which the two separate complex structures are given (see \cite[Lemma 5.11]{AMS-Hamiltonian}).
	\end{rem}
	
	\begin{rem} The localisation of $\Glo$ at equivalences is a model for the category of `derived stably almost complex orbifolds', but we will not need to explicitly pass to the localisation in this paper. \end{rem}

	\begin{defn} Two charts $\bT_i = (G, \cT_i, E_i, s_i)$ are \emph{cobordant } if there is a $G$-equivariant cobordism of thickenings $\cT_i$ carrying a stably almost complex $G$-bundle $\scrE$ and section $s$ restricting to the given data over the ends. The cobordism is \emph{proper} if $s^{-1}(0)$ is compact.
	\end{defn}
	
	For the next definition, we use the complexification homomorphism $G \to G_{\bC}$ which is defined and injective for any compact Lie group.
	
	\begin{defn} \label{defn:transverse_complex}
		A \emph{transverse complex}  $G$-structure on a smooth $G$-manifold $X$ is a $G_{\bC}$-equivariant complex analytic structure on $X \times_G G_{\bC}$.  A \emph{transverse almost complex} $G$-structure on a smooth $G$-manifold $X$ is a $G_{\bC}$-equivariant almost complex structure on $X\times_G G_{\bC}$.
			\end{defn}
	
	In either case, any $G_{\bC}$-orbit carries its natural integrable complex structure.  Morphisms from $X$ to $Y$ in the category of transverse complex $G$-manifolds are $G$-maps for which the induced map $X\times_G G_{\bC} \to Y\times_G G_{\bC}$ is holomorphic; in the category of transverse almost complex $G$-manifolds they are $G$-maps for which the induced map is pseudo-holomorphic. 
	
	\begin{lemma}\label{lem:subcategory}
		The category of complex $G$-manifolds embeds into the category of transverse complex $G$-manifolds, and the category of transverse almost complex $G$-manifolds with finite stabilisers embeds into $\Glo$, compatibly with the subcategory of proper maps.
	\end{lemma}
	
	\begin{proof} 
		If $X$ is a complex $G$-manifold we equip $X\times G$ with the diagonal action, where $G$ acts on itself by conjugation (on the submanifold $X \times \{e\}$ this loses no information about stabiliser groups). $X\times G$ then has a natural transverse complex structure, and holomorphic maps of complex manifolds then define transverse complex maps of the associated transverse complex manifolds. For the second statement, a transverse almost complex $G$-manifold $Z$ defines a global chart $(G,Z,0,0)$.
	\end{proof}
	
	\begin{remark} In Lemma \ref{lem:subcategory} one could alternatively replace $X$ by $X\times \frak{g}$, where $\frak{g}$ carries the adjoint action. Under the hypothesis that the $G$-action on $X$ has  finite stabilisers this is also the vector bundle over $X$ given by the tangent spaces to the orbits. The resulting transverse complex $G$-structures define equivalent global charts by passing to a neighbourhood of $e\in G$ and appealing to germ equivalence. 
	\end{remark}
	
	In line with Lemma \ref{lem:subcategory}, we will freely talk about $G$-equivariant `divisors' on transverse complex (or transverse almost complex) manifolds and on the associated global charts; such a divisor has an associated $G$-equivariant complex line bundle.

	\subsection{Axioms of counting theories}
	
	We use the term \emph{graded group} to refer to either $\bZ/2$- or $\bZ$-graded groups.  By a symmetric monoidal functor we always mean a \emph{lax} symmetric monoidal functor. 
	
		\begin{defn} \label{def:counting_theory}
		A \emph{counting theory} $\bE$ on global charts consists of the following data:
		\begin{enumerate}
			\item a  symmetric monoidal contravariant functor $E^*$ (cohomology) from $\Glo$ to the category of graded abelian groups;
			\item a symmetric monoidal covariant functor $E_*$ (homology) from $\Glo$ to graded abelian groups, together with a module structure over $E^*$ which we denote by $\cap$;
			\item on the subcategory of proper maps, a symmetric monoidal covariant functor $E^{lf}_*$ (locally finite homology, i.e. Borel-Moore homology) to graded abelian groups, together with a module structure over $E^*$, and a natural isomorphism $ E_*(\bT) \cong E^{lf}_*(\bT)$ for proper charts;
			\item a class $ [\bT] \in E^{lf}_{\dim \bT}(\bT)$ (the virtual fundamental class) for each object;
			\item an isomorphism of symmetric monoidal functors $E^*(\bT) \cong E^{lf}_{\dim \bT - *}(\bT)$ on the subcategory of proper maps of transverse almost complex $G$-manifolds, which maps the virtual fundamental class to the unit; and
			\item  a (commutative rank $1$) formal group law $L = L_{\bE}$ with coefficients in $E^*(pt)$.
		\end{enumerate}
		These data are supposed to satisfy the following axioms:
		\begin{enumerate}
			\item (Functoriality) The morphisms of homology, cohomology, and locally finite homology groups associated to an equivalence are isomorphisms, and preserve fundamental classes;
			\item (Products) the virtual class $[\bT_1 \times \bT_2]$ is the image of $[\bT_1] \otimes [\bT_2]$ under the monoidal structure;
			\item \label{pullback axiom} (Pullback) For each equivariant transverse pullback diagram, whose base consists of (tranverse almost complex) manifolds,
			\begin{equation} \label{eq:transverse_diagram}
				\begin{tikzcd}
					\bT_{0} \ar[r,"f"] \ar[d,"\pi_0"] & \bT_{1} \ar[d,"\pi_1"] \\
					F_{0} \ar[r,"g"] & F_{1},
				\end{tikzcd}
			\end{equation}
			 the maps of groups are isomorphisms and the map $g$ is proper,   we have
			\begin{equation}
				f_* [\bT_{0}] =  [\bT_{1}] \cap \pi_1^*( g_* [F_{0}])
			\end{equation}
			where we use the identification $E^{lf}_*(F_{1}) \cong E^{\dim - *}(F_{1})$ to consider the image of $[F_0]$ as a cohomology class on $F_1$; 
			\item (Chern class) For a transverse almost complex $G$-manifold $M$, there is a homomorphism
			\[
			\Pic_G(M) \to E^*(M); \quad \scrL \mapsto \tilde{c}(\scrL)
			\]
			satisfying 
			\[
			\tilde{c}(\scrL_1 \otimes \scrL_2) =  L(\tilde{c}(\scrL_1), \tilde{c}(\scrL_2))
			\] and for which, if $\iota: W \subset M$ is the zero-set of an equivariant section of $\scrL$ which vanishes transversely, then $\iota_*[W] = \tilde{c}(\scrL)$ under the isomorphism $E_*^{lf}(M) \to E^{\dim(M)-*}(M)$;
			\item \label{normal crossing axiom} (Normal crossing) Given a diagram of  $G$ manifolds with $G$-transverse complex analytic structures
			\begin{equation} \label{eq:normal_crossings_diagram}
				\begin{tikzcd}
					& F \ar[d,"p"] \\
					N \ar[r,"g"] & M
				\end{tikzcd}
			\end{equation}
			in which $N$ is a smooth divisor and $p^{-1}(N)$ is a strict normal crossings divisor, with top strata $F_1, \ldots, F_k$ that are $G$-invariant, we have
			\begin{equation}
				p^* (g_* [N]) = L([F_1], \ldots, [F_k]) \in E^* F.
			\end{equation}
			\end{enumerate}
	\end{defn}
	
	We clarify the statement that the functor $E_*$ (and similarly for $E^{lf}_*$) is a module over $E^*$: first, note that the diagonal map lifts $E^*$ to a functor valued in graded rings.  We then require that $E_*(\bT)$ be equipped with a module structure over $E^*(\bT)$, that the morphism $E_*(\bT) \to E_*(\bT')$ be an $E^*(\bT')$-module map, and finally that the morphism
	\begin{equation}
		E_*(\bT_1) \otimes E_*(\bT_2) \to   E_*(\bT_1 \times \bT_2) 
	\end{equation}
	be a morphism of $E^*(\bT_1) \otimes E^*(\bT_2)$-modules. The first of these can be expressed in formulae as follows: given a map $f : \bT' \to \bT $, and classes $\alpha \in E^* \bT$ and $\beta \in E_* \bT'$, we have
	\begin{equation}\label{eqn:module_map_property}
		f_* \left( f^* \alpha \cap \beta \right) = \alpha \cap f_* \beta.          
	\end{equation}
	A similar formula corresponds to the fact that $E^{lf}_*$ is a module over the restriction of $E^*$ to proper maps. The $G$-invariance of the top strata $F_i$ in the final axiom (i.e. the fact that they are not non-trivially permuted) is automatic if $G$ is connected.

	\begin{lemma} \label{lem:cobordism} Suppose a proper cobordism $\bW$ of charts $\bT$ and $\bT'$ is induced by taking regular values of a projection $\bW \to B$ for some finite-dimensional $G$-manifold $B$. Then  for any counting theory, the push-forwards of $[\bT]$ and of $[\bT']$ to $E_*^{lf}(\bW)$ agree.
	\end{lemma}
	
	\begin{proof}
		This follows from the pullback axiom. 
	\end{proof}

	\begin{remark} \label{rem:axiom_3_variation} It would be more natural to formulate a version of Axiom \ref{normal crossing axiom} that applies when $p^{-1}(N)$ is a non-strict normal crossing divisor, and which appeals in concrete cases to Proposition \ref{prop:general_divisor}. The resulting expression for $p^*g_*[N]$ is unfortunately not explicit.
	\end{remark}
	
	\begin{remark} In the examples considered below, the `Chern class' and `normal crossing' axioms are closely related; essentially the former implies the latter, by reducing\footnote{This can be made precise for $K_p(n)$-local theories, where we work Borel equivariantly as in algebraic cobordism,  but is not strictly true for  complex $K$-theory where we work with equivariant bundles.} to Proposition \ref{prop:formal group}. For clarity we have kept both axioms, rather than trying to minimise the hypotheses. One might achieve more formal clarity by modelling the axioms after those for an oriented bivariant cohomology theory in the sense of \cite{Fulton-MacPherson-bivariant}.
	\end{remark}
	
	\begin{rem}
		By imposing the condition that the (co)-homology theories entering in the definition of a counting theory be symmetric monoidal, we have excluded examples arising from cohomology theories with an associative  but non-commutative multiplication. One natural example is given by the Morava $K$-theories $K(n)$ at the prime $2$. We expect that, just as one can  understand the failure of these theories to define a commutative product in ordinary cohomology by carefully keeping track of the side along which one performs each product \cite{Boardman} and by working in an appropriate category of bimodules over the coefficients of $K(n)$, it should be possible to extend our results to this context.
	\end{rem}
	\subsection{Borel equivariant $K(n)$-local theories }
	
	\label{sec:borel-equiv-cohom}
	We write $ \cT|Z$ for the pair $(\cT, \cT \setminus Z)$. Given a $K(n)$-local complex oriented (commutative) multiplicative cohomology theory $\bK$, define
	\begin{align}
		H^*(\bT; \bK) & \equiv   H^{*}_{G}(Z; \bK) =  H^{*}(Z \times_G EG; \bK)\\
		H^{lf}_*(\bT; \bK) &  \equiv H^{\dim \cT/G -*}_{G}(\cT|Z; \bK) = H^{\dim \cT/G -*}(\cT|Z \wedge_{G} EG_+; \bK)\\
		H_*(\bT; \bK) &  \equiv H^{\dim \cT/G -*}_{G,c}(\cT|Z; \bK) = H^{\dim \cT/G -*}(\cT^+|Z \wedge_{G} EG_+; \bK)
	\end{align}
	to be the Borel equivariant  co-homology groups (relative or with compact support), where $\dim \cT/G = \dim \cT - \dim G$ (and $\cT^+$ denotes the $1$-point compactification). Furthermore, we define the virtual class
	\[
	[\bT] := e_G(V) \in H^{lf}_*(\bT;\bK)
	\]
	to be the $G$-equivariant Euler class of the obstruction space $V$.
	\begin{rem}
		We expect the homology group to be isomorphic to the corresponding (equivariant) generalised Steenrod-Sitnikov homology groups of the zero locus $Z$. We are not aware of a discussion of this equivariant statement in the literature.
	\end{rem}

	\begin{lemma}\label{lem:properties_in_local_case}
		The cohomology ring $ H^*(\bT; \bK)$ is a contravariantly functorial symmetric monoidal functor, the homology ring $ H_*(\bT; \bK)$ is a symmetric monoidal covariant functor, which is a module over cohomology, as is the locally finite homology group  $ H^{lf}_*(\bT; \bK) $, for proper maps. The maps associated to equivalences are isomorphisms which preserve the fundamental class in $H^{lf}_*(\bT)$.
	\end{lemma}
	\begin{proof}
		The contravariant functoriality of cohomology is straightforward, and the fact that it is symmetric monoidal follows from the assumption that $\bK$ is multiplicative, and the homotopy equivalence
		\begin{equation}\label{eq:Borel_construction_monoidal}
			\left(Z_1 \times_{G_1} EG_1\right) \times \left(Z_2 \times_{G_2} EG_2\right) \cong \left(Z_1 \times Z_2 \right) \times_{G_1 \times G_2} E(G_1 \times G_2).
		\end{equation}

		The action of cohomology on locally finite homology follows from the continuity property of cohomology and the excision property for homology: every class in $H^{*}_{G}(Z; \bK)$ lifts to some neighbourhood $\nu Z$ of $Z$ in $\cT$, and the desired map is induced by the (relative) cup product
		\begin{equation}
			H^{*}_{G}(\nu Z; \bK) \otimes   H^{*}_{G}(\nu Z|Z; \bK) \to H^{*}_{G}(\nu Z|Z; \bK)
		\end{equation}
		and the isomorphism $ H^{*}_{G}(\nu Z|Z; \bK) \cong H^{*}_{G}(\cT|Z; \bK)$. Applying the same argument for compactly supported cohomology yields the action of cohomology on homology.
		
		In order to prove the covariant functoriality of locally finite homology, it will be convenient to first use the Thom isomorphism
		\begin{equation}
			H^{\dim \cT/G -*}_{G}(\cT|Z; \bK) \cong  H^{-*}_{G}(\cT|Z^{\frak{g}-T \cT  }; \bK)
		\end{equation}
		where $\cT^{\frak{g}-T \cT}$  denotes the Thom spectrum of the (virtual) difference between the tangent space $T \cT$ and the Lie algebra of $G$. Since we work Borel equivariantly, by definition
		\[
		H^{-*}_{G}(\cT|Z^{\frak{g}-T \cT  }; \bK) = \pi_*(F(EG_+ \wedge \cT|Z^{\frak{g}-T\cT}, \bK)^G)
		\]
		We use adjunction to write
		\[
		F(EG_+ \wedge \cT|Z^{\frak{g}-T\cT}, \bK)^G = F(EG_+,\Sigma^{-\frak{g}}F(\cT|Z^{-T\cT},\bK))^G.
		\]
		We use the norm map $EG_+ \wedge X \to F(EG_+,X)$ with $X = \Sigma^{-\frak{g}}F(\cT|Z^{-T\cT},\bK)$, which induces an isomorphism on derived fixed-points when $\bK$ is $K(n)$-local cf. \eqref{eqn:norm_map}, and then the Adams isomorphism $Y/G = (\Sigma^{-\frak{g}}\iota_*Y)^G$ cf. \eqref{eqn:Adams_isomorphism} noting that $Y = EG_+\wedge X$ is $G$-free, to obtain
		\begin{equation} \label{eq:lfH-isomorphism-pullout-G}
			H^{lf}_*(\bT; \bK) \cong \pi_*\left( EG \wedge_G F(\cT|Z^{-T \cT}; \bK) \right).
		\end{equation}
		We have a similar description of homology as
		\begin{equation} \label{eq:H-isomorphism-pullout-G}
			H_*(\bT; \bK) \cong \pi_*\left( EG \wedge_G F(\cT^+|Z^{-T \cT}; \bK) \right),
		\end{equation}
		where we can model $ F(\cT^+|Z^{-T \cT}; \bK)$ as the spectrum of sections of the tangent spherical fibration of $\cT$, which are trivial outside a compact set and on the complement of $Z$.

		Having reformulated the homology theories, we are ready to prove their functoriality, so we consider a map $\bT \to \bT'$. This includes the datum of a $G$-map $\cT \to \cT'$, with the property that $Z$ maps to $Z'$. By choosing a (complex) $G$-representation $W$,  and a $G$-embedding $\cT \to W \times \cT'$ whose composition with the projection to $\cT'$ lifts the map $\cT \to \cT'$, Thom collapse defines a $G$-map
		\begin{equation} \label{eq:Thom-collapse}
			{\cT'|Z'}^{W} \to  \cT|Z^{W + T \cT' - T \cT },
		\end{equation}
		whenever $Z \to Z'$ is proper, which dually induces a map
		\begin{equation} \label{eq:Thom-collapse-destabilised}
			F(\cT|Z^{-T \cT}; \bK) \to F(\cT'|Z'^{-T \cT'}; \bK).  
		\end{equation}
		Finally, the homomorphism $G \to G'$ yields a map
		\begin{equation}
			EG \wedge_G F(\cT|Z^{-T \cT}; \bK) \to  EG' \wedge_{G'} F(\cT'|Z'^{-T \cT'}; \bK)
		\end{equation}
		which on homotopy groups defines the desired map
		\begin{equation}\label{eq:push}
			H^{lf}_*(\bT; \bK) \to  H^{lf}_*(\bT'; \bK).
		\end{equation}
		Given two different choices of $G$-representation $W$ and $W'$, they can both be embedded in $W \oplus W'$, and stabilising the representation in that way does not change the map \eqref{eq:Thom-collapse-destabilised}, by the Thom isomorphism theorem for $W'$ which relates the maps of \eqref{eq:Thom-collapse} for $W$ and $W\oplus W'$; it follows that \eqref{eq:push} does not depend on $W$. 
		
		For a non-proper map $\bT \to \bT'$, the Thom collapse map in Equation \eqref{eq:Thom-collapse} takes value in the $1$-point compactification of the Thom space $ \cT|Z^{W + T \cT' - T \cT }$, and factors through the $1$-point compactification of ${\cT'|Z'}^{W}$. It thus induces a map
		\begin{equation}
			H_*(\bT; \bK) \to  H_*(\bT'; \bK).   
		\end{equation}
		Suppose we have open neighbourhoods $Z \subset \nu Z  \subset \cT$ and $Z' \subset \nu Z' \subset \cT'$ for which the map $\cT \to \cT'$ sends $\nu Z$ into $\nu Z'$ (and $Z$ into $Z'$). All the steps above are compatible with the associated excision isomorphisms, which shows that the push-forward is a module map, i.e. that Equation \eqref{eqn:module_map_property} holds.  Invariance of the fundamental class under equivalences follows from Lemma \ref{lem:factor_equivalence}; for all but the first the Euler class is obviously invariant, whilst for the first class of equivalences this uses that if $\Gamma$ acts freely on $M$ then the $\Gamma$-equivariant Euler class of a $\Gamma$-bundle $W \to M$ is the pullback of the ordinary Euler class of $W/\Gamma \to M/\Gamma$ under $H_\Gamma^*(M;\bK) \cong H^*(M/\Gamma;\bK)$.

		Finally, the symmetric monoidal structure on both homology and locally finite homology follows from the multiplicativity of the Borel construction formulated in Equation \eqref{eq:Borel_construction_monoidal},  the fact that the module action is compatible with the external product maps from the external multiplicativity of the Adams isomorphism, and the fact that it is a module map (c.f. \cite[Appendix C]{AbouzaidBlumberg2021}).
	\end{proof}

	\begin{lemma}
		A multiplicative complex oriented $K(n)$-local cohomology theory defines a counting theory in the sense of Definition \ref{def:counting_theory}.
	\end{lemma}
	\begin{proof}
          There are three remaining properties to prove. To prove the formula for the virtual fundamental class of the pullback in Diagram \ref{eq:transverse_diagram}, it suffices (by stabilising with $G$-representations) to consider the case when $F_0 \to F_1$ is an embedding and $\cT_1 \to F_1$ is a submersion. Indeed, stabilising $F_1$ (and $\bT_1$) by a sufficiently large representation, we may ensure the first property, without changing $F_0$ and $\bT_0$, and replacing $\bT_1$ by its stabilisation with respect to the tangent bundle of $F_1$ also ensures that the second property holds, with $\bT_0$ correspondingly replaced by the same stabilisation. The reduction to this case then follows by the invariance of the fundamental class under stabilisation, and the functoriality of pullback and pushforward.

         In the case under consideration, the obstruction bundle of $\bT_0$ is $f^*V$ where $V\to \cT_1$ is the obstruction bundle, so we must compute
		\[
		f_*(e_G(f^*V)) = f_*(f^*e_G(V) \cap 1_{\cT_0})
		\] where $1_{\cT_0}$ denotes the unit. Since $\pi_1$ is a submersion and the diagram is a pullback,  the Thom class of $\cT_0\subset \cT_1$ is the pullback of $[F_0]$, and the result now  follows from \eqref{eqn:module_map_property}.
		
		Since  by hypothesis $\bK$ is multiplicative and complex-oriented, there are Chern classes for complex line bundles and a formal group law $L$ associated to the complex orientation determined by the equality
		\begin{equation} \label{eq:formal_group_chern_classes}
			c_1(\scrL_1 \otimes \cdots \otimes \scrL_k) = L(c_1(\scrL_1), \ldots, c_1(\scrL_k)). 
		\end{equation}
		Considering the context of Diagram \ref{eq:normal_crossings_diagram}, the image $g_* [N]$ of the fundamental class of $N$ agrees with the first Chern class of the corresponding complex line bundle over the Borel construction. The fact that the inverse image of $N$ is a normal crossings divisor gives an isomorphism between the pullback of this line bundle and the tensor product of the $G$-equivariant complex line bundles $\scrL_{F_i}$ associated to its top strata. Identifying the virtual fundamental class of $F_i$ with the first Chern class of $\scrL_{F_i} $, the result now follows from Equation \eqref{eq:formal_group_chern_classes}.
	\end{proof}
	
	\subsection{Equivariant $K$-theory}
	\label{sec:equiv-cohom}
	Next, we consider equivariant $K$-theory, and define\footnote{Recall that $\cT/G$ is a (stably) complex orbifold, so the shifts could be supressed since we work $2$-periodically.}
	\begin{align}
		K^*(\bT) & \equiv   K^{*}_{G}(Z) \\
		K^{lf}_*(\bT) &  \equiv  K^{\dim \cT/G -*}_{G}(\cT|Z) \\
		K_*(\bT) & \equiv  K^{\dim \cT/G -*}_{G,c}(\cT|Z). 
	\end{align}
	Note the formal similarity with the definition in Section \ref{sec:borel-equiv-cohom}; the subscript $G$ does not refer anymore to Borel-equivariant cohomology, but instead to the fact that we are considering $G$-equivariant vector bundles. The stable complex structure of $V$, and the choice of section which does not vanish away from $Z$, yields an equivariant Euler class in $K^{\dim \cT/G -*}_{G}(\cT|Z) $ which we define to be the virtual fundamental class of $\bT$.
	
	\begin{rem} \label{rem:reminder}
	We give a brief reminder on Thom isomorphisms in $K$-theory. 
	Recall that given a bounded complex $F^{\bullet}$ of vector bundles on $X$, we get an element of $K$-theory via $F^\bullet \mapsto \sum_i (-1)^i [F_i]$. If the complex is exact on $A\subset X$,  it naturally defines an element in $K^0(X,A)$. Now let $p: E \to X$ denote projection of a $G$-bundle $E$ over $X$. We define a homomorphism $\phi: K_{G,c}(X) \to K_{G,c}(E)$ as follows. Fix a complex $\{F^{\bullet}\}$ exact away from a compact subset of $X$ and representing an element of $K_{G,c}(X)$; we map this to the complex $\phi: F^\bullet \mapsto \Lambda^*(E) \otimes p^*F^{\bullet}$, where $\Lambda^*(E)$ is the Koszul complex over $E$ associated to $p^*E$ and its tautological diagonal section (which vanishes on the zero-section of $E$). Note that $\Lambda^*(E) \otimes p^*F^{\bullet}$ has compact support on $E$, i.e. is acylic outside a compact set, so defines an element of $K_{G,c}(E)$. (By contrast, the Koszul complex $\Lambda^*(E)$ itself has support the zero-section $X\subset E$, so this does not define an element of $K_{G,c}(E)$ unless $X$ itself is compact.) The Thom isomorphism asserts that the map $ \phi: K_{G,c}(X) \to K_{G,c}(E)$ is an isomorphism; there is an analogous  Thom isomorphism  without the compact support condition, i.e. an isomorphism $K_G(X) \to K_G(E)$ for any locally compact $X$, cf. \cite[Proposition 3.2]{Segal:equivtK}. The Thom class is represented by the Koszul complex of $E$, so the Euler class lifts to a class relative $X\backslash Z$ when $E$ admits a section non-vanishing away from $Z$. 
	\end{rem}

	\begin{rem}
	If $X$ is complex algebraic, any coherent sheaf $\scrF \to X$ defines an element of $K(X)$, which can be computed by taking a bounded projective resolution.  If $X$ is a complex algebraic variety and $Z\subset X$ is a closed subvariety, then the $K$-theoretic fundamental class $[Z]$ is the image of $[\mathcal{O}_Z]$ under the natural map from $K_{alg}(X) \to K(X)$. Indeed, to check this it suffices to work locally. If $Z = V(I)$ is cut out in $\bC^n$ by an ideal $I \leq \bC[x_1,\ldots,x_n] = R$ and $P^{\bullet}$ is a bounded projective resolution of $R/I$, then $P^{\bullet}$ is exact at any $Q\not
\in Z$. Therefore $[P^{\bullet}]$ defines a class in $K(\bC^n,\bC^n\backslash Z)$. One uses deformation to the normal cone to show this class satisfies the conditions of the Thom class, see e.g. \cite[Theorem 18.8]{Dugger}.
\end{rem}

	The ring structure on cohomology, and its action on the homology theories is defined in the same way as for the Borel equivariant theories using the continuity of cohomology. The functoriality of cohomology is easy to establish. For  homology (or locally finite homology, analogously), we proceed as follows: we first identify
	\begin{equation}
		K^{\dim \cT/G -*}_{G,c}(\cT|Z) \cong  K_{G,c}(\cT|Z^{ T \cT - \frak{g}}). 
	\end{equation}
	The work of Atiyah and Singer \cite{Atiyah}, cf. the discussion after \cite[Theorem 2.6]{Atiyah} and remembering Remark \ref{rem:conical_subset},  identifies every element of the right hand side with the symbol of a \emph{$G$-equivariant transversally elliptic} operator on $\cT$, relative those on the complement of $Z$. This point of view will be essential in the following discussion, because the notion of \emph{$G$-equivariant elliptic operator}, which is usually used to construct pushforwards in equivariant $K$-theory, is not relevant to our context, since it corresponds to $ K_{G,c}(\cT|Z^{ T \cT}) $.

	Continuing with our discussion of the functoriality of homology, consider the embedding $G \to G \times G'$ given by the graph of the homomorphism $G \to G'$. We choose a complex $G'$ representation $W'$, and a lift of the graph of the map $\cT \to \cT'$ to a  $G \times G'$ equivariant embedding
	\begin{equation}
		\cT \times_{G} \left(G \times G'\right) \to \cT \times  W' \times \cT',
	\end{equation}
	where $\cT \times W' \times \cT'$ is equipped with the product action of $G \times G'$.

	We have natural maps
	\begin{equation}
		\begin{tikzcd}
			K_{G,c}(\cT|Z^{ T \cT - \frak{g}}) \cong   K_{G \times G',c}(\left(\cT \times_{G} \left(G \times G'\right)|Z  \times_{G} \left(G \times G'\right) \right)^{ T \cT  \times_{G} \left(G \times G'\right) - \frak{g} - \frak{g'}} )\ar[d]  \\  K_{G \times G',c}(\left(\cT \times W' \times \cT' | Z  \times_{G} \left(G \times G'\right) \right)^{T \cT - \frak{g} + T   W' \times \cT' - \frak{g'}}) \ar[d] \\
			K_{G \times G',c}(\left(\cT \times W' \times \cT' |Z \times W' \times  Z' \right)^{T \cT - \frak{g} + T   W' \times \cT' - \frak{g'}}) \ar[d] \\
			K_{G \times G',c}(\left(\cT  \times \cT' |Z  \times  Z' \right)^{T \cT - \frak{g} + T \cT' - \frak{g'}})
		\end{tikzcd}
	\end{equation}
	where the isomorphism in the top row is given by the Thom homomorphism of \cite[Theorem 4.1]{Atiyah}, the first vertical arrow by \cite[Theorem 4.3]{Atiyah}, the second by the fact that the inclusion of the image of $Z$ in $Z'$ implies that complexes which are acyclic in the complement of $  Z  \times_{G} \left(G \times G'\right)$ are necessarily acyclic in the complement of $W' \times Z'$, and the last by the (inverse) Thom isomorphism associated to the inclusion of the origin in $W'$ (this uses that $W'$ was chosen to be complex). As in Lemma \ref{lem:properties_in_local_case}, the dependence on the auxiliary $W'$ in this construction is removed by a suitable application of the Thom isomorphism, cf. Lemma \ref{lem:morphisms_in_Kthy} below. The corresponding functoriality of  locally finite (rather than ordinary) homology follows exactly the same steps, using the Thom isomorphism without compact supports from Remark \ref{rem:reminder}.
	
	\begin{rem}
		If $G$ acts on $\cT$ and $G'$ acts on $\cT'$, both with finite stabilisers, the map $\cT \times_G (G\times G') \to \cT\times\cT'$ arising from the graph of a $G$-equivariant map $\cT \to \cT'$ can still fail to be an embedding but be only finite-to-one; this is why we need to further introduce the representation $W'$ above.
	\end{rem}
	\begin{rem}
		Note that, while the action of $G'$ on a representation $W'$ cannot be locally free if $G'$ is not finite or $W'$ is not trivial, the product $W' \times \cT'$ has a locally free action, so that the difference $ T  W' \times \cT' - \frak{g'}$ still gives the tangent bundle to $T_{G'} W' \times \cT' $.
	\end{rem}
	Next, we consider the projection $\cT  \times \cT' \to \cT'$, with fibre the $G$-manifold $\cT$. At each point in $ \cT'$, an element of $ K_{G \times G',c}(\left(\cT  \times \cT' |Z  \times  Z' \right)^{T \cT - \frak{g} + T \cT' - \frak{g'}})$ determines a $G$-transversely elliptic operator on $ \cT$, with symbol  invertible away from $Z$ (as well as outside a compact set), so that the Atiyah-Singer index for transversely elliptic operators \cite{Atiyah} assigns to each such operator a virtual vector space given by the $G$-invariant part of the kernel and co-kernel. (Explicitly one can take the kernel and cokernel of the associated $Spin^c$-Dirac operator.) Performing this construction in families \cite{Baldare} yields the family index map
	\begin{equation}
		K_{G \times G',c}(\left(\cT \times \cT'\right)^{T \cT  - \frak{g} + T  \cT' - \frak{g}'}) \to K_{G',c}\left(  \cT'| Z'^{T \cT' - \frak{g}'} \right).
	\end{equation}
	We summarise the above discussion as follows:
	\begin{lemma}\label{lem:morphisms_in_Kthy}
		The complex structures on $T \cT - \frak{g}$ and $T \cT' - \frak{g'}$ determine a map
		\begin{equation}
			K^{\dim \cT/G -*}_{G,c}(\cT|Z) \to   K^{\dim \cT'/G' -*}_{G',c}(\cT'|Z')
		\end{equation}
		associated to each morphism $\bT \to \bT'$ of global Kuranishi charts. If this morphism is an equivalence, the image of $[\bT]$ under this map agrees with $[\bT']$.
	\end{lemma}
	\begin{proof}
		Independence of our map on the choice of representation $W'$ follows in a standard way by using the direct sum of representations to reduce the problem to showing that the push-forward maps associated to an embedding of representations $W'_0 \subset W'_1$ agree, which is an application of the Thom homomorphism, compare the corresponding step in the proof of Lemma \ref{lem:properties_in_local_case}.
		
		The fact that equivalences preserve fundamental classes follows by checking that they are preserved by each of the elementary moves of Lemma \ref{lem:factor_equivalence}; for the first two moves, this is a consequence of the compatibility of (equivariant) Thom classes with change of groups, for the second move, this follows from the multiplicativity of such classes and their compatibility with Thom isomorphisms, and in the last case follows from locality (i.e. naturality of the Euler class).
	\end{proof}

	The theory of Chern classes and the formal group for $K$-theory are standard.  We now discuss the last two axioms: 
	
	\begin{lemma}
		Given a pullback diagram \eqref{eq:transverse_diagram}, we have  $f_* [\bT_{0}] =  [\bT_{1}] \cap \pi_1^*( g_* [F_{0}])$.
	\end{lemma}
	
	\begin{proof}
		This is a variant of results which have already been used in the study of algebraic K-theoretic Gromov-Witten invariants, e.g. \cite[Proposition 3]{Lee-QKFoundations}. Up to stabilising $F_1$ and $\bT_1$, the map $F_0 \to F_1$ can be assumed to be an $G$-embedding of transverse complex manifolds, hence has a push-forward in $K_G$-theory coming from the Thom class of the normal bundle, and the $K$-theory Thom class is given by the Clifford complex.  Since we have a pullback diagram,  the map $\cT_0 \to \cT_1$ is also an embedding of $G$-manifolds, so the construction of the push-forward in transverse $K$-theory follows the more elementary template from Section \ref{sec:transverse_Kthy} and is again associated to the Clifford complex of the normal bundle. The result then follows from the naturality  of the Euler class. 
	\end{proof}

	\begin{lemma} A complex analytic normal crossing divisor has a fundamental class, which is compatible with pullback, given by applying the multiplicative formal group law to the fundamental classes of its top strata.
	\end{lemma}
	\begin{proof}
		This was explained in Section \ref{Subsec:divisor} and Example \ref{ex:divisor_Ktheory}. 
	\end{proof}
	
	\begin{Example} \label{Ex:Ktheory_not_Borel}
		For a finite group $G$, $K_G(pt) = R(G) \to \bZ$ has a trace which is defined by taking the dimension of the invariants; since characters are orthonormal this induces a non-degenerate pairing on $R(G)$, which can be interpreted as a Poincar\'e duality statement for the $K$-theory of the quotient of the point by $G$, considered as an orbifold. For example, let $G=C_2$ be the cyclic group of  two elements. Then $K_G(pt) = R(G) = \bZ[1,x]$ where $x=\bC-1$ is the difference of the sign and the trivial representation. The pairing has matrix $\bigl( \begin{smallmatrix} 1 & -1 \\ -1 & 2 \end{smallmatrix} \bigr)$  
		in this basis, and is unimodular.  
		
		The Borel equivariant $K$-theory in this case is given by completion at the augmentation ideal $R(G)^{\wedge}_I = \bZ[1] \oplus \bZ_2[x]$, where $\bZ_2$ denotes the $2$-adics. This has no non-degenerate pairing associated to a $\bZ$-valued trace.  By contrast, if $\bK = K_2(1)$ is the first Morava $K$-theory at the prime $2$, we have $H^*(BG;\bK) = \bZ/2 [1] \oplus \bZ/2 [x]$ and the trace induces the non-degenerate pairing 
		$\bigl( \begin{smallmatrix} 1 & 1 \\ 1 & 0 \end{smallmatrix} \bigr)$. 
		This illustrates that there is no Poincar\'e duality isomorphism for orbifolds in Borel-equivariant $K$-theory (even though there is one prime at a time), and for $K$-theory working Borel-equivariantly does not yield a counting theory.
	\end{Example}

	\begin{rem} Shaoyun Bai suggested that, by working consistently at the level of the orbifold $\scrT/G$ rather than $G$-equivariantly on $\scrT$, one could replace the use of transverse $K$-theory with push-forward maps  for orbifold $K$-theory (for representable maps these are constructed in \cite{Bunke-Schick}). Again the theory is underpinned by a suitable family index theorem. \end{rem}

	\subsection{Standing notation\label{Sec:notation}}
	
	We will write $\bE$ for a counting theory on $\Glo$, and denote by $E^*(\bT)$ (and $E_*(\bT), E_*^{lf}(\bT)$ etc) the associated (co)homology groups on global charts; this notation encompasses both the cases $H^*(\bT;\bK)$ arising from a Morava-local theory and $K^*(\bT)$ arising from complex $K$-theory. Furthermore, if $\ccMbar$ is a complex orbifold (in particular, a Deligne-Mumford moduli space of stable curves), we will write $E^*(\ccMbar)$ for the output of the counting theory $\bE$ on a global chart $\bT(\ccMbar) = (G,\widetilde{\ccM},0,0)$ arising from a presentation of $\ccMbar = \widetilde{\ccM}/G$ as a global quotient. This both simplifies notation and reflects the fact that the invariants we consider are naturally associated to the underlying complex orbifold.  (Thus, $K^*(\ccMbar)$ really denotes the \emph{orbifold} $K$-theory $K_G^*(\widetilde{\ccM})$, but we will not labour that.)

	A compact symplectic $2n$-manifold $X$ has an almost complex structure which is canonically defined up to homotopy. It defines a global chart with $G=\{e\}$ and trivial obstruction bundle. The axioms of a counting theory include a duality isomorphism
	\[
	E^*(X\times X) \cong E_{4n-*}(X\times X)
	\]
	and the diagonal $\Delta \subset X\times X$ defines a class in $E^{2n}(X\times X)$. We will write $\Delta$ for the class in either homology or cohomology, depending on the context.
	
	\section{Global Kuranishi charts for Gromov-Witten theory}
	\label{sec:glob-kuran-charts}
	
	In this section we construct global charts for moduli spaces of stable $J$-holomorphic maps from nodal curves to a symplectic manifold $X$ with tame almost complex structure $J$ in a fixed homology class $\beta \in H_2(X;\bZ)$.  The thickening $\scrT$ will admit a submersion over a smooth quasi-projective open $\scrF$ in the space of stable maps to some projective space, through which the stabilisation map $\scrT \to \ccMbar_{g,h}$ to the moduli space of domains will factor.  The resulting structure $\scrT \to \scrF \to \ccMbar_{g,h}$, with the first map a submersion to a smooth $G$-manifold and the second a map of $G$-spaces with trivial obstruction bundle, naturally fits with the axioms of a counting theory, once we use Lemma \ref{lem:subcategory} to replace $\scrF$ by a transverse complex $G$-manifold.

	We will use a slightly different approach to global Kuranishi charts than \cite{AMS-Hamiltonian}.
	Very roughly, in both approaches, we wish to associate a natural vector space to each stable curve together with a surjection to the cokernel of the $\overline{\partial}$-operator.
	In the paper \cite{AMS-Hamiltonian} such a vector space consisted of sections of a holomorphic vector bundle over this curve.
	In the present paper, we will explicitly construct such a space using eigenvalues of a Laplacian (See Section \ref{Subsec:preliminary}).
	The advantage of the former approach is that one can use a Gromov trick to describe the resulting thickening as a moduli space of holomorphic curves in a different symplectic manifold (See \cite[Section 6.3]{AMS-Hamiltonian}).
	The disadvantage of that approach is that one needs to use H\"{o}rmander techniques to prove that the vector space indeed surjects onto the cokernel of the $\overline{\partial}$-operator. 	In our current approach, proving that the vector space surjects onto the cokernel is almost immediate.

	\begin{enumerate}
		\item Sections \ref{Subsec:preliminary} to \ref{Subsec:easier} construct global charts for the moduli space of holomorphic maps from curves of a fixed genus to a target symplectic manifold, representing a prescribed homology class. The output of the construction has total space which is naturally a $G$-almost complex manifold, which can be made transversely almost complex by the procedure from Lemma \ref{lem:subcategory}.
				\item Sections \ref{sectiongeneralprethicken} to \ref{sectionequivalences} generalize the global Kuranishi chart construction
		so that it can be applied to related moduli spaces arising either in comparing choices or establishing axioms.
		\item Section \ref{subsection independence} shows that the global Kuranishi chart construction for $\ccMbar_{g,h}(X,J,\beta)$ does not depend on choices up to equivalence.
		\item Section \ref{sectionforgetful} discusses the forgetful map.
				\item Section \ref{sec:split_curves} respectively  \ref{sec:genusreduction} construct global Kuranishi charts for split respectively self-glued curves.
						\item Finally, Section \ref{Subsec:constant_maps}  explains how to incorporate stabilisation maps to moduli spaces of domains, when  these are themselves presented as global quotient orbifolds.		
	\end{enumerate}

	Deferred to an Appendix to help readability, we note that:
	
	\begin{enumerate}
	\item Section \ref{Subsec:gluing} states the gluing theorem needed for the construction of these global charts.
\item Sections \ref{subsec:FOOOlocalslices} and \ref{sec:comp-glob-kuran} compare our construction (locally) with Fukaya, Oh, Ohta, and Ono's approach which relies on stabilising divisors.
\end{enumerate}

	\subsection{A Preliminary Lemma}\label{Subsec:preliminary}
	
	We will consider spaces of perturbed holomorphic curves, where the perturbations  to the $\cdbar$-operator are drawn from large finite-dimensional subspaces of the universal space of perturbations.  The following definition captures this recurring set-up:

	\begin{defn} \label{defn fd}
		Let $G$ be a compact Lie group acting smoothly on a manifold $B$, and $\pi : V \lra{} B$ be a smooth $G$-vector bundle.
		A \emph{finite dimensional approximation scheme}
		$(V_\mu,\lambda_\mu)_{\mu \in \bN}$
		for $C^\infty_c(V)$
		is a sequence of finite dimensional $G$-representations $(V_\mu)_{\mu \in \bN}$
		and a sequence of
		$G$ equivariant linear maps:
		\begin{equation}
			\lambda_\mu : V_\mu \to C^\infty_c(V), \ \mu \in \bN
		\end{equation}
		to the space of smooth sections of $V$, 
		satisfying
		\begin{enumerate}
			\item $V_\mu$ is a subrepresentation of $V_{\mu+1}$ for each $\mu$,
			\item $\lambda_\mu|_{V_{\mu-1}} = \lambda_{\mu-1}$ for each $\mu$ and
			\item \label{item degree condition} the union of the images $\lambda_\mu(V_{\mu})$  is dense in $C^\infty(V)$
			with respect to the $C^\infty_{loc}$-topology.
		\end{enumerate}
	\end{defn}

	\begin{lemma}
		A finite dimensional approximation scheme exists.
	\end{lemma}
	\begin{proof}
		We will first prove this when the base $B$ is closed.
		Choose a $G$ invariant metric and compatible connection $\nabla$ on $V$.
		Let $\Delta : C^\infty(V) \to C^\infty(V)$ be the Laplacian given by the trace of $\nabla^2$.
		We define $V_\mu$ to be the sum of the first $\mu$ real non-negative eigenspaces of $\Delta$
		and $\lambda_\mu : V_\mu \hookrightarrow C^\infty(V)$ the natural inclusion map.
		The $G$ action preserves $V_\mu$ and $\lambda_\mu$.
		Also since $\Delta$ is a self adjoint elliptic operator, the sum of these eigenspaces is $C^\infty$ dense
		in $C^\infty(V)$ and hence the union of the images of $\lambda_\mu$ is too.
		
		Now let us prove this lemma in the case where $B$ is the interior of a compact manifold with boundary $\overline{B}$ with the property that the $G$-action and $G$-bundle $V$ extend to $\overline{B}$
		and $\overline{V}$ respectively.
		We can glue two copies of $\overline{B}$ along their boundaries giving its double $B_2$; the vector bundle $\overline{V}$ also doubles to a $G$-vector bundle $V_2$.
		Now construct $G$-invariant subspaces $V_{\mu,2} \subset C^\infty(V_2)$
		whose union is dense in $C^\infty(V_2)$ and we let $\lambda_{\mu,2} : V_{\mu,2} \lra{} C^\infty(V_2)$
		be the natural inclusion maps.
		Choose a sequence of $G$-invariant compactly supported
		bump functions $\rho_\mu : B \lra{} [0,1]$, $\mu \in \bN$ so that $\cup_{\mu \in \bN} \rho_\mu^{-1}(1) = B$.
		Then the maps 
		\begin{equation}
			\lambda_\mu : V_{\mu,2} \lra{} C^\infty_c(V), \quad \lambda_\mu(b) := \rho_\mu(b)\lambda_{\mu,2}(b)
		\end{equation}
		satisfy the desired properties.
		
		Finally, if $B$ is a general open manifold, then it is a countable union of open subsets $(B_j)_{j \in \bN}$
		whose closure is a codimension $0$ submanifold with boundary.
		Let $\lambda_{\mu,j} : V_{\mu,j} \lra{} C^\infty_c(V|_{B_j}) \subset C^\infty_c(V)$
		be $G$-equivariant linear maps satisfying properties
		(1)-(3) for $V|_{B_j}$.
		Then the sums $\lambda_{\mu} := \oplus_{j=1}^\mu \lambda_{\mu,j}$, $\mu \in \bN$,
		satisfy the desired properties for $V$.
	\end{proof}

	\subsection{Moduli Spaces of Curves in Projective Space}\label{Subsec:moduli}
	Fix $g,h,d \in \bN$.
	Let $\scrF_{g,h,d} \subset {\ccMbar}_{g,h}(\C \bP^{d-g},d)$
	be the subspace of stable nodal genus $g$ curves $\phi : \Sigma \to \C \bP^{d-g}$ of degree $d$ with $h$ marked points
	so that
	\begin{enumerate}
		\item \label{item regular}
		we have that $H^1(\phi^* O(1)) = 0$ and
		\item the automorphism group of each map is trivial.
	\end{enumerate}
	By the Euler exact sequence, such curves $\phi$ satisfy 
	the regularity condition $H^1(\phi^*T\C \bP^{d-g}) = 0$, which in turn implies that $\scrF_{g,h,d}$ is a smooth quasi-projective variety.  This would suffice in this section, but the stronger vanishing condition 
	is useful later, cf. Lemma \ref{lem:submerse_to_domains}.
		
	We define $\pi_{g,h,d} : \scrC_{g,h,d} \lra{} \scrF_{g,h,d}$
	to be the universal curve
	and $\scrC^o_{g,h,d}$ the complement of its marked points and nodes.
	We let $\scrC_{g,h,d}|_\phi$ (resp. $\scrC^o_{g,h,d}|_\phi$) be the fiber of $\scrC_{g,h,d}$ (resp. $\scrC^o_{g,h,d}$) over $\phi \in \scrF_{g,h,d}$.
	We let $\omega_{\scrC_{g,h,d}/\scrF_{g,h,d}}$
	be the relative dualizing sheaf of this universal curve.
	The group $GL_{d-g+1}(\bC)$
	acts on all of these spaces making
	$\pi_{g,h,d}$ equivariant.

	\begin{remark} \label{lemma local submersion}
		It can be shown that each point $\phi \in \scrF_{g,h,d}$, admits neighborhood $U_\phi$
equipped with a submersion $U_\phi \lra{} \ccMbar_{g,h+r}$ for some $r$.
	\end{remark}

	For the construction of the thickening of moduli spaces of maps, we shall need a variant of this construction: define $\scrF^{(2)}_{g,h,d} \subset {\ccMbar}_{g,h}((\C \bP^{d-g})^2,(d,d))$
	to be the subspace of stable nodal genus $g$ curves $u : \Sigma \to (\C \bP^{d-g})^2$ of bidegree $(d,d)$ with $h$ marked points
	so that if we compose $u$ with either projection map to $\C \bP^{d-g}$,
	we get an element of $\scrF_{g,h,d}$.
	This can be seen as a space of pairs of maps $(u,v)$ where $u,v : \Sigma \lra{} \C \bP^{d-g}$ is in $\scrF_{g,h,d}$. 	There is a natural diagonal $GL_{d-g+1}(\C)$ action on $\scrF^{(2)}_{g,h,d}$
	given by sending $(u,v)$ to $(u \cdot g, v \cdot g)$ for each $g \in GL_{d-g+1}(\C)$.
	
	Let $\Delta : \scrF_{g,h,d} \lra{} \scrF^{(2)}_{g,h,d}$ be the diagonal embedding
	sending $u : \Sigma \lra{} \C \bP^{d-g}$ to the map $(u,u)$.
	Let $H_{g,h,d} \lra{} \scrF_{g,h,d}$ be the vector bundle whose fiber over $u$
	is $H^0(u^*T \C \bP^{d-g})$.	
	
	Let $\Pi : \scrF^{(2)}_{g,h,d} \lra{} \scrF_{g,h,d}$
	be the projection map to the first factor sending $(u,v)$ to $u$.
	On a neighborhood $U_{g,h,d} \subset \scrF^{(2)}_{g,h,d}$ of the image of $\Delta$,
	this map is a submersion
	and the tangent space of the fiber of $\Pi|_{U_{g,h,d}}$ over a diagonal element $(u,u)$, $u : \Sigma \lra{} \C \bP^{d-g}$
	is $H^0(u^*T \C \bP^{d-g})$.
	As a result, after shrinking $U_{g,h,d}$,
	we can choose a $U(d-g+1)$-equivariant fiber preserving diffeomorphism
	\begin{equation} \label{eqn diagonal neighborhood}
		\widetilde{\Delta} : H_{g,h,d} \to U_{g,h,d}
	\end{equation}
	whose restriction to the zero section is $\Delta$.

	\subsection{Domain Metrics}
	
	We wish to put a coherent family of metrics on all
	domains of appropriate maps from a nodal curve to
	a symplectic manifold. We shall introduce a general framework for such a construction in this section, which we will later specialise to different settings. We thus find it convenient to consider a compact Lie group $G$, and write	$G_\C$ for its complexification.

		\begin{defn} \label{defn equivariant family}
		A \emph{$G_\C$-equivariant family of nodal curves}
		is a flat family of nodal curves $\pi_\scrF : \scrC \to \scrF$ (possibly with marked points defined by sections of $\pi_\scrF$)
		where $\scrF$ is a smooth quasi-projective variety and where $G_\bC$
		acts on the domain and codomain of $\pi_\scrF$  so that
		\begin{enumerate}
			\item \label{itempiscrf} $\pi_\scrF$ is $G_\bC$-equivariant.
			\item The marked point sections are also required to be $G_\bC$-equivariant.
			\item \label{itemfaithful} For each $\phi \in \scrF$, the stabilizer group of $\phi$ acts faithfully on the fiber $\scrC|_\phi$.
		\end{enumerate}
	We define $\scrC^o$ to be the complement of the nodes and marked points.
	\end{defn}

	The faithfulness assumption \eqref{itemfaithful} is needed in Lemma \ref{lemmafFpalaiproper} later on, which in turn is needed to put an appropriate invariant metric on an infinite dimensional space of curves.
	
	\begin{example}
		The family of curves $\scrC_{g,h,d} \to \scrF_{g,h,d}$
		defined in Section \ref{Subsec:moduli}
		is a $G_\C$-equivariant family of nodal curves where $G = U(d-g+1)$
		and $G_\C = GL_{d-g+1}(\bC)$. 	Note that we did not assume that the stabilizer group is finite, as this key example includes situations where the stabilizer group is non-compact, due to the presence of rational components which carry fewer than $3$ marked points or nodes.
	\end{example}
	
	Let $(X,\omega)$ be a closed symplectic manifold, let
	$\beta \in H_2(X;\Z)$.
	Let $J$ be an $\omega$-tamed almost complex structure.
	Let $\pi_\scrF : \scrC \to \scrF$
	be a $G_\bC$-nodal family as in Definition \ref{defn equivariant family}, whose fibres have
        $h$ marked point that define sections $p_1,\cdots,p_h : \scrF \to \scrC$.
	Energy quantization implies there is a minimal bound  on the energy (i.e. integral of $\omega$) over all unstable irreducible components $\Sigma' \subset \Sigma$ of $J$-holomorphic stable maps $u : \Sigma \lra{} X$ in class $\beta$. We denote this bound by  $e > 0 $.

	\begin{defn} \label{defn space of curves}
		Let
		$\fF_\scrF$ be the infinite dimensional space of
		all tuples $(\phi,u)$ where $\phi \in \scrF$ and $u : \scrC|_\phi \lra{} X$
		is a smooth map representing $\beta$ with the property that
		\begin{enumerate}
		\item $u$ is $J$-holomorphic in a small neighborhood of its nodes,
		\item $\int_{\Sigma'} u^*\omega \geq 0$ for each
		irreducible component $\Sigma'$ of the domain of $u$, and
		\item $\int_{\Sigma'} u^*\omega \geq e$ if $\Sigma'$ is an unstable component of $\Sigma$.
		\end{enumerate}
		The topology on $\fF_\scrF$ is defined as follows:
		let $Y := \Omega^{0,1}_{\scrC^o/\scrF} \otimes_\C TX$, and
		let $\Gamma_u \subset \scrC|_\phi \times X \subset \scrC \times X$
		be the graph of $u$.
		The operator $\overline{\partial}_J$ defines a smooth section of $Y|_{\Gamma_u}$
		and hence the image of this section
		is a closed subset $C_u$ of the total space of $Y$.
		The topology on $\fF_\scrF$ is induced by the Hausdorff metric
		on such closed subsets.

		We define the \emph{universal curve}
		$\pi_\fC : \fC_\scrF \lra{} \fF_\scrF$ over $\fF_\scrF$ to be the pullback of the curve $\scrC \to \scrF$ via the natural projection map:
		\begin{equation}
			\fF_\scrF \to \scrF, \quad (\phi,u) \to \phi.
		\end{equation}
	\end{defn}

        We shall presently use the fact that the topology on $\fF_\scrF$ is metrizable. 
	
	\begin{remark} \label{remarkelliptic}
		By elliptic regularity, we have that if a sequence $(\phi_i,u_i)$ converges to $(\phi,u)$ in $\fF_\scrF$,
	then $u_i$ converges to $u$ in a $C^1_{loc}$ sense away from the nodes after parameterizing
	 these domains appropriately.	
	\end{remark}
	
	One consequence of the previous remark is that the restriction of the topology on $\fF_\scrF$ to a curves whose domain is a fixed element $\phi$ of $\scrF$
	is the $C^1$ topology and hence it is separable.
	There is also natural continuous $G_\bC$ action on
	$\fF_\scrF$ and $\fC_\scrF$ induced from the one on $\scrF$ and $\scrC$
	making the natural map $\pi_\fC : \fC_\scrF \lra{} \fF_\scrF$
	equivariant.

	\begin{defn} 
		A \emph{fiberwise metric} on $\fC_\scrF$
		is a continuous map
		$\mu : \fC_\scrF \times_{\fF_\scrF} \fC_\scrF \to \bR_{\geq 0}$
		so that the restriction of $\mu$ to each irreducible component of the fiber over an element of 
	$\fF_\scrF $ 	is a distance metric induced from a smooth Riemannian metric.
      \end{defn}

For the next definition, let $\omega_{\fC_\scrF/\fF_\scrF}(p_1,\cdots,p_h)$ be the pullback of the dualizing
line bundle $\omega_{\scrC/\scrF}$ to $\fC_\scrF$, twisted by the marked point sections $p_1,\cdots,p_h$.
\begin{defn} \label{defn domain metric}
A \emph{consistent domain metric} for $\scrF$ consists of a fiberwise metric on $\fC_\scrF$ 
		together with a Hermitian metric on $\omega_{\fC_\scrF/\fF_\scrF}(p_1,\cdots,p_h)$
		with the property that both are invariant under the
		$G_\bC$-action on $\fC_\scrF$.  
\end{defn}
		
Since the notion of a consistent domain metric is local in $ \fF_\scrF $, it makes sense for any $G_\bC$-invariant open subset $U \subset \fF_\scrF$. 

	We wish to equip the fibres of $\pi_\fC$ with such a consistent domain metric. We emphasise that this requires equivariance under the non-compact group $G_\bC$, hence cannot be achieved by a naive averaging argument. 
	We will construct this metric locally and then patch these metrics together using a $G_\bC$-equivariant partition of unity (or equivalently, a partition of unity
	on the quotient $\fF_\scrF / G_\bC$).
	To show that that such a partition of unity exists, we will be using a classical result of Palais \cite{palais1961existence}, which we now recall:
        
	Let $H$ be a Lie group (possibly non-compact) and $Y$ a $H$-space which is Tychonoff.

	\begin{defn} \cite[Definition 1.1.1, 1.2.1, 1.2.2]{palais1961existence} \label{defpalaisd}
		For any two subsets $U,V \subset Y$, we write $((U,V)) := \{h \in H \ : \ hU \cap V \neq \emptyset \}$.
		We say that $U$ is \emph{thin} relative to $V$ if $((U,V))$ has compact closure in $H$.
		We say that $U$ is \emph{thin} if $U$ is thin relative to itself.
		A subset $S \subset Y$ is a \emph{small subset} if every point in $Y$
		admits a neighborhood $U$ which is thin relative to $S$.
		We say that the $H$-action on $Y$ is \emph{Palais proper}
		if every point $y \in Y$ admits a neighborhood which is a small subset.
	\end{defn}

	\begin{theorem} \cite[Theorem 4.3.4.]{palais1961existence}
		If $Y$ is a metrizable, separable and  Palais proper $H$-space, then
		$Y/H$ is metrizable and separable.
	\end{theorem}

	The main point of the theorem above is that it works for group actions on infinite dimensional spaces. We are particularly interested in the infinite dimensional space $\fF_\scrF$.

	\begin{corollary} \label{corollarypalais}
		If $Y$ is a metrizable, separable and Palais proper $H$-space,
		and if  $(U_i)_{i \in I}$ is an open cover for which each $U_i \subset Y$
		is $H$-invariant, there is an $H$-equivariant partition of unity subordinate to this cover.
	\end{corollary}

        We now check that the situation we are studying fits within Palais's framework:

	\begin{lemma} \label{lemmafFpalaiproper}
		$\fF_\scrF$ is a Palais proper $G_\bC$-space.		
	\end{lemma}
	\begin{proof}
		Choose a smooth metric on $\scrC$
		as well as a  proper smooth  function $f : \scrC \to [0,\infty)$.
		Now let $L : \fF_\scrF \to (0,\infty)$
		send $(\phi,u)$ to the Lipschitz number of $u$ using the metric on $\scrC|_\phi$ and $X$.
		Define
		\begin{equation}
			F : \fF_\scrF \to \bR, \ F((\phi,u)) := L((\phi,u)) + f(\phi).
		\end{equation}
		To prove that $\fF_\scrF$ is Palais proper,
		it is sufficient for us to show that the sublevel sets of $F$
		are thin in the sense of Definition \ref{defpalaisd}
		since $\fF_\scrF$ is covered by them.
		
		Let $U := F^{-1}([0,C))$ for some $C>0$.
		Let $\underline{U} \subset \scrF$ be the image
		of $U$ in $\scrF$.
		Let $(g_i)_{i \in \bN}$
		be an unbounded sequence of elements in $G_\bC$.
		Suppose that $g_i \cdot \underline{U} \cap \underline{U} \neq \emptyset$
		for each $i \in \bN$.
		So, there exist $(\phi_i,u_i) \in U$
		satisfying $(g_i \phi_i, u_i \circ g_i^{-1}) \in U$
		for each $i \in \bN$.
		After passing to a subsequence,
		we can assume $\phi_i \to \phi_\infty$ as $i \to \infty$
		since $f$ is proper.
		Let $G' \subset G_\bC$ be the stabilizer group of $\phi_\infty$.
		There exists a sequence $(q_i)_{i \in \bN}$
		in $G'$ so that $g_i q_i^{-1} \to \id$ in $G_\bC$ as $i \to \infty$.
		We also have that $u_i$ converges to a Lipschitz continuous function
		$u_\infty$.
		Now since unbounded sequences of M\"{o}bius transformations
		have unbounded Lipschitz numbers, we have that
		the Lipschitz number of $u_\infty \circ q_i^{-1}$ tends to infinity.
		Hence the Lipschitz number of $u_i \circ g_i^{-1}$ tends to infinity
		too.
		But this implies $(\phi_i,u_i) \notin U$ for $i$ large enough.
		Hence $g_i \cdot U \cap U = \emptyset$ for all $i$ large enough.
	\end{proof}

	\begin{defn}
		We say a submanifold $S \subset X$ is \emph{transverse} to a smooth nodal curve $u : \Sigma \to X$
		if $u(\Sigma^{sing}) \cap S = \emptyset$ and $u|_{\Sigma - \Sigma^{sing}}$ is transverse to $S$
		where $\Sigma^{sing} \subset \Sigma$ is the subset of nodes.
	\end{defn}
	
	Note that by Remark \ref{remarkelliptic},
	we have that if $(\phi,u) \in \fF_\scrF$ is transverse to $S$
	then there is a neighborhood $U$ of $(\phi,u)$ in $\fF_\scrF$
	with the property that every element of $U$ is transverse to $S$.

	\begin{defn} \label{defnDcompatible}
		Let $D \subset X$ be a codimension $2$ submanifold.
		We say that $(\phi,u) \in \fF_\scrF$ is \emph{$D$-compatible}
		if 
		\begin{enumerate}
			\item $u$ is transverse to $D$, and
			\item there exist points $p_1,\cdots,p_{h_u} \in \scrC|_\phi$
			making $(\scrC|_\phi, p_1,\cdots,p_{h_u})$ into a stable marked curve, and
			\item $\{p_1,\cdots,p_{h_u}\} = u^{-1}(D)$.
		\end{enumerate}
		We call $\{p_1,\cdots,p_{h_u}\}$ a collection of \emph{$D$-compatible marked points}.
		We will let $U_D \subset \fF_\scrF$ be the subspace of $D$-compatible
		curves. 
	\end{defn}
	
Note that $U_D$ is $G_{\bC}$-invariant, since the action of $G_{\bC}$ corresponds to pre-composing the map $u$ with an identification of domains, so that transversality of curves to divisors in $X$ is invariant under the action.
	\begin{lemma} \label{lemmalocal}
		For each codimension $2$ submanifold $D$, we have that $U_D$ admits a consistent domain metric.
	\end{lemma}
	\begin{proof}
		Let $\ccCbar_{g,h} \to \ccMbar_{g,h}$ be the universal curve over the moduli space
		of curves of genus $g$ with $h$ marked points for each $h \in \bN$ with $(g,h) \neq (0,0),(0,1),(0,2),(1,0)$.
		Choose a smooth metric on $\ccCbar_{g,h}$ invariant under permuting the marked points
		for each $h \in \bN$.
		Also choose a Hermitian metric on the relative dualizing bundle $\omega_{\ccCbar_{g,h}/\ccMbar_{g,h}}$.
		Let $(\phi,u) \in U_D$.
		Then we have $D$-compatible marked points $\{p_1,\cdots,p_{h_u}\}$ on $\scrC|_\phi$.
		Hence the domain $\scrC|_\phi$ is identified with a fiber of $\ccCbar_{g,h_u}$
		which is unique up to permuting marked points.
		Hence we can pull back the metric $\ccCbar_{g,h_u}$ to $\scrC|_\phi$
		and the Hermitian metric to $\omega_{\scrC|_\phi}$.
		All of these metrics vary continuously with respect to $(\phi,u)$ so they assemble together giving a consistent domain metric on $U_D$.
	\end{proof}

	\begin{lemma} \label{lemma metrics on domains}
		$\scrF$ admits a consistent domain metric.
	\end{lemma}
	\begin{proof}
		Let $\scrD$ be the set of codimension $2$ submanifolds of $X$.
		Then the set $(U_D)_{D \in \scrD}$ (Definition \ref{defnDcompatible}),
		defines a $G_\bC$-equivariant open cover of $\fF_\scrF$.
		Hence we have a $G_\bC$-equivariant
		partition of unity $(\rho_D)_{D \in \scrD}$
		subordinate to this cover (Corollary \ref{corollarypalais}).
		By Lemma \ref{lemmalocal}, there is a consistent domain metric on $U_D$
		for each $D \in \scrD$.
		We now patch together these consistent domain metrics using this partition of unity.
		Let us do this for the distance metrics $\mu_D$, $D \in \scrD$.
		Define
		\begin{equation}
			\mu'_D : \fF_\scrC \times_{\fF_\scrF} \fF_\scrC \to \bR_{\geq 0}, \ \mu'_D(x,x') = \left\{  \begin{array}{ll}
				\rho_D(\pi_\scrF(x)) \mu_D(x,x') & \textnormal{if} \ \pi_\scrF(x) \in U_D \\
				0 & \textnormal{otherwise}.
			\end{array}\right.
		\end{equation}
		Then we have $\sum_{D \in \scrD} \mu'_D$ is a distance metric which is $G_\bC$-invariant.
		A similar argument works for Hermitian metrics on the relative dualizing bundle.
	\end{proof}

	\subsection{Basic Construction of Global Kuranishi Charts}\label{Subsec:easier}
	
	Let $(X,\omega)$ be a closed symplectic manifold, let
	$\beta \in H_2(X;\Z)$ and let $g,h,d \in \bN$.
	For now, $d$ is arbitrary but later on we will fix it.
	Let $J$ be an $\omega$-tame almost complex structure.
	To avoid clutter, we write $\scrF := \scrF_{g,h,d}$,
	$\scrC := \scrC_{g,h,d}$, $\scrC^o := \scrC^o_{g,h,d}$,
	$\pi := \pi_{g,h,d}$ and
	$\scrF^{(2)} := \scrF^{(2)}_{g,h,d}$.
	Also let $H_{g,h,d}$, $U_{g,h,d}$ and $\widetilde{\Delta}$ be as in Equation \ref{eqn diagonal neighborhood}.
	These manifolds are naturally $G$-manifolds where $G := U(d-g+1)$.
	For any element $\phi$ in one of these manifolds and any $g \in G$
	we write $\phi \cdot g$ for the corresponding element acted on by $g$.

	\begin{defn} \label{bundleY} 
		Let $Y := \Omega^{0,1}_{\scrC^o/\scrF} \otimes_\C TX$
		be the $G$ vector bundle over $\scrC^o \times X$ whose fiber over a point $(c,x) \in \scrC^o \times X$ is the space of anti-holomorphic maps from $T_c(\scrC^o|_{\pi(c)})$ to $T_x X$.
	\end{defn}
	
	Choose a finite dimensional approximation scheme $(V_{\mu},\lambda_{\mu})_{\mu \in \bN}$
	for $C^\infty_c(Y)$ as in Definition \ref{defn fd}.

	\begin{defn} \label{defn prethickenedmodulispace}
		We define the \emph{pre-thickened moduli space}
		$\scrT^\pre = \scrT^\pre_{g,h,d}(\beta,V_{\mu},\lambda_{\mu})$
		to be the space of triples $(\phi,u,e)$,
		where $(\phi,u) \in \fC_\scrF$ (Definition \ref{defn space of curves}) and $e \in V_\mu$, satisfying the following equation:
		\begin{equation} \label{eqn lambdapde}
			\overline{\partial}_J u|_{\scrC^o|_{\phi}} + (\lambda_\mu(e)) \circ \Gamma_u =0
		\end{equation}
		where
		\begin{equation}
			\Gamma_u : \scrC^o|_\phi \to \scrC^o \times X, \quad \Gamma_u(\sigma) := (\sigma,u(\sigma))
		\end{equation}				
		is the graph map.
	\end{defn}

	The space $\scrT^{\pre}$ carries a topology coming from the natural topology on $V_\mu$ and the Hausdorff distance topology
	on the graphs given by the closure of the image of $\Gamma_u$ in $\scrC \times X$.  Whilst $\fC_\scrF$ was infinite-dimensional, imposing the perturbed Cauchy-Riemann equation \eqref{eqn lambdapde} means that $\scrT^{\pre}$ is a finite-dimensional space.

		The group $G$ acts on $\scrT^\pre$
	by sending 
	\[
	(\phi,u,e) \, \mapsto \, (\phi \cdot g,u \cdot g,e \cdot g)
	\]
	for each $g \in G$ where $u \cdot g : \scrC|_{\phi \cdot g} \to X$ is the composition of the natural map $\scrC|_{\phi \cdot g} \lra{\cdot g^{-1}} \scrC|_\phi$ with $u$.
	\begin{rem} Note that every point in the region where $e = 0$ is fixed by the diagonal circle $S^1 \subset G$. When constructing the actual thickening of the moduli spaces of curves, we will enlarge $ \scrT^\pre$, with one effect being that the stabiliser groups all become finite.    
	\end{rem}

	\begin{defn}
		The \emph{pre-obstruction bundle}
		$E^\pre := E^\pre_{g,h,d}(\beta,V_\mu,\lambda_\mu)$
		over $\scrT^\pre$
		is the direct sum of the pullback of
		$H_{g,h,d}$ to $\scrT^\pre$
		with the trivial bundle $V_\mu$.
	\end{defn}
	
	The group $G$ acts on $H_{g,h,d}$
	in the natural way and hence we have a natural $G$-action on $E^\pre$ coming from the action above as well as the action on $V_\mu$.

	\begin{rem} \label{rem:pre_thickening_not_right_dimension}
		The virtual dimension of $\ccMbar_{g,h}(X,\beta,J)$ is \emph{not} equal to $\dim(\scrT^{\pre}) - \rk (E^{\pre})$.
				\end{rem}

	To proceed with our construction of the genuine thickening of the moduli spaces of curves, we need to imitate the main construction of \cite{AMS-Hamiltonian}, and realise $\C \bP^{d-g} $ as the (projectivisation) of a space of holomorphic sections of a line bundle on the domain of the perturbed pseudo-holomorphic maps that we are considering. In order to formulate this precisely, we need some more notation.

	\begin{defn} \label{defn line bundle data}
		A choice of \emph{line bundle data} associated to $(g,h)$ consists of a triple $\bL = (L,k,\mathfrak{D})$
		where
		\begin{itemize}
			\item  $L$ is
			a Hermitian line bundle
			over $X$ with curvature $-2\pi i \Omega$ where $\Omega$ is a symplectic form taming $J$,
			\item $k$ is a large integer,
			\item $\mathfrak{D}$ is a consistent domain metric for $\scrF$ as in Definition \ref{defn domain metric}.
                        \end{itemize}
                We assume that the line bundle $L$ admits a root of order at least $3$.		For each $(\phi,u,e) \in \scrT^\pre$,
		define
		$L_{u,\bL}$ to be the Hermitian line bundle over $\scrC|_\phi$
		equal to 
		\begin{equation} \label{eq:positive_line_bundle}
		L_{u,\bL} = (\omega_{\scrC_{g,h,d}/\scrF_{g,h,d}}(p_1,\cdots,p_h)|_\phi \otimes u^* L)^{\otimes k}
		\end{equation}
		where $\omega_{\scrC_{g,h,d}/\scrF_{g,h,d}}(p_1,\cdots,p_h)$ is the relative dualizing bundle
		of $\scrC_{g,h,d} \to \scrF_{g,h,d}$ and $p_1,\cdots,p_h$ are the divisors in $\scrC_{g,h,d}$
		corresponding to the marked points.
	\end{defn}
	The Hermitian structure $\langle -,-\rangle$ on $L_{u,\bL}$ comes from the Hermitian structure on $L$ and the Hermitian structure on $\omega_{\scrC_{g,h,d}/\scrF_{g,h,d}} (p_1,\cdots,p_h)|_\phi = \omega_{\fC_{\scrF_{g,h,d}}/\fF_{\scrF_{g,h,d}}}(p_1,\cdots,p_h)|_{(\phi,u)} $
	coming from $\mathfrak{D}$. The choice of $\mathfrak{D}$ also equips each fibre $\scrC|_\phi$ with a volume form $\Omega_{\scrC|_\phi}$.
	
	We define:
	\[
	d = d_{\bL} := k(\langle  [\Omega], \beta \rangle + 2g - 2 + h),
	\] which is the degree of $L_{u,\bL}$ (previously $d$ was some unspecified integer, but now we fix it via this formula).
        \begin{rem}
          The divisibility condition in Definition \ref{defn line bundle data} is required in order for the line bundle in Equation \eqref{eq:positive_line_bundle} to be positive on spheres with exactly one special point (a marked point or a node). Indeed, this is the only case in which the relative canonical line bundle is negative, so that we need the pullback of $L$ to be sufficiently positive in order for the tensor product to be ample. 

          For all our applications, it will suffice to require that $3 \leq k$, as the corresponding line bundle will then be very ample on every component of a curve representing an element of the moduli space of maps; this is straightforward for genus $0$ components, holds for genus $1$ components by the fact that a degree $3$ line bundle determines an embedding in projective space, and in general from the fact that the cube of the canonical bundle is very ample. This differs from the corresponding construction in \cite{AMS-Hamiltonian}, where the parameter $k$ was required to be arbitrarily large, as we used the sections of high degree line bundles to obtain deformations of the Cauchy-Riemann operator that achieve transversality. 
        \end{rem}
	
	\begin{defn} \label{defn Lframing}
		A \emph{holomorphic $\bL$-framing} on $(\phi,u,e) \in \scrT^\pre$ is a complex basis $F = (F_0,\cdots,F_{d-g})$ of $H^0(L_{u,\bL})$ for which the Hermitian matrix $H_F$ of inner products of basis elements,  with $(i,j)$-th entry
		\begin{equation} \label{eqn:L2}
		\int_{\scrC|_\phi} \langle f_i, f_j\rangle \Omega_{\scrC|_\phi},
		\end{equation} has positive eigenvalues.  	
			We call this a \emph{unitary $\bL$-framing} if $H_F$ is the identity matrix (i.e. if $F$ is a unitary basis).
	\end{defn}
	
	The  underlying $L^2$-inner product featuring in \eqref{eqn:L2} is intrinsic, i.e. independent of the choice of  holomorphic frame, since the metric coming from $\frak{D}$ is invariant under the whole of $G_{\bC}$. That invariance also gives  the following basic naturality property of the framing matrices:
	
	\begin{lemma} \label{lem:framing_matrices_transform}
	For any $A \in GL_{d-g+1}(\bC)$, we have $H_F \cdot A = H_{F\cdot A}$.
	\end{lemma}
	
	\begin{proof} This holds because the consistent domain metric $\mathfrak{D}$ for $\scrF$ is invariant under the full group $GL_{d-g+1}(\bC) = G_{\bC}$. \end{proof}
	
	We shall now make use of the following construction: starting with a curve $\phi \in  \scrF$, whose domain curve $\Sigma_\phi$ is tautologically equipped with a map
	\begin{equation}
		\phi :     \Sigma_\phi \lra{} \bC \bP^{d-g},                         
	\end{equation}
	each $\bL$-framing determines a possibly different map
	\begin{equation} \label{equation}
		\phi_F : \Sigma_\phi \lra{} \bC \bP^{d-g}, \quad \phi_F(\sigma) := [F_0(\sigma),\cdots,F_{d-g}(\sigma)].
	\end{equation}
	Since by taking $k\geq 3$ we can ensure that the bundle $L_{u,\bL} $ has the property that its first cohomology vanishes on all components of $\Sigma$, the map $\phi_F$ is injective, which means that it is automorphism-free.  In particular, $\phi_F$ defines another element of the space $\scrF$.

	\begin{defn} \label{defn thickened}	
		The \emph{thickened moduli space}
		$\scrT = \scrT_{g,h,d}(\beta,V_{\mu},\lambda_{\mu})$
		is the space of quadruples $(\phi,u,e,F)$
		where
		$(\phi,u,e) \in \scrT^\pre$ and
		$F$ is a holomorphic framing of $(\phi,u,e)$
		satisfying the property that the pair $(\phi,\phi_F)$ lies in the previously fixed neighbourhood $U_{g,h,d}$ of the diagonal.
		The \emph{obstruction bundle}
		$E = E_{g,h,d}(\beta,V_{\mu},\lambda_{\mu})$
		is the direct sum of a copy of the space $\scrH$ of $(d-g+1) \times (d-g+1)$ Hermitian matrices with the pullback of $E^\pre$ to $\scrT$ via the natural projection map forgetting the framing.
	\end{defn}
	Explicitly,  the fiber of the obstruction bundle $E$ over $(\phi,u,e,F)$ is $H_{g,h,d}|_\phi \oplus V_\mu \oplus \scrH$. 
	
	The next step is to define a $G$-equivariant section of the obstruction bundle.  We have an exponential map $\exp: \scrH \stackrel{\sim}{\longrightarrow} \scrH^+$ from Hermitian matrices to those with positive eigenvalues.
		
	\begin{defn} \label{defn Kurnishi chart}
		The \emph{global Kuranishi chart associated to $\ccMbar_{g,h}(X,J,\beta)$} is
		\begin{equation}
			\bT = \bT_{g,h,d}(\beta,V_{\mu},\lambda_{\mu}) = (G,\scrT,E,s),
		\end{equation}  where the section of the obstruction bundle $E$ over $\scrT  $ is given by
		\begin{equation}
			s : \scrT \lra{} E, \quad s(\phi,u,e,F) := (\widetilde{\Delta}^{-1}((\phi,\phi_F)),e, \exp^{-1}(H_F)) \in E|_\phi = H_{g,h,d}|_\phi \oplus V_\mu \oplus\scrH.
		\end{equation}
	\end{defn}
	
	Note that the above definition currently involves some abuse of terminology, as we have not explained in any sense what our construction has to do with $\ccMbar_{g,h}(X,J,\beta) $. This is remedied by the next result, where as in \cite{AMS-Hamiltonian}, it is crucial to the final point (`relative smoothness') that elements of $p^{-1}(\phi)$ are maps with a fixed domain $\Sigma = \scrC|_{\phi}$, where $p : \scrT \lra{} \scrF$ is the natural forgetful map taking $(\phi,u,e,F) \mapsto \phi$. 
		
	\begin{remark}
	The obstruction bundle summand $H_{g,h,d}$ in $E^\pre$ is used partly to deal with $\Pic^0$ issues. In the pre-thickened moduli space $\scrT^\pre$, $O(1)$ pulled back to the domain of $\phi$ can be any degree $d$ line bundle (these are parameterized by $\Pic^0$).
	However, our framing data $F$ selects only one such line bundle.
	The section of $H_{g,h,d}$ encodes the difference between such bundles; removing the indeterminacy this way is a variation of the `doubly framed curve' trick originally  introduced to establish independence of global charts on various choices.

	The same $\Pic^0$-issue is dealt with differently by Hirschi and Swaminathan in \cite{HS}. They pick out the correct line bundle by adding an additional equation to the obstruction section, valued in (the universal cover of) $\Pic^0$, leading to an arguably simpler construction. \end{remark}

	\begin{theorem} \label{thm:basicglobalchart}
		Assuming that  $\mu$ has been chosen sufficiently large (and $3 \leq k$), the space $\scrT$ is a topological manifold in a neighbourhood of $s^{-1}(0)$, on which $G$ acts with finite stabilisers and finitely many orbit types. Moreover, we have a natural homeomorphism
		\begin{equation} \label{eq:moduli_space_quotient}
			s^{-1}(0) / G = \ccMbar_{g,h}(X,J,\beta)
		\end{equation}
		which is an isomorphism of orbispaces (i.e. respects stabiliser groups), 
		and the natural forgetful map $p$ is a topological submersion with a $C^1_{loc}$-structure (See \cite[Definition 4.27]{AMS-Hamiltonian}).
	\end{theorem}

Note that the Gromov topology on $\ccMbar_{g,h}(X,J,\beta)$
coincides with the Hausdorff topology on $s^{-1}(0) / G$.
This is an identical argument to the one in \cite[Lemma 6.14]{AMS-Hamiltonian}.

	\begin{proof}
		We discuss the proof of the homeomorphism in Equation \eqref{eq:moduli_space_quotient} first.
Along $s^{-1}(0)$, we have that $e=0$, so that the underlying curve is a solution to the original holomorphic curve equation $\overline{\partial}_J u= 0$. 
Hence we have a natural projection map $s^{-1}(0) \to  \ccMbar_{g,h}(X,J,\beta)$.
We wish to show that $G$ acts transitively on the fibers of this map  (in fact the stabilizer group exactly corresponds to the automorphism group of $u$).
So, fix $u : \Sigma \to X$ in  $\ccMbar_{g,h}(X,J,\beta)$ and let $Q \subset s^{-1}(0)$
be the preimage of $u$.
For each $(\phi,u,e,F) \in Q$, we have $e=0$ and $\phi=\phi_F$
and so all such elements are of the form $(\phi_F,u,0,F)$
and hence only depend on the framing $F$, which is a basis of $H^0(L_u)$.
Also since  $\exp^{-1}(H_F) = 0$, we have that $H_F$ is unitary.
Fix such an element, $(\phi_F,u,0,F) \in Q$.
Any other element in $Q$ must be equal to
 $(\phi_{F \cdot A},u,0,F \cdot A) = (\phi_F  \cdot A,u,0,F \cdot A)$ for some $A \in GL_{d+1}(\bC)$.
 Now $H_F \cdot A = H_{F \cdot A}$ by Lemma \ref{lem:framing_matrices_transform}, and since both $H_F$ and $H_{F \cdot A}$
 are unitary, this implies that $A$ must be unitary and hence $G$
 acts transitively on $Q$.
 We can therefore identify the quotient of this zero-locus with the desired moduli space.
		
		To justify the statement about stabiliser groups and isomorphism of orbispaces, note that an element of $s^{-1}(0)$ is a tuple $(\phi_F, u, 0, F)$ which is completely determined by the framing $F$. An automorphism $f: \Sigma \to \Sigma$ as a stable map to $X$, so with $u \circ f = u$, will pull back the framing $F$ to some $F \cdot g$ with $\phi_F \circ f = \phi_{F\cdot g}$. On the other hand, for $g\in U(d-g+1)$ to stabilise $(\phi_F,u,0,F)$ we need some automorphism $f: \Sigma \to \Sigma$ of $u$ with $\phi_F \circ f = g \circ \phi_F = \phi_{F\cdot g}$.  It follows that the stabiliser groups agree as claimed.
		
		We will now show that the natural forgetful map $p$ is a submersion with a $C^1_{loc}$-structure.
		By unique continuation and elliptic regularity
		we have for each $(\phi,u,0) \in \scrT^\pre$ that $C^\infty_c(Y)$ surjects onto the cokernel of the linearized $\overline{\partial}$-operator $D : W^{1,2}(u^*TX) \to L^2(\Omega^{0,1}(\widetilde{u}^*TX))$ 
		where $\widetilde{u}$ is the composition of the normalization map $\widetilde{\Sigma} \to \Sigma$ with $u$.
		Since $(V_\mu,\lambda_\mu)_{\mu \in \bN}$ is a finite dimensional approximation scheme,
		we get that $\lambda_\mu(V_\mu)$ surjects onto the cokernel of $D$ for any fixed large enough $\mu$.  Gromov compactness then implies that $\mu$ may be chosen uniformly so that surjectivity holds for all curves.

		Hence by the gluing result Corollary \ref{cor algfam}, we have that ${\scrT'}^\pre \to \scrF$ admits a $C^1_{loc}$ structure where ${\scrT'}^\pre \subset \scrT^\pre$ is an open neighborhood of the region where $e=0$.
		
		Let $V \to \scrF$ be the vector bundle over $\scrF$ whose fiber
		over $u : \Sigma \to \bC \bP^{d-g}$ is $H^0(u^*{\mathcal O}(1))$.
		Choose a connection for $V$.
		Let $P \lra{} \scrF$ be the principal $GL_{d-g+1}(\bC)$-bundle over $\scrF$
		whose fiber over $u : \Sigma \lra{} \bC \bP^{d-g}$ consists of all possible bases for $V$.
		Let $\pi : \scrT^\pre \to \scrF$ send $(\phi,u,e)$ to $\phi$.
		Then the thickening $\scrT$ naturally embeds into the pullback of $P$ to $\scrT^{\pre}$  as an open subset since we have natural isomorphisms
		$H^0(L_{u,\bL}) \cong H^0(\phi_F^* O(1)) \cong H^0(\phi^*O(1)) = \pi^*V|_{(\phi,u,e)}$
		where the second isomorphism is given by parallel transport along the straight line
		joining $\phi$ and $\phi_F$ inside $H_{g,h,d}|_\phi \cong \Pi(\widetilde{\Delta}(H_{g,h,d}|_\phi)) \subset \scrF$
		where $\Pi$ is the projection map sending $(\psi_1,\psi_2)$ to $\psi_1$.
		Hence $\scrT' \to \scrF$ also has a $C^1_{loc}$-structure coming from the one on $\pi^*P|_{\scrT^\pre}$ where $\scrT'$ is the preimage of ${\scrT'}^\pre$.
	\end{proof}
	
	\begin{remark} Along $e=0$, all maps in the thickening are non-constant on any unstable component of a domain (since the finite automorphism hypothesis implies that these are stable maps). By continuity, and since the degree of $u^*L$ on a component is cohomological, this then holds throughout an open neighbourhood of $e=0$; compare to Definition \ref{defn space of curves}. \end{remark}

	\begin{defn} \label{defn topological GKC} \cite[Definition 4.1]{AMS-Hamiltonian}
		A \emph{topological global Kuranishi chart} is a tuple $(G',T',E',s')$
		where $G'$ is a compact Lie group, $T'$ is a topological $G'$-manifold with finite stabilizers
		and $E'$ is a $G'$-vector bundle with a $G'$ equivariant section $s'$.
		
	\end{defn}

	\begin{corollary} \label{cor globalKuranishichart}
		We have that $(G,\scrT',E',s')$ is a topological global Kuranishi chart
		where $\scrT' \subset \scrT$ is a neighborhood of $e=0$, 
	  $E' := E|_{\scrT'}$ and $s' := s|_{\scrT'}$. This chart is equivalent to a smooth global Kuranishi chart $\bT^{sm}$ for which the natural evaluation map $\bT^{sm} \to X^h \times \scrF$ is smooth $G$-equivariant map.
		Hence the natural evaluation map $\bT^{sm} \to X^h \times \ccMbar_{g,h}$ is smooth.
	\end{corollary}
	\begin{proof}
	It is explained in \cite[Section 4.1]{AMS-Hamiltonian} that given a topological global chart for which $T' \to \scrF$ admits a $G$-equivariant $C^1_{loc}$-topological submersion over a smooth $G$-manifold, one can apply abstract smoothing theory to an equivalent chart (obtained by stabilization and germ equivalence) to yield a smooth global chart.  By construction, the obstruction bundle is stably almost complex and $\scrT'$ carries a transverse almost complex structure.  The rest follows  from \cite[Proposition 4.30]{AMS-Hamiltonian}
		and \cite[Lemma 4.5]{AMS-Hamiltonian}.
	\end{proof}
	
	The smooth global chart $\bT^{sm}$ does not quite fit with the set-up of Section \ref{sec:counting-theories}, since $\scrT'$ is naturally stably almost complex rather than transverse complex, and the obstruction bundle $E'$ is not a complex vector bundle (because it carries a copy of the real vector space of Hermitian matrices); but this can be repaired by the trick from Lemma \ref{lem:subcategory}, adding a copy of the Lie algebra $\frak{g}$ of skew-Hermitian matrices to both $\scrT'$ and $E'$.
	
	\begin{remark}
		The global chart here differs, when $g=0$, from that constructed in \cite{AMS-Hamiltonian}. There, the framing data $F$ was used in the perturbation of the $\overline{\partial}$-equation defining the thickening. Here, the framing appears only in the section of the obstruction bundle.  In any formulation of the higher genus case, as here or in the slightly different approach in \cite{HS}, a key point is to address issues around the non-uniqueness of a holomorphic line bundle on $\Sigma$ with given degree, once $g(\Sigma) > 0$.
	\end{remark}

	We call a chart as constructed above a \emph{distinguished} global Kuranishi chart for the moduli space of stable maps. A distinguished chart is  associated to the data of $\bL$ and $(V_{\mu},\lambda_{\mu})_{\mu \in \bN}$ as well as a choice of natural numbers $k$ and $\mu$.
	
	\subsection{General Pre-Thickened Moduli Spaces} \label{sectiongeneralprethicken}

	Let $(X,\omega)$ be a closed symplectic manifold, let
	$\beta \in H_2(X;\Z)$. For the proof that the invariants we extract from the global Kuranishi charts from Section \ref{Subsec:easier} do not depend on choices, we shall need many variants of this basic construction. We introduce an abstract framework which covers all the cases that we need, so we start by fixing a compact Lie group $G$.

	\begin{defn} \label{defnfdapproxgeneral}
		For any $G_\C$-equivariant family of nodal curves $\scrC \to \scrF$ (Definition \ref{defn equivariant family})
		we let $Y_\scrF$
		be the $G$ vector bundle over $\scrC^o \times X$ whose fiber over a point $(c,x) \in \scrC^o \times X$ is the space of anti-holomorphic  maps from $T_c(\scrC^o|_{\pi(c)})$ to $T_x X$.
		
		A \emph{finite dimensional approximation scheme for $\scrF$}
		is a tuple $(V_{\mu},\lambda_{\mu})_{\mu \in \bN}$ where $(V_\mu)_{\mu \in \bN}$ are $G$-representations satisfying $V_{\mu-1} \subset V_\mu$
		and $\lambda_\mu : V_\mu \to C^\infty_c(Y_\scrF)$ are $G$-equivariant linear maps satisfying $\lambda_\mu|_{V_{\mu-1}} = \lambda_{\mu-1}$ for each $\mu \in \bN$, and  so that for each $\phi \in \scrF$, we have that
		$(V_\mu,\lambda_\mu|_{\scrC|_\phi})$ is a finite dimensional approximation scheme for $Y_\scrF|_{\scrC|_\phi}$
		as in Definition \ref{defn fd}.		
	\end{defn}
	
	\begin{remark}
	We could have chosen a finite dimensional approximation scheme for $C^\infty_c(Y_\scrF)$ as in Definition \ref{defn fd}.
	However, the  definition above has the advantage of being compatible with pullbacks, cf. Definition \ref{defnpulbackoffdscheme} below. \end{remark}

	\begin{defn} \label{defnprethickenedgeneral}
		The \emph{pre-thickened moduli space}, $\scrT^\pre =\scrT^\pre(\beta,\lambda_\mu)$
		is the space of tuples $(\phi,u,e)$ where
		$(\phi,u) \in \fC_\scrF$ (Definition \ref{defn space of curves}) and $e \in V_\mu$, satisfying the following equation:
		\begin{equation} 
			\overline{\partial}_J u|_{\scrC^o|_{\phi}} + (\lambda_\mu(e)) \circ \Gamma_u =0
		\end{equation}
		where
		\begin{equation}
			\Gamma_u : \scrC^o|_\phi \to \scrC^o \times X, \quad \Gamma_u(\sigma) := (\sigma,u(\sigma))
		\end{equation}				
		is the graph map, and  for which the stabilizer group $G_{(\phi,u,e)} \subset G$ of $(\phi,u,e)$ is finite.
		The topology is the induced Hausdorf metric topology on the closures of the images of the graphs $\Gamma_u$ in $\scrC\times X$, 
		combined with the natural linear topology on $V_\mu$.
		There is a natural $G$-action on this space sending $(\phi,u,e)$ to $(\phi \cdot g, u \cdot g, e \cdot g)$
		where $u \cdot g$ is the composition $\scrC|_{\phi \cdot g} \lra{\cdot g^{-1}} \scrC|_\phi \lra{u} X$.

		An element  $(\phi,u,e)$ is   \emph{regular} if
		the image of $\lambda_\mu$ surjects onto the cokernel of the linearized $\overline{\partial}$-operator
		associated to $u$.
		We define $\scrT^{\pre,\reg} \subset\scrT^\pre$ to be the subspace of  regular elements.
	\end{defn}

	\begin{lemma} \label{propcontainscompact}
		Let $K \subset \scrF$ be a compact subset.
		Then for all $\mu$ sufficiently large,
		we have that $\scrT^{\pre,\reg}$ contains the subspace
		\begin{equation}
			\scrT^\pre|_{K,e=0} := \{(\phi,u,0) \in\scrT^\pre \ : \ \phi \in K\} \subset\scrT^\pre.
		\end{equation}
	\end{lemma}
	\begin{proof}
		For each $J$-holomorphic curve $u : \scrC|_\phi \to X$,
		we have that $V_\mu$ surjects onto the cokernel of the linearized $\overline{\partial}$-operator
		for all $\mu \geq \mu_{\phi,u}$ for some $\mu_{\phi,u} \in \bN$.
		Hence it holds
		for all $(\phi',u')$ in a neighborhood $U_{\phi,u}$ of $(\phi,u)$ with respect to the Hausdorff topology on graphs $\Gamma_{u'}$.
		Using that $K$ is compact and Gromov compactness, we can cover $\scrT^\pre|_{K,e=0}$ by a finite number of open subsets $(U_{\phi_i,u_i})_{i \in I}$
		and hence we choose $\mu$ to be larger than $\max_{i \in I} \mu_{\phi_i,u_i}$.
	\end{proof}

	The following proposition follows from Corollary \ref{cor algfam}.
	\begin{lemma} \label{propregular}
		Let $\scrT^\pre$ be a pre-thickened moduli space as above.
		Then the map $\scrT^{\pre,\reg} \to \scrF$ admits a $C^1_{loc}$-structure.
		Also, if there exists another $G_\C$-equivariant family of nodal curves $\scrC' \to \scrF^{(2)}$ and a $G_\C$-equivariant submersion $\scrF \to \scrF^{(2)}$
		so that $\scrC$ is the pullback
		of $\scrC'$
		then the composite map
		$\scrT^{\pre,\reg} \to \scrF^{(2)}$ admits a $C^1_{loc}$-structure. \qed
	\end{lemma}

	\begin{example} 
		The pre-thickened moduli space from Definition \ref{defn prethickenedmodulispace}
		is equal to the pre-thickened moduli space from Definition \ref{defnprethickenedgeneral}
		with $\scrF = \scrF_{g,h,d}$ (Section \ref{Subsec:moduli}).
		The group $G$ here acting on this space is $U(d-g+1)$ with $G_\C = GL_{d-g+1}(\bC)$.
	\end{example}

	We also have parameterized versions of Definition \ref{defnprethickenedgeneral}	
	and Propositions \ref{propcontainscompact} and \ref{propregular}.
	These are needed to show that our Gromov-Witten invariants do not depend on the choice of almost complex structure.
	Let $(J_t)_{t \in [0,1]}$ be a smooth family of $\omega$-tame almost complex structures.
	(More generally, throughout the discussion below the interval $[0,1]$ can be replaced with any compact manifold with boundary.)
		
	\begin{defn} \label{defnprethickenedgeneralparameterized}
		The \emph{parameterised pre-thickened moduli space}, $\scrT^\pre_{(J_t)} =\scrT^\pre_{(J_t)}(\beta,\lambda_\mu)$
		is the space of tuples $(t,\phi,u,e)$ where $t \in [0,1]$,
		$(\phi,u) \in \fC_\scrF$ (Definition \ref{defn space of curves}) and $e \in V_\mu$, satisfying the following equation:
		\begin{equation} \label{eqn lambdapde parameterized}
			\overline{\partial}_{J_t} u|_{\scrC^o|_{\phi}} + (\lambda_\mu(e)) \circ \Gamma_u =0
		\end{equation}
		where
		\begin{equation}
			\Gamma_u : \scrC^o|_\phi \to \scrC^o \times X, \quad \Gamma_u(\sigma) := (\sigma,u(\sigma))
		\end{equation}				
		is the graph map, and for which the stabilizer group $G_{(\phi,u,e)} \subset G$ of $(\phi,u,e)$ is finite
		for each $(\phi,u,e) \in\scrT^\pre$.
		The topology is the induced Hausdorf metric topology on the images of the graphs $\Gamma_u$
		combined with the topology on $[0,1]$ and $V_\mu$.
		There is a natural $G$-action on this space sending $(t,\phi,u,e)$ to $(t,\phi \cdot g, u \cdot g, e \cdot g)$
		where $u \cdot g$ is the composition $\scrC|_{\phi \cdot g} \lra{\cdot g^{-1}} \scrC|_\phi \lra{u} X$.

		An element $(\phi,u,e)$ is \emph{regular} if
		the image of $\lambda_\mu$ surjects onto the cokernel of the linearized $\overline{\partial}$-operator
		associated to $u$.
		We define $\scrT^{\pre,\reg}_{(J_t)} \subset\scrT^\pre_{(J_t)}$ to be the subspace of regular elements.  	\end{defn}
	
              \begin{rem}
                We can use the weaker notion of regularity requiring only surjectivity onto the cokernel of the \emph{parametrised}  linearized $\overline{\partial}$-operator, i.e. the quotient of the kernel of the  $\overline{\partial}$-operator by the image of the tangent space of the parameter space. 
              \end{rem}

	Exactly as for Lemma \ref{propregular}, and again using Corollary \ref{cor algfam}, we get:

	\begin{lemma} \label{propcontainscompactproperparameterized}
		Let $K \subset \scrF$ be a compact subset.
		Then for all $\mu$ sufficiently large,
		we have that $T_{(J_t)}^{\pre,\reg}$ contains the subspace
		\begin{equation}
			\scrT^\pre_{(J_t)}|_{K,e=0} := \{(t,\phi,u,0) \in\scrT^\pre \ : \ \phi \in K\} \subset\scrT^\pre_{(J_t)}.
		\end{equation}
		Moreover,  the map $\scrT^{\pre,\reg}_{(J_t)} \to \scrF$ admits a $C^1_{loc}$-structure. \qed
	\end{lemma}

	\subsection{General Construction of Global Kuranishi Charts} \label{sectiongeneralconstruction}

	Let $(X,\omega)$ be a closed symplectic manifold, let
	$\beta \in H_2(X;\Z)$.
	
	\begin{defn} \label{defnliegroupexp}
		For any compact Lie group $G$,
		we let $\exp : i\mathfrak{g} \to G_\C/G$ be the exponential map in the imaginary direction.
	\end{defn}

	\begin{defn} \label{defnGbundle}
		Let $G$, $\check{G}$ be compact Lie groups.
		Let $B$ be a $\check{G}$-space. A \emph{$\check{G}$-equivariant principal $G$-bundle}
		is a principal $G$-bundle $P \to B$ together with a $\check{G}$-action on $P$ commuting with the $G$-action on $P$
		so that the map $P \to B$ is $\check{G}$-equivariant.

		Given a principal $G$-bundle $P$, we set $P_\C := P \times_G G_\C$ to be the corresponding principal $G_\C$ bundle.
		
		Let $U_G \subset i\frak{g}$ be a $G$-invariant neighborhood of $0$
		and $Q_G \subset G_\C$ a $G$-invariant neighborhood of $G$
		so that the map $\exp|_{U_G} : U_G \to Q_G / G$ is a diffeomorphism.
		We let
		\begin{equation}
			\log_P : P_\C \dashrightarrow i\mathfrak{g}
		\end{equation}
		be the partially defined map
		whose restriction to each fiber $P_\C|_x \cong G_\C$
		is the composition
		$Q_G \twoheadrightarrow Q_G / G \lra{\exp^{-1}} i \mathfrak{g}$.
		We call the bundle $P \times_G Q_G$ the \emph{domain} of $\log_P$.
	\end{defn}
	
	\begin{defn} \label{microbundles}
	Let $G$ be a compact Lie group.
		A \emph{smooth $G$-microbundle} over a smooth $G$-manifold $B$ is a triple
		$(E,i,p)$ where
		\begin{enumerate}
			\item $E$ is a smooth $G$-manifold,
			\item $i : B \to E$ is a smooth $G$-equivariant embedding and
			\item $p : E \to B$ is a smooth $G$-equivariant map satisfying $p \circ i = \id$.
		\end{enumerate}
		A \emph{microbundle isomorphism} $\rho : E \to E'$ between microbundles $(E,i,p)$ and $(E',i',p')$ over $B$
		is a smooth open embedding $\rho : U \to E'$ where $U \subset E$ is an open neighborhood of $i(B)$
		satisfying $\rho \circ i = i'$ and $p|_U = p' \circ \rho$.
	\end{defn}

	We will sometime write $E$ instead of $(E,i,p)$ if the context is clear.  
	An example of a $G$-microbundle is a smooth $G$-vector bundle. This is actually the universal example: since $p$ is a smooth submersion, any smooth microbundle has a vertical tangent bundle, to which it is isomorphic.   The microbundle language simplifies notation since it avoids picking an isomorphism from  a neighbourhood of $i(B)$ to such a normal bundle. We encountered a $G$-microbundle $\scrF^{(2)}$ in Section \ref{Subsec:moduli}, with its projection map $\Pi$ to, and diagonal map $\Delta$ from, the base $B = \scrF$. 		
	
	\begin{defn} \label{pre-Kuranishi data}	
		Let $G$ be a compact Lie group and $\scrC \to \scrF$ a $G_\C$-equivariant family of nodal curves (Definition \ref{defn equivariant family}).
		We define \emph{pre-Kuranishi data} over $\scrC \to \scrF$ to be a tuple $(G,E,\lambda_\mu)$
		where
		\begin{enumerate}
			\item $E = (E,i,p)$ is a smooth $G$-microbundle over $\scrF$,
			\item $\lambda_\mu$ arise from a finite dimensional approximation scheme $(V_\mu,\lambda_\mu)_{\mu \in \bN}$
			for $\scrF$ (Definition \ref{defnfdapproxgeneral}).
		\end{enumerate}		
	\end{defn}

	\begin{defn} \label{defnc1locsubmanifold}
		Let $p : R \to B$ be a topological submersion from a  topological manifold to a smooth manifold, which admits a $G$-equivariant $C^1_{loc}$-structure.
		A \emph{$C^1_{loc}$-submanifold} of $R$ is a topological submanifold $R' \subset R$
		so that for each $x \in R'$, letting $b=p(x)$, there is a product neighborhood $\iota : W \stackrel{\sim}{\longrightarrow} W|_b \times p(W)$ for the total space whose restriction 
		$\iota|_{R'} : W|_{R'} \to (W|_{R'})|_b \times p(W)$ is a product neighborhood of $R'$.
		(\cite[Definition 4.27]{AMS-Hamiltonian}).
	\end{defn}
	
	In Definition \ref{defnc1locsubmanifold} note that the image $p(R') \subset B$ is open; there is a natural generalisation to $C^1_{loc}$-submanifolds with boundary, but again we are only concerned with the case when their image in $B$ has codimension zero.

	\begin{defn} \label{defnnoncomplexthickening}
		Let $(G,E,\lambda_\mu)$ be pre-Kuranishi data over some $G_\C$-nodal family $\scrC \to \scrF$ with $E = (E,i,p)$.
		We define
		\emph{global chart data associated to $(G,E,\lambda_\mu)$}
		to be a tuple $(\scrT,s)$
		where
		\begin{enumerate}
			\item $\scrT = P_\C$ for some principal $G$-bundle $P$ over a closed $C^1_{loc}$ submanifold $B_\scrT$ of $\scrT^\pre$ possibly with boundary where $\scrT^\pre :=\scrT^\pre(\beta,\lambda_\mu)$,
			\item equipping $\scrT$ with the diagonal $G$-action (inside its natural $G\times G$-action), $s$ is a $G$-equivariant continuous map $\scrT \to E$ with $s^{-1}(0)$ compact and contained in $\scrT^{\pre,\reg}$ (Definition \ref{defnprethickenedgeneral})
			and so that $P_\scrF = p \circ s$ where $P_\scrF : \scrT \to \scrF$ and $p : E \to \scrF$ are the natural projection maps.
			Equivalently, $s$ is a $G$-equivariant section of the pullback of $E$ to $\scrT$.
		\end{enumerate}	

\[\begin{tikzcd}
	& {\scrT=P_{\bC}} && E \\
	{s^{-1}(0)} & {B_\scrT} & {\scrT^{\pre}} & \scrF \\
	& {\scrT^{\pre,\reg}}
	\arrow[from=1-2, to=2-2]
	\arrow[hook, from=2-2, to=2-3]
	\arrow[hook, from=3-2, to=2-3]
	\arrow["{(\phi,u) \to \phi}"', from=2-3, to=2-4]
	\arrow["p", from=1-4, to=2-4]
	\arrow["s"{description}, from=1-2, to=1-4]
	\arrow[hook, from=2-1, to=2-2]
	\arrow[hook, from=2-1, to=3-2]
	\arrow["{P_\scrF}", from=1-2, to=2-4]
\end{tikzcd}\]

		Given global chart data $(\scrT,s)$ as above, and an element $(\phi,u,e) \in\scrT^\pre$,
		a \emph{$\scrT$-framing} of $(\phi,u,e)$ is an element $F \in \scrT|_{(\phi,u,e)}$.
		
		Let $\widetilde{N}_\scrF$ be the pullback of the  normal bundle $N_\scrF$ of $i(\scrF)$ (Definition \ref{pre-Kuranishi data})
		to $\scrT$ via the projection map $\scrT \to \scrF$.
		Let $Q_\scrT \subset P_\C$ be the domain of $\log_P$ (Definition \ref{defnGbundle}).
		Choose an isomorphism of microbundles $\rho : N_\scrF \to E$
		where $N_\scrF$ is the normal bundle of $i(\scrF)$ inside $E$.		
		Let $U \subset E$ be the image of $\rho$.
		The \emph{global Kuranishi chart associated to $(\scrT,s)$}
		is the tuple 
		\begin{equation}
			(G,s^{-1}(U) \cap Q_\scrT, i\mathfrak{g} \oplus \widetilde{N}_\scrF \oplus V_\mu, \log_P \oplus (\rho^{-1} \circ s) \oplus P_\mu)
		\end{equation}
		where $P_\mu : \scrT \to V_\mu$ is the natural projection map.
	\end{defn}

	The following proposition is a consequence of Lemma \ref{propregular} combined with the fact that a 
	$G$-equivariant principal $G$-bundle whose base admits a $C^1_{loc}$ structure over a smooth $G$-manifold also admits a $C^1_{loc}$ structure
	over the same smooth $G$-manifold.
	\begin{prop} \label{propchart}
		The global Kuranishi chart associated to $(\scrT,s)$ is a topological global Kuranishi chart
		whose base admits a $C^1_{loc}$ structure over $B_\scrT \subset \scrF$.
		If $\underline{\scrC} \to \underline{\scrF}$ is another $G_\C$-equivariant family of nodal curves and $\rho : \scrF \to \underline{\scrF}$
		is a smooth $G_\C$-equivariant morphism which is a submersion so that $\scrC$ is the pullback of $\underline{\scrC}$,
		then $\scrT$ admits a $C^1_{loc}$ structure over $\underline{\scrF}$. \qed
	\end{prop}
	
	The global Kuranishi chart does not depend on the choice of $\rho$ in the sense that
	if we choose a different $\rho$, we get an isomorphic obstruction bundle over the same thickening and this isomorphism
	sends the section of one global Kuranishi chart to the other.
	The resulting global Kuranishi chart in the proposition above is not necessarily for $\ccMbar_{g,h}(X,\beta,J)$. It could be for a subspace of $\ccMbar_{g,h}(X,\beta,J)$ or a fiber product of such space (See, for example, Sections \ref{sec:split_curves} and \ref{sec:genusreduction}).

	\begin{Example} \label{examplegl}
		Let $\scrF = \scrF_{g,h,d}$, $\scrC = \scrC_{g,h,d}$,
		$\scrF^{(2)} = \scrF^{(2)}_{g,h,d}$, and $\Delta$ be as in Section \ref{Subsec:moduli}
		with $G = U(d-g+1)$.
		Let $P_\scrF : \scrF^{(2)} \to \scrF$ be the natural projection map.
		Then $\scrF^{(2)} = (\scrF^{(2)},\Delta,P_\scrF)$ is a smooth $G$-microbundle over $\scrF$.
		Choose a finite dimensional approximation scheme  $(V_\mu,\lambda_\mu)_{\mu \in \bN}$ for $\scrF$.
		Then $(G,\scrF^{(2)},\lambda_\mu)$ is pre-Kuranishi data over $\scrC$.

		The principal bundle $P \to\scrT^\pre$ has fiber over $(\phi,u,e) \in\scrT^\pre$ equal to the space of unitary $\bL$-framings on $(\phi,u,e)$ (Definition \ref{defn Lframing}).
		Then $\scrT = P_\bC$. The space of framed curves from Definition \ref{defn thickened} is an open subset of $\scrT$ containing $P$.
		We let $s : \scrT \to \scrF^{(2)}$ be the map sending $(\phi,u,e,F)$ to $(\phi,\phi_F)$ as in Definition \ref{defn thickened}.
		Then $(\scrT,s)$ is global chart data associated to $(G,\scrF^{(2)},\lambda_\mu)$.
		The corresponding global Kuranishi chart is the one described in Definition \ref{defn Kurnishi chart}.
		The map $\widetilde{\Delta}$ from Equation \eqref{eqn diagonal neighborhood} identifies the microbundle $\scrF^{(2)}$ with the normal bundle of $\scrF$ inside $\scrF^{(2)}$.	
	\end{Example}
	
	\subsection{Parametrized families and cobordism invariance\label{Subsec:families}}
	
	We also have parameterized versions of Definition \ref{defnnoncomplexthickening}
	and Proposition \ref{propchart}.
	These are needed to show that our Gromov-Witten invariants do not depend on the choice of almost complex structure.
	Let $(J_t)_{t \in [0,1]}$ be a smooth family of $\omega$-tame almost complex structures.
	More generally, the interval $[0,1]$ can be replaced with any compact manifold with boundary.

	\begin{defn} \label{defnnoncomplexthickeningparameterized}
		Let $(G,E,\lambda_\mu)$ be pre-Kuranishi data over some $G_\C$-nodal family $\scrC \to \scrF$ with $E = (E,i,p)$.
		We define the
		\emph{parameterized global chart data associated to $(G,E,\lambda_\mu)$}
		to be a tuple $(\scrT',s')$
		where
		\begin{enumerate}
			\item $\scrT' = P'_\C$ for some principal $G$-bundle $P'$ over a closed $C^1_{loc}$ submanifold $B_\scrT$ of $\scrT^\pre_{(J_t)}$ possibly with boundary where $\scrT^\pre_{(J_t)} :=\scrT^\pre_{(J_t)}(\beta,\lambda_\mu)$ (Definition \ref{defnprethickenedgeneralparameterized}),
			\item $s'$ is a $G$-equivariant continuous map $\scrT' \to E$ with $s^{-1}(0)$ compact and contained in $\scrT^{\pre,\reg}_{(J_t)}$ (Definition \ref{defnprethickenedgeneralparameterized})
			and so that $P'_\scrF = p \circ s'$ where $P'_\scrF : \scrT \to \scrF$ and $p : E \to \scrF$ are the natural projection maps.
			Equivalently, $s'$ is a $G$-equivariant section of the pullback of $E$ to $\scrT'$.
		\end{enumerate}	
		Given parameterized global chart data $(\scrT',s')$ as above, and an element $(t,\phi,u,e) \in\scrT^\pre_{(J_t)}$,
		a \emph{$\scrT'$-framing} of $(t,\phi,u,e)$ is an element $F \in \scrT'|_{(t,\phi,u,e)}$.
		
		Let $\widetilde{N}_\scrF$ be the pullback of the  normal bundle $N_\scrF$ of $i(\scrF)$ (Definition \ref{pre-Kuranishi data})
		to $\scrT'$ via the projection map $\scrT \to \scrF$.
		Let $Q_{\scrT'} \subset P'_\C$ be the domain of $\log_{P'}$ (Definition \ref{defnGbundle}).
		Choose an isomorphism of microbundles $\rho : N_\scrF \to E$
		where $N_\scrF$ is the normal bundle of $i(\scrF)$ inside $E$.		
		Let $U \subset E$ be the image of $\rho$.
		The \emph{global Kuranishi chart associated to $(\scrT,s)$}
		is the tuple 
		\begin{equation}
			(G,{s'}^{-1}(U) \cap Q_{\scrT'}, i\mathfrak{g} \oplus \widetilde{N}_\scrF \oplus V_\mu, \log_{P'} \oplus (\rho^{-1} \circ s') \oplus P'_\mu)
		\end{equation}
		where $P'_\mu : \scrT' \to V_\mu$ is the natural projection map.
	\end{defn}

	The following proposition is a consequence of Proposition \ref{propregular} combined with the fact that
	$G$-equivariant principal $G$-bundles whose base admits a $C^1_{loc}$ structure over a smooth $G$-manifold also admit a $C^1_{loc}$ structure
	over the same smooth $G$-manifold.
	\begin{prop} \label{propchartparameterized}
		The global Kuranishi chart associated to $(\scrT',s')$ is a topological global Kuranishi chart
		whose base admits a $C^1_{loc}$ structure over $B_{\scrT'} \subset \scrF$.
	\end{prop}

To deduce from this the invariance of smooth global charts, we include a brief discussion on uniqueness of smoothings.  Let $M \to B$ be a topological submersion of a topological manifold over a smooth manifold, with a $C^1_{loc}$-structure. The entire discussion below is meant to be $G$-equivariant, but we suppress further discussion of the group action. In \cite{AMS-Hamiltonian} the construction of a smooth chart from a $C^1_{loc}$ chart relies on lifting the tangent microbundle of $M$ to a vector bundle. Lashof's theory \cite{Lashof} implies that, perhaps after further stabilisation (by a $G$-representation), the smoothing is uniquely determined up to diffeomorphism by the stable isomorphism class of the vector bundle lift.  The bundle lift constructed in \cite{AMS-Hamiltonian} in turn depends on two pieces of data: the $C^1_{loc}$-structure on $M \to B$, and the choice of a fibrewise submersion in the sense of \cite[Definition 4.22]{AMS-Hamiltonian} (which is a map from an open neighbourhood of the diagonal in $M \times M$ to $M$ satisfying a raft of conditions).  There is a natural notion of stabilising a fibrewise submersion: one for $M \to B$ induces one on $M\times \bR^m \to B \times \bR^m$. It follows from the extension result of Bai and Xu \cite[Lemma 6.21]{BaiXu}  that any $S^k$-parametrized family of fibrewise submersions $M \to B$ extends over $\bR^{k+1}$ after stabilising $M$ and $B$ by $\bR^{k+1}$ in this way, so the space of fibrewise submersions up to stabilisation is contractible. 

\begin{remark}
In the case of topological manifolds with boundary, one requires a relative version of the smoothing theory to ensure that the smooth global chart obtained above is well-defined under the relation of cobordism.  Such a relative smoothing theory is carefully developed in \cite[Appendix B]{BaiXu2}, see also  \cite[Section 2]{Rezchikov} for a reduction to the case of $\bZ/2$-equivariant smoothings.
\end{remark}

It follows that the stable vector bundle lift of the tangent microbundle is completely determined by the $C^1_{loc}$-structure.  Since this is controlled by Proposition \ref{propchartparameterized}, the virtual class associated to a stable smoothing of a $C^1_{loc}$-global chart is indeed invariant under cobordism.

	The abstract smoothing theory gives rise to a global chart $(G,\scrT,E,s)$ with $\scrT$ and $E$ smooth but with the obstruction section $s$ still only continuous.  One can however replace it by a smooth section, as follows:
	
	\begin{lemma} \label{lem:mollify}
	Let $(G,\scrT,E,s)$ be a topological global Kuranishi chart over a base $\scrF$ with $\scrT \to \scrF$ admitting a $C^1_{loc}$-structure.  Then there is a stable smoothing $(G,\scrT', E',s')$ with the section $s'$ smooth and with $(s')^{-1}(0) = s^{-1}(0)$.
	\end{lemma}
	
	\begin{proof}
	The previous argument constructs a stable smoothing $\scrT''$ of $\scrT$ such that there is still a $C^1_{loc}$-topological submersion $\scrT'' \to \scrF$, and such that the obstruction section $s''$ in the smoothed chart is still $C^1_{loc}$ and has $(s'')^{-1}(0) = s^{-1}(0) = Z$, i.e. which has the same zero-set as $s$. Fix a smooth function $\rho: \scrT'' \to \bR$ which vanishes along $Z$ and for which all derivatives of $\rho$ vanish along $Z$. Mollify $s''$ away from $Z$ to obtain a section $s'$ which is smooth on the complement of $Z$ and globally continuous.  Since $s'$ can be taken arbitrarily $C^1_{loc}$-close to $s''$ along $Z$, the section $\rho s'$ is then globally smooth.
		\end{proof}
	
	The smoothing construction depends only on contractible choices, so yields a smooth global chart well-defined up to equivalence.

	\subsection{Equivalences Between  Kuranishi Data} \label{sectionequivalences}

	We will be constructing global Kuranishi charts using various global Kuranishi chart data
	as in Definition \ref{defnnoncomplexthickening}.
	We next describe some operations  on such charts which yield equivalences,  e.g. which then induce  the stabilisation, germ equivalence or group enlargement operations  from \cite{AMS-Hamiltonian} (which were shown to be equivalences in \emph{op. cit.}, and which are $C^1_{loc}$-analogues of the moves from Section \ref{sec:counting-theories}, which they induce on smoothings). As usual we fix a closed symplectic manifold $(X,\omega)$ and
	$\beta \in H_2(X;\Z)$.

	\begin{defn} \label{defnpulbackoffdscheme}
		Let $\scrC \to \scrF$ be a $G_\C$-equivariant family of nodal curves and let $f : \scrF^{(2)} \to \scrF$ be a smooth $G_\C$-equivariant morphism
		so that $f^*\scrC \to \scrF^{(2)}$ is the pullback of $\scrC$.
		Let $(V_\mu,\lambda_\mu)_{\mu \in \bN}$ be a finite dimensional approximation scheme for $\scrF$ (Definition \ref{defnfdapproxgeneral}).
		Then the \emph{pullback} $f^*\lambda_\mu : V_\mu \to \Omega^{0,1}(\scrC'/\scrF^{(2)};TX)$ is uniquely defined by the following property:
		$f^*\lambda_\mu(e)$
		restricted to $(f^*\scrC)|_{\phi'}$ is equal to $\lambda_\mu(e)$ restricted to the isomorphic curve $\scrC|_\phi$, where $\phi$ is the image of $\phi'$, for each $e \in V_\mu$ and $\phi' \in \scrF^{(2)}$.
		We call $(V_\mu,f^*\lambda_\mu)_{\mu \in \bN}$ the \emph{pullback} of $(V_\mu,\lambda_\mu)_{\mu \in \bN}$ along $\scrF^{(2)} \to \scrF$.		
	\end{defn}	
	
	Note that the pullback of a finite dimensional approximation scheme
	is a finite dimensional approximation scheme.

	\begin{defn} \label{defngroupenlargement}
		Let $(G,E,\lambda_\mu)$
		be pre-Kuranishi data associated to a $G_\C$-equivariant family of nodal curves $\scrC \to \scrF$.
		A \emph{principal bundle enlargment} of $(G,E,\lambda_\mu)$
		consists of the pre-Kuranishi data
		$(G \times \check{G},\check{E},\check{\lambda}_\mu)$
		over a $G_\C$-nodal family $\check{\scrC} \to \check{\scrF}$
		where
		\begin{enumerate}
			\item there exists a $G_\C$-equivariant principal $\check{G}$-bundle $P \to \scrF$
			so that $\check{\scrF} = P_\C$ and $\check{\scrC}$ is the pullback of $\scrC$,
			\item $\check{E}$ is the pullback of $E$ to $P_\C$ and
			\item $\check{\lambda}_\mu$ is the pullback of $\lambda_\mu$ along $\check{\scrF} \to \scrF$.
		\end{enumerate}
		
		Now let $(\scrT,s)$ be global chart data associated to $(G,E,\lambda_\mu)$.
		An \emph{enlargement} of $(\scrT,s)$ consists of global Kuranishi data
		$(\check{\scrT},\check{s})$ for a principal bundle enlargement $(G \times \check{G}, \check{E},\lambda_\mu)$
		as above
		where
		\begin{enumerate}
			\item $\check{\scrT}$ is a $G$-equivariant principal $\check{G}$-bundle over $\scrT$,
			\item $\check{s}$ composed with the projection map to $E$ is equal to the projection map to $\scrT$ composed with $s$.
		\end{enumerate}
	\end{defn}

	\begin{lemma}  \label{lemmagroupenlargement}
		Let $(\scrT,s)$ and $(\check{\scrT},\check{s})$ be as in Definition \ref{defngroupenlargement} above.
		The global Kuranishi chart associated to $(\check{\scrT},\check{s})$ is obtained from the one associated to $(\scrT,s)$
		via group enlargement (in the sense of \cite{AMS-Hamiltonian} or Section \ref{sec:counting-theories}). \qed
	\end{lemma}
	
	\begin{defn} \label{defnapproxenlargement}
		Let $(G,E,\lambda_\mu)$ be pre-Kuranishi data over a $G_\C$-nodal family $\scrC \to \scrF$ as in Definition \ref{defnprethickenedgeneral}.
		Let $(\check{V}_\mu,\check{\lambda}_\mu)_{\mu \in \bZ}$ be another finite dimensional approximation scheme for $\scrF$ (Definition \ref{defnfdapproxgeneral}).
		
		Let $(\scrT,s)$ be global Kuranishi data associated to $(G,E,\rho,\lambda_\mu)$ and $(\scrT',s')$ global Kuranishi data associated to $(G,E,\lambda_\mu \oplus \check{\lambda}_\mu)$. 
		We say that $(\scrT',s')$ is an \emph{approximation stabilization of $(\scrT,s)$} if the restriction of $(\scrT',s')$ to $\scrT^\pre(\beta,\lambda_\mu)$
		is $(\scrT,s)$.
	\end{defn}

	\begin{lemma}
		If $(\scrT',s')$ is an approximation stabilization of $(\scrT,s)$ as above then 
		the global Kuranishi chart associated to $(\scrT',s')$ is obtained from the one associated to $(\scrT,s)$ via stabilization. \qed
	\end{lemma}

	\begin{defn} \label{defnc1locmap}
		Let $p : M \to F$ admit a $C^1_{loc}$ structure. We define
		$T(M/F)$ to be the corresponding vertical tangent bundle.	
		
		Let $p_0 : M_0 \to F$, $p_1 : M_1 \to F$ admit $C^1_{loc}$ structures
		over a smooth manifold $F$.
		A \emph{fiberwise $C^1_{loc}$ map} $f : M_0 \to M_1$ is a continuous map where
		\begin{enumerate}
			\item $p_0 = p_1 \circ f$,
			\item the restriction $f|_\phi : M_0|_\phi \to M_1|_{\phi}$ for each $\phi \in F$ is smooth and
			\item there is a continuous map $D^{ver}f : T(M_0/F_0) \to T(M_1/F_1)$ whose restriction to $T(M_0|_\phi)$ is the derivative of $f|_{\phi}$ for each $\phi \in F$.
		\end{enumerate}
			\end{defn}

For the next definition, let $(G,E,\lambda_\mu)$ be pre-Kuranishi data over a $G_\C$-nodal family $\scrC \to \scrF$ as in Definition \ref{defnprethickenedgeneral}.
		Let $Q : \check{\scrF} \to \scrF$ be a smooth surjective $G_\C$-equivariant morphism
		and let $\check{\scrC} \to \check{\scrF}$ be the corresponding pullback by $Q$.
        
	\begin{defn} 
		We say that pre-Kuranishi data $(G,\check{E},\check{\lambda}_\mu)$ over $\check{\scrC} \to \check{\scrF}$ is a \emph{$Q$-lift}
		of $(G,E,\lambda_\mu)$
		if
		\begin{enumerate}
			\item There is a map $\widetilde{Q}$  fitting into the following commutative diagram:
						\[\begin{tikzcd}
				{\check{E}} & E \\
				{\check{\scrF}} & \scrF
				\arrow[from=1-1, to=2-1]
				\arrow[from=1-2, to=2-2]
				\arrow["Q", from=2-1, to=2-2]
				\arrow["{\widetilde{Q}}", from=1-1, to=1-2]
			\end{tikzcd}\]
			where the vertical maps are the natural projection maps and
			
			\item  $\check{\lambda}_\mu$ is the pullback of $\lambda_\mu$ via $Q$.
		\end{enumerate}
		
		\end{defn}
		We now explain how to relate global Kuranishi data in this setting: let $(\scrT,s)$ be global Kuranishi data for $(G,E,\lambda_\mu)$, and, to simplify the notation, let $\check{\scrT}^\pre :=\scrT^\pre(\beta,\check{\lambda}_{\mu})$ and $\scrT^\pre :=\scrT^\pre(\beta,\lambda_\mu)$.
                \begin{defn}        \label{defngenearlstablization}             
                A \emph{general stabilization} of $(\scrT,s)$ consists of global Kuranishi data $(\check{\scrT},\check{s})$ for a $Q$-lift $(G,\check{E},\check{\lambda}_\mu)$ as above
		so that
		\begin{enumerate}
			\item there exists a map $\widetilde{P}$ fitting into the following commutative diagram:
			
						\[\begin{tikzcd}
				{\check{E}} & E \\
				{\check{\scrT}} & \scrT \\
				{\check{\scrT}^\pre} & {\scrT^\pre} \\
				{\check{\scrF}} & \scrF
				\arrow[from=2-2, to=3-2]
				\arrow[from=2-1, to=3-1]
				\arrow[from=3-1, to=4-1]
				\arrow[from=3-2, to=4-2]
				\arrow["Q", from=4-1, to=4-2]
				\arrow["{\widetilde{Q}}", from=1-1, to=1-2]
				\arrow[bend right=40, from=1-1, to=3-1]
				\arrow[bend left=40, from=1-2, to=3-2]
				\arrow["{\check{s}}", from=2-1, to=1-1]
				\arrow["s", from=2-2, to=1-2]
				\arrow[from=3-1, to=3-2]
				\arrow["{\widetilde{P}}", from=2-1, to=2-2]
			\end{tikzcd}\]		
						where  unlabelled arrows are the natural projection maps, all spaces in the diagram admit $C^1_{loc}$ structures over $\scrF$ and where all the morphisms  are $C^1_{loc}$ maps over $\scrF$.
			\item all the maps in the diagram are $G$-equivariant (cf. the action from Definition \ref{defnnoncomplexthickening}) and $\widetilde{P}$ maps $\check{s}^{-1}(0)$ homeomorphically to $s^{-1}(0)$ (respecting stabiliser groups).
			\item For each $x \in \check{s}^{-1}(0)$, $D^{\ver}\check{s} : T(\check{\scrT}/\scrT)|_x \to T(\check{E}/E)|_{\check{s}(x)}$ is an isomorphism.
					\end{enumerate}				
			\end{defn}

	\begin{prop} \label{propstabilization}
		Let $(\check{\scrT},\check{s})$ be a general stabilisation of $(\scrT,s)$ as in Definition \ref{defngenearlstablization}.
		Then the global Kuranishi charts associated to $(\scrT,s)$ and to $(\check{\scrT},\check{s})$
		are related by stabilization and germ equivalence.
	\end{prop}
	\begin{proof}
		First of all, we can replace $E$ and $\check{E}$ with any isomorphic smooth $G$-microbundles.
		As  a result, we can think of them as vector bundles over $\check{\scrT}^\pre$ and $\scrT^\pre$ respectively.
		The pullback of these vector bundles $\check{E}'$ and $E'$ to $\check{\scrT}$ and $\scrT$ respectively
		have natural sections $\check{s}'$ and $s'$ induced by $\check{s}$ and $s$ respectively.
		We can also assume that $\widetilde{Q}$ is a linear submersion and hence $\check{E}'$ is isomorphic to $E^\perp \oplus E''$
		where $E''$ is the pullback $\widetilde{P}^*E$.
		Then $\check{s}' = s^\perp \oplus s''$ with respect to this splitting where $s''$ is the pullback of $s'$.
		Also $s^{\perp}$ is transverse to zero along $\check{s}^{-1}(0)$.
			\end{proof}

	\subsection{Independence of Choices} \label{subsection independence}

	We need to show the global Kuranishi
	charts constructed in
	Section \ref{Subsec:easier} are independent of choices of finite dimensional approximation scheme,
	consistent domain metrics and Hermitian line bundles, etc.
	We will do this by using global Kuranishi chart data from Section \ref{sectiongeneralconstruction}.
	Dependence on $J$ was already dealt with in Proposition \ref{propchartparameterized} above.
	
	\begin{remark}
	The general template for building equivalences follows an idea from \cite{AMS-Hamiltonian}; given two charts associated to two choices of data, we find group enlargements of each which are simultaneously stabilisation-equivalent to a global chart of `doubly framed' curves (in which both sets of choices are carried through the construction at once).  Equivalences therefore appear as zig-zags in such a diagram. This template will recur several times in the coming sections. \end{remark}
	
	Let $g,h,d,\check{d} \in \bN$.
	Let $\scrF = \scrF_{g,h,d}$,
	 $\scrC = \scrC_{g,h,d}$, 
	$\scrF^{(2)} = \scrF^{(2)}_{g,h,d}$,
	$\check{\scrF} := \scrF_{g,h,\check{d}}$,
	 $\check{\scrC} := \scrC_{g,h,\check{d}}$ and
	$\check{\scrF}^{(2)} := \scrF^{(2)}_{g,h,\check{d}}$
	be as Section \ref{Subsec:moduli}.
	We let $G := U(d-g+1)$ and $\check{G} := U(\check{d}-g +1)$.
		Choose finite dimensional approximation schemes
	$(V_\mu,\lambda_\mu)_{\mu \in \bN}$
	and $(\check{V}_\mu,\check{\lambda}_\mu)_{\mu \in \bN}$
	for $\scrF$ and $\check{\scrF}$ respectively.
	
	We have associated pre-thickenings
	$\scrT^\pre :=\scrT^\pre(\beta,\lambda_\mu)$
	and $\check{\scrT}^\pre :=\scrT^\pre(\beta,\check{\lambda}_\mu)$
	for some large $\mu$.
	
	Now choose consistent domain metrics for $\scrF$ and $\check{\scrF}$
	as in Definition \ref{defn domain metric}
	and Hermitian line bundles $L \to X$ and $\check{L} \to X$
	whose curvatures tame $J$, 
	and fix an integer $2 \leq k $.
	
	We have a principal $G$-bundle $P \to\scrT^\pre$
	whose fiber over $(\phi,u,e) \in\scrT^\pre$
	is a unitary basis of
	$L_u := u^*(\omega_{\scrC/\scrF}(p_1,\cdots,p_h)|_\phi \otimes u^* L)^{\otimes k}$
	where $p_1,\cdots,p_h$ are the marked point sections.
	We let $d$ be the size of this basis.
	We have a similarly defined principal $G$-bundle $\check{P} \to \check{\scrT}^\pre$
	whose fiber over $(\phi,u,e) \in \check{\scrT}^\pre$
	is a unitary basis of
	$\check{L}_u := u^*(\omega_{\check{\scrC}/\check{\scrF}}(\check{p}_1,\cdots,\check{p}_h)|_\phi \otimes u^* \check{L})^{\otimes k}$
	where $\check{p}_1,\cdots,\check{p}_h$ are the marked point sections.
	We let $\check{d}$ be the size of this basis.
	Define $\scrT := P_\bC$ and $\check{\scrT} := \check{P}_\bC$.
	
	Each unitary basis $F$ gives us a map $\phi_F$ from the domain of $u$
	to $\bC \bP^{d-g}$ and hence an element $(\phi,\phi_F)$ of $\scrF^{(2)}$.
	Therefore we have a $G$-equivariant map
	$P \to \scrF^{(2)}$ which extends to a $G_\C$-equivariant map $s : \scrT \to \scrF^{(2)}$.
	We define $\check{s} : \check{\scrT} \to \check{\scrF}^{(2)}$ analogously. 
	Then $(G,\scrT,s)$ is global Kuranishi data associated to $(G,\scrF^{(2)},\lambda_\mu)$
	and $(\check{G},\check{\scrT},\check{s})$ is global Kuranishi data associated to $(\check{G},\check{\scrF}^{(2)},\check{\lambda}_\mu)$.

	\begin{prop} \label{proprelated}
		$(G,\scrT,s)$ and $(\check{G},\check{\scrT},\check{s})$
		are related by a sequence of
		principal bundle enlargements (Definition \ref{defngroupenlargement})
		and general stabilizations (Definition \ref{defngenearlstablization}) and their inverses.
		Hence their associated global Kuranishi charts are equivalent.
	\end{prop}
	\begin{proof}
	This is done via three intermediate global Kuranishi data.
	The first is constructed as follows:
	We start with pre-Kuranishi data $(G \times \check{G},\scrF^{(2)},\lambda_\mu)$
	where $G \times \check{G}$ acts by projecting to $G$ first.
	We let $\scrT_1$ be the principal $\check{G}$-bundle
	over $\scrT$ whose fiber over a point in $\scrT|_{(\phi,u,e)}$, 
	$(\phi,u,e) \in\scrT^\pre$
	is a unitary basis $\check{F}$
	of $H^0(\check{L}_u)$ using the consistent
	domain metric for $\check{\scrF}$.
	We let $s_1$
	be the composition $\scrT_1 \to \scrT \lra{s} \scrF^{(2)}.$
	Then $(\scrT_1,s_1)$ is global Kuranishi data associated to
	 $(G \times \check{G},\scrF^{(2)},\lambda_\mu)$.
	 It is also a principal bundle enlargement of $(\scrT,s)$ (Definition \ref{defngroupenlargement}).

	 We will now construct the second global Kuranishi chart data.
	 We start with the space $\scrF_2$
	 of genus $g$ curves with $h$
	 marked points mapping to $\bC \bP^d \times \bC \bP^{\check{d}}$
	 whose projection to $\bC \bP^d$
	 is a member of $\scrF$
	 and whose projection to $\bC \bP^{\check{d}}$
	 is in $\check{\scrF}$.
	 We let $\scrC_2 \to \scrF_2$
	 be the corresponding universal curve
	 with marked point sections $p^2_1,\cdots,p^2_h$.
	 This is a $G_\bC \times \check{G}_\bC$-nodal family.
	 
	 We let $\scrF_2^{(2)}$ be the space
	 of genus $g$ curves with $h$
	 marked points mapping to $(\bC \bP^d)^2 \times (\bC \bP^{\check{d}})^2$
	 whose projection to $(\bC \bP^d)^2$
	 is a member of $\scrF^{(2)}$
	 and whose projection to $(\bC \bP^{\check{d}})^2$
	 is in $\check{\scrF}^{(2)}$.

	 We let $(V_{\mu,2},\lambda_{\mu,2})_{\mu \in \bN}$
	 be the finite dimensional approximation scheme
	 for $\scrF_2$
	 given by the sum of the pullback (\ref{defnpulbackoffdscheme}) of
	 $(V_\mu,\lambda_\mu)_{\mu \in \bN}$
	 and $(\check{V}_\mu,\check{\lambda}_\mu)_{\mu \in \bN}$
	 respectively.
	 
	 We have an associated pre-thickening
	$\scrT^\pre_2 :=\scrT^\pre(\beta,\lambda_{\mu,2})$
	for some large $\mu$.
	
	 We have a principal $G \times \check{G}$-bundle $P_2 \to\scrT^\pre_2$
	whose fiber over $(\phi,u,e) \in\scrT^\pre_2$ is 
	the space of pairs $F,\check{F}$
	of unitary bases of $H^0(L_u)$
	and $H^0(\check{L}_u)$
	respectively.
	So elements of $\scrT^\pre_2$ are tuples of the form $(\phi,u,e,F,\check{F})$.
	We also have a map
	$s_2 :\scrT^\pre_2 \to \scrF_2^{(2)}$
	sending $(\phi,u,e,F,\check{F})$ to $(\phi,\phi_F,\check{\phi},\phi_{\check{F}})$.

	 Then $(\scrT_2,s_2)$ is global Kuranishi data associated to
	 $(G \times \check{G},\scrF_2,\lambda_{\mu,2})$.
	 Also $(\scrT_2,s_2)$
	 is a general stabilization (Definition \ref{defngenearlstablization})
	 followed by an approximation stabilization (Definition \ref{defnapproxenlargement})
	 of $(\scrT_1,s_1).$

	 Finally we construct the third global Kuranishi data.
	 This is the same as the first one,
	 but with $\scrF$ and $\check{\scrF}$
	 swapped.
	 We start with pre-Kuranishi data $(G \times \check{G},\check{\scrF}^{(2)},\check{\lambda}_\mu)$
	where $G \times \check{G}$ acts by projecting to $\check{G}$ first.
	We let $\scrT_3$ be the principal $G$-bundle
	over $\check{\scrT}$ whose fiber over a point in $\check{\scrT}|_{(\phi,u,e)}$, 
	$(\phi,u,e) \in \check{\scrT}^\pre$
	is a unitary basis $F$
	of $H^0(L_u)$ using the consistent
	domain metric for $\scrF$.
	We let $s_3$
	be the composition $\scrT_3 \to \check{\scrT} \lra{\check{s}} \check{\scrF}.$
	Then $(\scrT_3,s_3)$ is global Kuranishi data associated to
	 $(G \times \check{G},\check{\scrF}^{(2)},\check{\lambda}_\mu)$.
	 It is also a principal bundle enlargement of $(\check{\scrT},\check{s})$ (Definition \ref{defngroupenlargement})
	 and $(\scrT_2,s_2)$ is a generalized stabilization  (Definition \ref{defngenearlstablization})
	 followed by an approximation stabilization (Definition \ref{defnapproxenlargement})
	 of $(\scrT_3,s_3)$.
		\end{proof}

	\subsection{Forgetful Maps} \label{sectionforgetful}
	
	We wish to show that if one forgets a marked point
	then one gets an appropriate morphism of global Kuranishi charts.
	For simplicity, we will forget the last marked point.
	The same argument holds when forgetting other marked points
	since one can relabel them and all of our constructions
	are equivariant under such relabelings.
	
	Our treatment of forgetful maps relies on the following trick.  In the usual global chart construction, we use that for a stable map $u: C \to X$ with marked points $p_1,\ldots,p_h$, and $L\to X$ having curvature $-2i\pi\Omega$, the bundle $\omega_C(p_1,\ldots,p_h) \otimes u^*L$ has strictly positive degree on all irreducible components. Now note that 
	\[
	\hat{L}_u :=\omega_C(p_1,\ldots,p_{h-1}) \otimes u^*L
	\]
	 has positive degree on an irreducible component $C'$ of $C$ except when the homology class represented by $u(C')$ vanishes, and one of the following properties hold:
	\begin{enumerate}
	\item $C'$ has genus zero with one marked point $p_h$ and two nodal points;
	\item $C'$ has genus zero with marked points $p_h, p_i$ for some $i<h$, and one nodal point;
	\item $C'$ has genus one with one marked point $p_h$ and no nodes.
	\end{enumerate}
In each of these cases, after forgetting $p_h$ the log canonical bundle of $C'$ has degree zero, so that $\hat{L}_u$ has degree zero on $C'$. The map $\phi_F$ associated to a framing $F$ of $H^0(\hat{L}_u^k)$ is regular and automorphism free in both the first two cases, whilst in the third case the curve obtained by forgetting $p_h$ would have genus zero and no marked points, in which case there is no moduli space of stable constant maps since there is no Deligne-Mumford stack $\ccMbar_{1,0}$ (so we don't expect a good theory of forgetful maps anyway). 

	With this in mind, we will now construct three different sets of global Kuranishi data.
	Let $g,h,d \in \bN$.  In light of the previous discussion, we will assume that $(g,h) \neq (1,1)$.

	\begin{enumerate}
		\item \label{item1Kuranishi} We let $G$, $\scrF$, $\scrC$, $\scrF^{(2)}$, $\scrT^\pre$, $P \to\scrT^\pre$, $\scrT$, $s$, $L_u$ be as in Section \ref{subsection independence}.
		Then we have that $(G,\scrF^{(2)}, \lambda_\mu)$ is pre-Kuranishi data associated to $\scrC \to \scrF$ for $\mu$ large
		and $(\scrT,s)$ is global Kuranishi data associated to this pre-Kuranishi data.
		The associated global Kuranishi chart $\bT$ (Definition \ref{defnnoncomplexthickening}) describes the moduli space $\ccMbar_{g,h}(X,\beta)$.
		\item \label{item2Kuranishi} Let $d' = d-k$   and choose a finite dimensional approximation scheme $(V'_\mu,\lambda'_\mu)$
		for $\scrC_{g,h-1,d'} \to \scrF_{g,h-1,d'}$. 
		We also have pre-Kuranishi data $(G,\scrF^{(2)}_{g,h-1,d'},\lambda'_\mu)$.
		Let $G' := U(d'-g+1)$.
		Choose consistent domain metric for $\scrF_{g,h-1,d'}$ (Definition \ref{defn domain metric}).
		We let $P' \to\scrT^\pre(\beta,\lambda'_\mu)$ be the principal $G'$-bundle
		whose fiber over $(\phi,u,e)$ is the space of unitary bases $F$ of the space of sections of the $k$-th power of the  line bundle
		$L'_u := u^*L \otimes (\omega_{\scrC_{g,h-1,d'}/\scrF_{g,h-1,d'}}(p'_1,\cdots,p'_{h-1}))$
		where $p'_1,\cdots,p'_{h-1}$ are the corresponding marked point sections.   (Thus, if $d$ is the degree of $L_u^k$ then $d'$ is the degree of $(L'_u)^k$ since we are twisting by one fewer marked point.)
		Then we have a thickening $\scrT' := P'_\bC$.
		For each such basis $F$ we have an element $(\phi,\phi_F) \in \scrF^{(2)}_{g,h-1,d'}$
		giving us a map $P' \to \scrF^{(2)}_{g,h,h-1,d'}$
		which extends to a $G'_{\bC}$-equivariant map $\scrT \to \scrF^{(2)}_{g,h-1,d'}$ for  $d' = d-k$.
		Then $(\scrT',s')$ is global Kuranishi chart data for $(G,\scrF^{(2)}_{g,h-1,d},\lambda'_\mu)$
		whose associated global Kuranishi chart $\bT'$ describes the moduli space $\ccMbar_{g,h-1}(X,\beta)$.
		\item We let $\widetilde{\scrF}$ be the moduli space of curves $\phi : \Sigma \to \bC \bP^{d'}$
		with $h$-marked points so that if we forget the last marked point, we get an element of $\scrF_{g,h-1,d}$.
		We let $\widetilde{\scrC} \to \widetilde{\scrF}$ be the corresponding universal curve
		and we let $(\widetilde{V}_\mu,\widetilde{\lambda}_\mu)_{\mu \in \bN}$
		be the pullback of $(V'_\mu,\lambda'_\mu)_{\mu \in \bN}$ via the forgetful map $\widetilde{\scrF} \to \scrF_{g,h-1,d}$.
		We let $\widetilde{\scrF}^{(2)}$ be the space of curves mapping to $(\bC \bP^{d'-g})^2$
		whose projection to each factor $\bC \bP^{d'-g}$ is an element of $\widetilde{\scrF}$.
		We then have pre-Kuranishi data $(G',\widetilde{\scrF}^{(2)},\widetilde{\lambda}_\mu)$.
		We have a principal $G'$-bundle $\widetilde{P} \to\scrT^\pre(\beta,\widetilde{\lambda}'_\mu)$
		given by the pullback of $P'$ via the map to $\scrT^\pre(\beta,\lambda'_\mu)$.
		We also have a natural map $\widetilde{s} : \widetilde{P} \to \widetilde{\scrF}$
		of the form $(\phi,\phi_F)$ where $F$ is the framing pulled back from $P'$.
		Then $(\widetilde{\scrT},\widetilde{s})$ is global Kuranishi data for $(G',\widetilde{\scrF}^{(2)},\lambda'_\mu)$
		and whose associated global Kuranishi chart $\widetilde{\bT}$ describes $\ccMbar_{g,h}(X,\beta)$ (See Lemma \ref{lemmaequivcharts} below).
	\end{enumerate}
	There is a natural morphism of global Kuranishi charts $\widetilde{\bT} \to \bT'$
	given by forgetting the last marked point.
	This is the forgetful map.

	\begin{lemma} \label{lemmaequivcharts}
		The global Kuranishi chart data associated to $(\scrT,s)$ and $(\widetilde{\scrT},\widetilde{s})$
		are related by a sequence of germ equivalences, approximation stabilizations and generalized stabilizations
		as well as their inverses. \qed
	\end{lemma}
	The proof of this lemma is identical to the proof of Proposition \ref{proprelated}
	where $\check{G}$ is replaced with $G'$,
	$\check{\scrF}$ is replaced with $\widetilde{\scrF}$,
	$\check{\scrF}^{(2)}$ is replaced with $\widetilde{\scrF}^{(2)}$
	etc.

\subsection{Split Curves} \label{sec:split_curves}

This section gives two different global Kuranish chart presentations
for the space of split nodal curves, i.e. the space of connected nodal curves in $X$  obtained
by gluing genus $g_1$ and $g_2$ nodal curves along a marked point, and shows that they are equivalent. These two charts correspond geometrically to the following two descriptions of a split stable map:
\begin{enumerate}
\item One considers those stable maps of genus $g=g_1+g_2$ whose domains lie in the boundary divisor $D_{g_1,g_2} \subset \ccMbar_g$;
\item One considers separately curves of genus $g_1$ and of genus $g_2$ each with one marked point, and then takes the preimage of the diagonal in $X\times X$ under evaluation from the product moduli space.
\end{enumerate}

Fix natural numbers $g_1$, $g_2$, $h_1$, $h_2$ corresponding to the
genus and number of marked points of each component.
We let $(X,\omega)$ be a closed symplectic manifold with $\beta \in H_2(X;\bZ)$
and $J$ an $\omega$-compatible almost complex structure.

The two global Kuranishi data for split curves are given as follows:

\begin{enumerate}
	\item Let $G = U(d_1+d_2-g+1)$ where $d_1,d_2$ are to be determined soon and $g = g_1 + g_2$.
	Let $d = d_1 + d_2$.
	We let $\scrF$ be the moduli space of pairs of curves $\phi_1 : \Sigma_1 \to \bC \bP^{d-g}$,
	$\phi_2 : \Sigma_2 \to \bC \bP^{d-g}$ with $h_1+1$ and $h_2+1$
	marked points $p_1,\cdots,p_{h_1+1}$, $q_1,\cdots,q_{h_2+1}$ respectively
	so that
	\begin{enumerate}
		\item their degrees sum up to $d$,
		\item $\phi_1(p_{h_1+1}) = \phi_2(q_{h_2+1})$,
		\item the resulting curves $\phi_1 \# \phi_2 : \Sigma_1 \cup_{p_{h_1+1}=q_{h_2+1}} \Sigma_2 \to \bC \bP^{d-g}$, $\phi_1 \# \phi_2|_{\Sigma_j} = \phi_j$, $j=1,2$
		satisfy $H^1((\phi_1 \# \phi_2)^*(O(1)) = 0$ and are automorphism free.
	\end{enumerate}
	Let $\scrC \to \scrF$ be the corresponding universal curve whose fiber over $\phi_1, \phi_2$ as above
	is the union $\Sigma_1 \cup_{p_{h_1+1}=q_{h_2+1}} \Sigma_2$.
	We let $\scrF^{(2)}$ be the moduli space of tuples $(\phi_1,\phi_2),(\phi'_1,\phi'_2)$ mapping to $(\bC \bP^{d-g})^4$
	with the property that $(\phi_1,\phi_2)$ and $(\phi'_1,\phi'_2)$ are in $\scrF$ respectively.
	Choose a finite dimensional approximation scheme $(V_\mu, \lambda_\mu)$ for $\scrC \to \scrF$.
	Choose a Hermitian line bundle $L \to X$ with curvature $-2i\pi\Omega$ with $\Omega$ taming $J$, an integer $3 \leq k$,
	and a consistent domain metric for $\scrF$ as in Definition \ref{defn domain metric}.
	For each $((\phi_1,\phi_2),u)$ in $\scrT^\pre(\beta,\lambda_\mu)$,
	we let $L_u := (\omega_{\scrC|_\phi/\scrF|_\phi}(p_1,\cdots,p_{h_1+1},q_1,\cdots,q_{h_2+1}) \otimes u^*L$.
	Let $P \to\scrT^\pre(\beta,\lambda_\mu)$
	be the principal $G$-bundle whose fiber over $(\phi_1,\phi_2,u)$ is a unitary basis $F$ for $H^0(L_u^{\otimes k})$.
	We now let $d_j$ be the degree of $L_u$ restricted to $\Sigma_j$ where $\Sigma_j$ is the domain of $\phi_j$ for each $j=1,2$.
	Then we have a thickening $\scrT := P_\C$ which consists of tuples $(\phi_1,\phi_2,u,F)$ where $(\phi_1,\phi_2,u) \in\scrT^\pre(\beta,\lambda^0_\mu)$
	and $F$ is a basis of $H^0(L_u^{\otimes k})$.
	For each $(\phi_1,\phi_2,u,F) \in \scrT$,
	let $\phi_F : \Sigma_1 \cup_{p_{h_1+1}=q_{h_2+1}} \Sigma_2 \to \bC \bP^{d-g}$ be as in Equation \eqref{equation}.
	Let $\phi_{F,j}$ be the restriction of $\phi_F$ to $\Sigma_j$ for $j=1,2$.
	We let $s : \scrT \to \scrF^{(2)}$ send $(\phi_1,\phi_2,u,F)$ to $((\phi_1,\phi_2),(\phi_{F_1},\phi_{F_2}))$.
	Then for $\mu$ large enough and $3 \leq k$, $(\scrT,s)$ is global Kuranishi chart data associated to $(G,\scrF^{(2)},\lambda_\mu)$
	(Definition \ref{defnnoncomplexthickening}).

	\item Let $\check{G} = U(d_1-g_1+1) \times U(d_2-g_2+1)$ with $d_1,d_2$ to be determined later.
	We let $\check{\scrF}$ be the moduli space of pairs of curves
	$\phi_1 : \Sigma_1 \to \bC \bP^{d_1-g_1}$, $\phi_2 : \Sigma_2 \to \bC \bP^{d_2-g_2}$ with $h_1+1$ and $h_2+1$
	marked points $p_1,\cdots,p_{h_1+1}$, $q_1,\cdots,q_{h_2+1}$ respectively
	so that
	\begin{enumerate}
		\item their degrees are $d_1$ and $d_2$ respectively,
		\item $H^1((\phi_j)^*(O(1)) = 0$ for $j=1,2$ 
		\item and $\phi_1$, $\phi_2$ are automorphism free.
	\end{enumerate}
	We let $\check{\scrC} \to \check{\scrF}$ be the curve whose fiber over
	$(\phi_1,\phi_2)$ as above is $\Sigma_1 \cup_{p_{h_1+1}=q_{h_2+1}} \Sigma_2$.
	Let $\check{\scrC}_j \to \check{\scrF}$ be the curve over $\check{\scrF}$ whose fiber over $(\phi_1,\phi_2)$
	is $\Sigma_j$ for each $j=1,2$.

	We let $\check{\scrF}^{(2)}$ be the moduli space of tuples $(\phi_1,\phi_2),(\phi'_1,\phi'_2)$ mapping to $(\bC \bP^{d-g})^4$
	with the property that $(\phi_1,\phi_2)$ and $(\phi'_1,\phi'_2)$ are in $\check{\scrF}$ respectively.
	Choose a finite dimensional approximation scheme $(\check{V}_\mu,\check{\lambda}_\mu)$ for
	$\check{\scrC} \to \check{\scrF}$.
	Let $L,k$ be as above and choose a consistent domain metric for $\check{\scrC}_j \to \check{\scrF}$ as in Definition \ref{defn domain metric} for each $j=1,2$.
	For each $((\phi_1,\phi_2),u)$ in $\scrT^\pre(\beta,\check{\lambda}_\mu)$,
	we let $L_{u,1} := (\omega_{\check{\scrC}_1|_{\phi_1}/\check{\scrF}|_{\phi_1}}(p_1,\cdots,p_{h_1+1}) \otimes u^*L|_{\check{\scrC}_1|_{\phi_1}})$
	and
	$L_{u,2} := (\omega_{\check{\scrC}_2|_{\phi_2}/\check{\scrF}|_{\phi_2}}(q_1,\cdots,q_{h_2+1}) \otimes u^*L|_{\check{\scrC}_2|_{\phi_2}})$.
	Let $\check{P} \to\scrT^\pre(\beta,\check{\lambda}_\mu)$
	be the principal $\check{G}$-bundle whose fiber over $(\phi_1,\phi_2,u)$ is a pair of unitary bases $(F_1,F_2)$
	of $H^0(L_{u,1}^{\otimes k})$ and $H^0(L_{u,2}^{\otimes k})$ respectively.
	Then we have a thickening $\check{\scrT} := \check{P}_\bC$ which consists of tuples $(\phi_1,\phi_2,u,F_1,F_2)$
	where $(\phi_1,\phi_2, u) \in\scrT^\pre(\beta,\check{\lambda}_\mu)$
	and $(F_1,F_2)$ are bases of $H^0(L_{u,1}^{\otimes k})$ and $H^0(L_{u,2}^{\otimes k})$ respectively.
	Let $\phi_{F_1} : \Sigma_1 \to \bC \bP^{d_1-g_1}$, $\phi_{F_2} : \Sigma_2 \to \bC \bP^{d_2-g_2}$
	be as in Equation \eqref{equation}.
	We let $\check{s} : \check{\scrT} \to \check{\scrF}^{(2)}$ send $(\phi_1,\phi_2,u,F_1,F_2)$ to $(\phi_1,\phi_2,\phi_{F_1},\phi_{F_2})$.
	Then for $k$, $\mu$ large enough, $(\check{\scrT},\check{s})$ is global Kuranishi chart data associated to $(\check{G},\check{\scrF}^{(2)},\check{\lambda}_\mu)$.
\end{enumerate}

\begin{prop} \label{propsplitrelation}
The Kuranishi  data $(G,\scrT,s)$ and $(\check{G},\check{\scrT},\check{s})$
	are related by a sequence of
	 enlargements (Definition \ref{defngroupenlargement})
	and general stabilizations (Definition \ref{defngenearlstablization}) and their inverses.
	Hence their associated Kuranishi charts are equivalent.
\end{prop}
\begin{proof}
	This is done via three intermediate global Kuranishi data.
	The first is constructed as follows:
	We start with pre-Kuranishi data $(G \times \check{G},\scrF^{(2)},\lambda_\mu)$ where $G \times \check{G}$ acts by projecting to $G$ first.
	We let $\scrT_1$ be the principal $\check{G}$-bundle
	over $\scrT$ whose fiber over a $(\phi_1,\phi_2,u,F) \in \scrT$ is a pair of unitary bases
	$(F_1,F_2)$
	of $H^0(L_{u,1}^{\otimes k})$ and $H^0(L_{u,2}^{\otimes k})$ respectively.
	We let $s_1$ be the composition
	$\scrT_1 \to \scrT \lra{s} \scrF^{(2)}$.
	Then $(\scrT_1,s_1)$ is global Kuranishi data associated to
	$(G \times \check{G},\scrF^{(2)},\lambda_\mu)$.
	It is also an enlargement of $(\scrT,s)$ (Definition \ref{defngroupenlargement}).

	We will now construct the second global Kuranishi chart data.
	We start with the space $\scrF_2$
	of pairs of curves $(\phi_1,\phi_2)$ curves mapping to $\bC \bP^{d-g} \times \bC \bP^{d_1-g_1}$ and $\bC \bP^{d-g} \times \bC \bP^{d_2-g_2}$ respectively
	so that $(\Pi_1 \circ \phi_1, \Pi_1 \circ \phi_2) \in \scrF$ and $(\Pi_2 \circ \phi_2, \Pi_2 \circ \phi_2) \in \check{\scrF}$
	where $\Pi_1$ and $\Pi_2$ are the projection maps to the first and second factors of $\bC \bP^{d-g} \times \bC \bP^{d_1-g_1}$ and $\bC \bP^{d-g} \times \bC \bP^{d_2-g_2}$ respectively.
	We let $\scrC_2 \to \scrF_2$
	be the curve
	over $\scrF_2$ whose fiber over
	$\phi_1 : \Sigma_1 \to \bC \bP^{d-g} \times \bC \bP^{d_1 -g_1}$,
	$\phi_2 : \Sigma_1 \to \bC \bP^{d-g} \times \bC \bP^{d_2 -g_2}$
	is $\Sigma_1 \cup_{p_{h_1+1} = q_{h_2+1}} \Sigma_2$
	where $p_1,\cdots,p_{h_1+1}$, $q_1,\cdots,q_{h_2+1}$
	are the marked points on the domains of $\phi_1$ and $\phi_2$ respectively.

	We let $\scrF^{(2)}_2$ be the space of pairs of curves $(\phi_1,\phi_2)$ mapping to $(\bC \bP^{d-g} \times \bC \bP^{d_1-g_1})^2$ and $(\bC \bP^{d-g} \times \bC \bP^{d_2-g_2})^2$ respectively
	so that $(P_1 \circ \phi_1, P_1 \circ \phi_2)$
	and $(P_2 \circ \phi_1, P_2 \circ \phi_2)$
	are in $\scrF_2$ where $P_1 : (\bC \bP^{d-g} \times \bC \bP^{d_1-g_1})^2 \to (\bC \bP^{d-g} \times \bC \bP^{d_1-g_1})^2$
	and
	$P_2 : (\bC \bP^{d-g} \times \bC \bP^{d_2-g_2})^2 \to (\bC \bP^{d-g} \times \bC \bP^{d_2-g_2})^2$
	are the projection maps to the first and second factors respectively.	

	Let $(V_{\mu,2},\lambda_{\mu,2})_{\mu \in \bN}$
	be a finite dimensional approximation scheme for $\scrF_2$
	given by the sum of the pullback (\ref{defnpulbackoffdscheme}) of
	$(V_\mu,\lambda_\mu)_{\mu \in \bN}$
	and $(\check{V}_\mu,\check{\lambda}_\mu)_{\mu \in \bN}$
	respectively.

	We have an associated pre-thickening
	$\scrT^\pre_2 :=\scrT^\pre(\beta,\lambda_{\mu,2})$
	for some large $\mu$.
	
	 We have a principal $G \times \check{G}$-bundle $P_2 \to\scrT^\pre_2$
	whose fiber over $(\phi,u,e) \in\scrT^\pre_2$
	the space of triples $(F,F_1,F_2)$
	of unitary bases of $H^0(L_u)$, $H^0(L_{u,1})$ and $H^0(L_{u,2})$
	respectively.
	So elements of $\scrT^\pre_2$ are tuples of the form $(\phi_1,\phi_2,u,F,F_1,F_2)$.
	We also have a map
	$s_2 :\scrT^\pre_2 \to \scrF_2^{(2)}$
	sending $(\phi_1,\phi_2,u,F,F_1,F_2))$ to $(\phi_1,\phi_2,(\phi_F|_{\Sigma_1} \times \phi_{F_1},\phi_F|_{\Sigma_2} \times \phi_{F_2}))$
	where $\Sigma_1 \cup_{p_{h_1+1}=q_{h_2+1}} \Sigma_2$ is the domain of $u$.

	 Then $(\scrT_2,s_2)$ is global Kuranishi data associated to
	 $(G \times \check{G},\scrF_2,\lambda_{\mu,2})$.
	 Also $(\scrT_2,s_2)$
	 is a general stabilization (Definition \ref{defngenearlstablization})
	 followed by an approximation stabilization (Definition \ref{defnapproxenlargement})
	 of $(\scrT_1,s_1).$

	 Finally we construct the third global Kuranishi data.
	 This is the same as the first one,
	 but with $\scrF$ and $\check{\scrF}$
	 swapped.
	 We start with pre-Kuranishi data $(G \times \check{G},\check{\scrF}^{(2)},\check{\lambda}_\mu)$
	where $G \times \check{G}$ acts by projecting to $\check{G}$ first.
	We let $\scrT_3$ be the principal $G$-bundle
	over $\check{\scrT}$ whose fiber over a point 
	$(\phi_1,\phi_2,u,e,F_1,F_2) \in \check{\scrT}$
	is a unitary basis $\check{F}$
	of $H^0(L_u^{\otimes k})$ using the consistent
	domain metric for $\check{\scrF}$.
	We let $s_3$
	be the composition $\scrT_3 \to \check{\scrT} \lra{\check{s}} \check{\scrF}.$
	Then $(\scrT_3,s_3)$ is global Kuranishi data associated to
	 $(G \times \check{G},\check{\scrF}^{(2)},\check{\lambda}_\mu)$.
	 It is also a group enlargement of $(\check{\scrT},\check{s})$ (Definition \ref{defngroupenlargement})
	 and $(\scrT_2,s_2)$ is a generalized stabilization  (Definition \ref{defngenearlstablization})
	 followed by an approximation stabilization (Definition \ref{defnapproxenlargement})
	 of $(\scrT_3,s_3)$.
\end{proof}

\subsection{Genus Reduction} \label{sec:genusreduction}

This section gives two different global Kuranish chart presentations
for the space of curves together with a special non-separating node and shows that they are equivalent. These two charts correspond geometrically to the following two descriptions of a self-glued stable map:
\begin{enumerate}
\item One considers those stable maps of genus $g$ whose domains map to the boundary divisor of $\ccMbar_g$ of curves with a non-separating node;
\item One considers curves of genus $g-1$ with two additional marked points, and then takes the preimage of the diagonal in $X\times X$ under  the evaluation map corresponding to these two marked points.
\end{enumerate}

Fix natural numbers $g$, $h$ corresponding to the
genus and number of marked points.
We let $(X,\omega)$ be a closed symplectic manifold with $\beta \in H_2(X;\bZ)$
and $J$ an $\omega$-tame almost complex structure.

We will now construct two different global Kuranishi data for these curves.

\begin{enumerate}
	\item Let $G = U(d-g+1)$ where $d$ is to be determined soon.
	We let $\scrF$ be the moduli space of genus $g-1$ curves $\phi : \Sigma \to \bC \bP^{d-g}$,
	with $h+2$
	marked points $p_1,\cdots,p_{h+2}$ respectively
	so that
	\begin{enumerate}
		\item their degree is $d$,
		\item $\phi(p_{h+1}) = \phi(p_{h+2})$,
		\item the resulting curve $\phi^\# : \Sigma^\# := \Sigma / \{p_{h+1}=p_{h+2}\} \to \bC \bP^{d-g}$, where $\phi$ is the composition $\Sigma \twoheadrightarrow \Sigma^\# \lra{\phi^\#} \bC \bP^{d-g}$,
		satisfies $H^1((\phi^\#)^*(O(1)) = 0$ and is automorphism free.
	\end{enumerate}
	Let $\scrC \to \scrF$ be the corresponding universal curve whose fiber over $\phi$ as above
	is the glued curve $\Sigma^\#$.
	We let $\scrF^{(2)}$ be the moduli space of tuples $(\phi,\phi')$ mapping to $(\bC \bP^{d-g})^2$
	with the property that $\phi$, $\phi'$ are in $\scrF$ respectively.
	Choose a finite dimensional approximation scheme $(V_\mu, \lambda_\mu)$ for $\scrC \to \scrF$.
	Choose a Hermitian line bundle $L \to X$ with curvature $-2i\pi\Omega$ with $\Omega$ taming $J$, an integer  $3 \leq k $,  
	and a consistent domain metric for $\scrF$ as in Definition \ref{defn domain metric}.
	For each $(\phi,u)$ in $\scrT^\pre(\beta,\lambda_\mu)$,
	we let $L_u := (\omega_{\scrC|_\phi/\scrF|_\phi}(p_1,\cdots,p_{h+1},p_{h+2}) \otimes u^*L$.
	Let $P \to\scrT^\pre(\beta,\lambda_\mu)$
	be the principal $G$-bundle whose fiber over $(\phi,u)$ is a unitary basis $F$ for $H^0(L_u^{\otimes k})$.
	We now let $d$ be the degree of $L_u$ restricted to the domain $\Sigma$ of $\phi$.
	Then we have a thickening $\scrT := P_\C$ which consists of tuples $(\phi,u,F)$ where $(\phi,u) \in\scrT^\pre(\beta,\lambda^0_\mu)$
	and $F$ is a basis of $H^0(L_u^{\otimes k})$.
	For each $(\phi,u,F) \in \scrT_0$,
	let $\phi'_F : \Sigma^\# \to \bC \bP^{d-g}$ be as in Equation \eqref{equation}
	and let $\phi_F$ be the composition $\Sigma \twoheadrightarrow \Sigma^\# \lra{\phi'_F} \bC \bP^{d-g}$.
	We let $s : \scrT \to \scrF^{(2)}$ send $(\phi,u,F)$ to $(\phi,\phi_F)$.
	Then for $\mu$ large enough, $(\scrT,s)$ is global Kuranishi chart data associated to $(G,\scrF^{(2)},\lambda_\mu)$
	(Definition \ref{defnnoncomplexthickening}).

	\item Let $\check{G} = U(\check{d})$ with $\check{d}$ to be determined later.
	We let $\check{\scrF}$ be the moduli space of genus $g-1$ curves
	$\phi : \Sigma\to \bC \bP^{\check{d}-g+1}$ with $h+2$
	marked points $p_1,\cdots,p_{h+1},p_{h+2}$ respectively
	so that
	\begin{enumerate}
		\item the degree of $\phi$ is $\check{d}$,
		\item $H^1(\phi^*(O(1)) = 0$
		\item and $\phi$ is automorphism free.
	\end{enumerate}
	We let $\check{\scrC} \to \check{\scrF}$ be the curve whose fiber over
	$\phi$ as is $\Sigma^\# := \Sigma / \{p_{h+1}=p_{h+2}\}$.
	We let $\widetilde{\scrC} \to \check{\scrF}$ be the curve whose fiber over $\phi$ is its domain $\Sigma$.

	We let $\check{\scrF}^{(2)}$ be the moduli space of tuples $(\phi,\phi')$ mapping to $(\bC \bP^{d'-g+1})^2$
	with the property that $\phi$ and $\phi'$ are in $\check{\scrF}$.
	Choose a finite dimensional approximation scheme $(\check{V}_\mu,\check{\lambda}_\mu)$ for
	$\check{\scrC} \to \check{\scrF}$.
	Let $L,k$ be as above and choose a consistent domain metric for $\check{\scrC} \to \check{\scrF}^{(2)}$ as in Definition \ref{defn domain metric}.
	For each $((\phi,u)$ in $\scrT^\pre(\beta,\check{\lambda}_\mu)$,
	we let $\check{L}_u := \omega_{\widetilde{\scrC}|_\phi/\check{\scrF}|_\phi}(p_1,\cdots,p_{h+2}) \otimes u^*L$.
	We let $\check{d}$ be the degree of $L=\check{L}_u$.
	Let $\check{P} \to\scrT^\pre(\beta,\check{\lambda}_\mu)$
	be the principal $\check{G}$-bundle whose fiber over $(\phi,u)$ is a unitary basis $F$
	of $H^0(\check{L}_u^{\otimes k})$.
	Then we have a thickening $\check{\scrT} := \check{P}_\bC$ which consists of tuples $(\phi,u,F)$
	where $(\phi, u) \in\scrT^\pre(\beta,\check{\lambda}_\mu)$
	and $F$ is a basis of $H^0(\check{L}_u^{\otimes k})$.
	Let $\phi_F : \Sigma \to \bC \bP^{\check{d}-g}$
	be as in Equation \eqref{equation}.
	We let $\check{s} : \check{\scrT} \to \check{\scrF}^{(2)}$ send $(\phi,u,F_1,F_2)$ to $(\phi,\phi_F)$.
	Then for $k$, $\mu$ large enough, $(\check{\scrT},\check{s})$ is global Kuranishi chart data associated to $(\check{G},\check{\scrF}^{(2)},\check{\lambda}_\mu)$.
\end{enumerate}

The proof of the following proposition is similar to the proof of Proposition \ref{propsplitrelation}
above and so we omit it.

\begin{prop}
The Kuranishi  data $(G,\scrT,s)$ and $(\check{G},\check{\scrT},\check{s})$
	are related by a sequence of
	 enlargements (Definition \ref{defngroupenlargement})
	and general stabilizations (Definition \ref{defngenearlstablization}) and their inverses.
	Hence their associated Kuranishi charts are equivalent. \qed
\end{prop}

\subsection{Notational conventions}
\label{Subsec:constant_maps}

Suppose $X$ is a point, and fix a single pair $(g,h)$. All stable maps are constant.  In this case,  the 3rd power of the relative log canonical bundle used in our general construction is a relatively very ample Hermitian line bundle $\mathcal{L}$ on the universal curve over $\ccMbar_{g,h}$, and outputs a global chart for $\ccMbar_{g,h}$. Although the map $\lambda$ in the approximation scheme vanishes so the thickening is entirely composed of constant maps together with $\mathcal{L}$-framings, there is a non-trivial obstruction bundle, coming from $H^0(u^* T\bC \bP^q)$, where $q+1 = \rk\,H^0(\calL|_\Sigma)$ for any curve $\Sigma \in \ccMbar_{g,h}$, together with the space $\scrH$ of $(q+1) \times (q+1)$-Hermitian matrices.
\medskip

One can relate the output of the general construction to a simpler global chart.  The space of \emph{unitary} $\mathcal{L}$-framed curves itself gives a global quotient presentation 
\[
\ccMbar_{g,h} = \widetilde{\calM} / H
\]
for the unitary group $H = U(q+1)$. The space $\widetilde{\calM}$ carries a natural $H$-transverse complex structure, and is naturally an element of the category $\Glo$.  This `unitary'  global chart, with trivial obstruction bundle,  is related to the one obtained from our construction by the stabilisation procedure for the $H$-bundle $H_{g,h,q+g} \oplus \scrH$.
\medskip

Consider now the case of a general target $X$, there is a larger global chart for  $\ccMbar_{g,h}(X,\beta,J)$ in which we frame both $L_u$ and also the bundle $(st)^*\scrL$ where $\scrL \to \scrC_{g,h}$ is relatively ample over the universal curve over Deligne-Mumford space and $\st$ is the classical stabilisation map.  In the notation above, this leads to a chart
 which is 
 $G$-complex and 
 $H$-transversely almost complex, and for which there is an $H$-equivariant  forgetful map $\scrT \to \widetilde{\calM}$ which plays the role of the stabilisation map $\scrT \to \ccMbar_{g,h}$. Combined with the trick from Lemma \ref{lem:subcategory}, this gives us stabilisation maps inside the category $\Glo$ (where $\widetilde{\calM}$ carries a trivial $G$-action for the original group $G$ arising from the framing choice $F$ in the construction of $\scrT$).

For simplicity, we will use the notation $\scrF_{g,h}(\bP^{d-g},d)$ and $\scrT_{g,h}(\beta)$ for the (smooth) spaces  of regular stable curves in projective space and a thickening of $\ccMbar_{g,h}(\beta)$ when the precise choices in the construction are not central to the immediate context. We will also usually suppress the additional framings of line bundles over domain curves, and the choice of global $H$-covering $\widetilde{\calM}$ for Deligne-Mumford space, and simply write stabilisation as a map $\scrT \to \scrF \to \ccMbar_{g,h}$. We thus have maps
\[
\scrT_{g,h}(\beta) \to \scrF_{g,h}(\bP^{d-g}, d) \to \ccMbar_{g,h}
\]
(where $d$ depends on $\omega(\beta)$ but also various choices in the construction) whose composite we denote by $\st$.

As discussed after Corollary \ref{cor globalKuranishichart}, a global chart constructed from framed stable maps is not naturally smooth but the thickening $\scrT$ admits a fibrewise $C^1_{loc}$-map $\scrT \to \scrF$ to a quasi-projective variety.  As explained in \cite{AMS-Hamiltonian}, this is enough to show that the tangent microbundle $T_{\mu}\scrT$ admits a vector bundle lift; by further stabilising, one can then obtain smooth global charts, well-defined up to equivalence in the sense of Definition \ref{defn:global_charts_category}. 
\emph{We will assume we have done this henceforth.} After smoothing and stabilisation to be transversely stably complex, our charts naturally belong to the category $\Glo$ considered previously. 
By \cite[Lemma 4.5]{AMS-Hamiltonian} we can also replace a map between the thickenings of such charts by a smooth
submersion.
We summarise the results of this section:

\begin{corollary}\label{cor:charts_summary}
	The moduli spaces $\ccMbar_{g,n}(X,J,\beta)$  of stable maps admit a distinguished system of global charts $\bT = (G,\scrT,E,s)$ with thickenings $\scrT:= \scrT_{g,n}(X,J,\beta)$ for which  \begin{enumerate}
	\item the virtual bundles $T\scrT - \frak{g}$ and  $E$ admit $G$-equivariant stable almost complex structures;
		\item evaluation $ev: \ccMbar_{g,n}(X,J,\beta) \to X^n$ and stabilisation $\st: \ccMbar_{g,n}(X,J,\beta) \to \ccMbar_{g,n}$ both extend to $G$-equivariant maps on $\scrT_{g,n}(X,J,\beta)$;
		\item evaluation and stabilisation maps, and the stable complex structures, are entwined by equivalences or cobordisms of distinguished global charts coming from varying auxiliary data or changing $J$, so the charts are well-defined up to isomorphism in the sense of Section \ref{sec:counting-theories}. \qed
			\end{enumerate} 
\end{corollary}

\section{Gromov-Witten invariants and tautological modules}
\label{sec:grom-witt-invar}
\subsection{Invariants as generalised homology classes}

Take a counting theory $\bE$  on the category $\Glo$ of global charts as in Section \ref{sec:counting-theories}. We write $L_{\bE}$ for the formal group associated to  $\bE$. 

Fix a compact symplectic $(X,\omega)$, a class $\beta \in H_2(X;\bZ)$, a pair $(g,n)$ and an almost complex structure $J$ taming $\omega$.  We have a global chart $\bT_{g,n}(\beta)$ with zero-set $Z$ such that $Z/G$ is the compactified moduli space $\ccMbar_{g,n}(\beta,J)$ of $J$-holomorphic stable maps. Evaluation and stabilisation define a map $\bT_{g,n}(\beta) \to X^n \times \ccMbar_{g,n}$ so one can push-forward the virtual class to a class
\[
I_{g,n}(\beta) \in E_*(X^n \times \ccMbar_{g,n}) \cong E^*(X^n \times \ccMbar_{g,n}),
\]
where we supress the dimension shift in the second isomorphism.

\begin{proposition} The above class is a well-defined symplectic invariant of $(X,\omega, \beta,g,n)$, and does not depend on the choice of global chart or taming almost complex structure $J$.
\end{proposition}

\begin{proof}  The global chart for $\ccMbar_{g,n}(\beta,J)$ is well-defined up to equivalence for fixed $J$ by Proposition \ref{propchartparameterized}, and up to cobordism for varying $J$ by the discussion of Section \ref{Subsec:families}.  The associated virtual class $[\bT_{g,n}(\beta)] \in E_*^{lf}(\scrT_{g,n}(\beta,J))$ is preserved by equivalence of global charts from the formalism of counting theories, and its push-forward is invariant under cobordism by Lemma \ref{lem:cobordism}.  \end{proof}

Again by cobordism invariance, $I_{g,n}(\beta)$ only depends on the deformation equivalence class of $\omega$, i.e. its connected component in the space of symplectic forms.

\subsection{First properties}

We discuss analogues of the effectivity, symmetry, and point-mapping axioms for Gromov-Witten invariants.

\begin{prop}
	If $\langle [\omega],\beta\rangle < 0$ then $I_{g,n}(\beta) = 0$.
\end{prop}
\begin{proof} The moduli space $\ccMbar_{g,n}(\beta,J) = \emptyset$ for any taming $J$, hence the zero-set $Z$ of the section $s$ is empty and so the equivariant Euler class $[\bT]$ vanishes, for any global chart.
\end{proof}

The symmetric group $\Sym_n$ acts on $X^n \times \ccMbar_{g,n}$ by permuting factors in $X^n$  respectively marked points in $\ccMbar_{g,n}$. This yields a representation 
\[
\Sym_n \to \Aut_{E^*(pt)}E_*(X^n \times \ccMbar_{g,n})
\]
to the $E^*(pt)$-linear automorphisms. 

\begin{prop} \label{prop:symmetry} The class $I_{g,n}(\beta)$ is $\Sym_n$-invariant. \end{prop}

\begin{proof} Given data $\mu$ defining a global chart for $\ccMbar_{g,n}(\beta)$, we can pullback by $\sigma \in \Sym_n$ to obtain new data $\sigma^*\mu$.  The virtual classes for the two sets of data are equivalent by Proposition \ref{propchartparameterized}.
\end{proof}

\begin{remark} Going beyond Proposition \ref{prop:symmetry}, one can incorporate finite-group equivariance into the construction of global charts and build $\Sym_n$-equivariant global charts for moduli spaces of stable maps, obtaining a refined invariant $I_{g,n}(\beta)^{\Sym_n} \in E^*_{\Sym_n}(X^n \times \ccMbar_{g,n})$.  This is relevant to a program of Givental \cite{Givental-permutation}. \end{remark} 

We now consider the case of maps in the trivial homology class. 
As usual, let $\widetilde{\calM}/H$ be a global chart presentation for $\ccMbar_{g,n}$. The orbibundle $\lambda_{\mathrm{Hodge}} := R^1p_*\mathcal{O}_{\ccC} \boxtimes TX \to \ccMbar_{g,n} \times X$ has a natural lift, which we still denote by $\lambda_{\mathrm{Hodge}}$, to an $H$-equivariant bundle over $\widetilde{\calM}$. 

\begin{prop} Let $\beta = 0 \in H_2(X;\bZ)$. The space $\ccMbar_{g,n}(X,0,J)$ admits a global chart 
	\[
	(G,\scrT, V,s) = (H, X \times \widetilde{\calM},  \lambda_{\mathrm{Hodge}}, 0).
	\]
\end{prop}

\begin{proof}
	For $\beta = 0 \in H_2(X;\bZ)$, the moduli space $\ccMbar_{g,n}(\beta,J)$ agrees with the product $ X \times \ccMbar_{g,n}$, since all stable maps are constant (however this space typically has the wrong virtual dimension). The stabilisation map $\scrT \to \ccMbar_{g,n}$ is the natural projection, whilst the evaluation map $\scrT \to X^n$ 
	is the composition of second projection and inclusion of the small diagonal $
	\scrT = \widetilde{\ccM}\times X  \stackrel{pr_2}{\longrightarrow} X \stackrel{\Delta}{\longrightarrow} X^n.$ 
	
	Construct the global chart by the usual procedure from Section \ref{Subsec:easier}. Over the thickening $\scrT^{\pre}(0,\lambda_\mu)$ there is a $G$-bundle $V$ with fibre the cokernel of the linearised $\overline{\partial}$-operator. The image of $\lambda_\mu$ surjects onto $V$ and the moduli space of constants is cut out `cleanly' (in the fibrewise $C^1_{loc}$-sense) by the section $s$. Whenever the zero-section of a global chart is cut out cleanly in this sense, the chart is stabilisation equivalent to the zero-locus equipped with the identically vanishing section of the obstruction bundle. Identifying $V$ with the Hodge bundle gives the result.
\end{proof}

\begin{Example} In the counting theories associated to Morava $K$-theories or complex $K$-theory, $I_{g,n}(X,0) \in E_*(X^n \times \ccMbar_{g,n})$ is given by the push-forward of the  Euler class of $ \lambda_{\mathrm{Hodge}}$ under the map $X \times \ccMbar_{g,n} \to X^n \times \ccMbar_{g,n}$ associated to the inclusion of the small diagonal $X \hookrightarrow X^n$. \end{Example}

The last of the `elementary' axioms for rational cohomological Gromov-Witten invariants is the `forgetful axiom' which allows one to forget a marked point. The usual setting involves the diagram of spaces 
\begin{equation}
	X^{n+1} \times \ccMbar_{g,n+1} \to X^n \times \ccMbar_{g,n+1}  \leftarrow  X^n \times \ccMbar_{g,n}  
\end{equation}
which respectively forget the last factor and are given by a choice of section of the forgetful map $\ccMbar_{g,n+1}  \to \ccMbar_{g,n}  $ given by a choice of marked point. One hopes to compare the pushforward of $ I_{g,n}(\beta)$ to $ E_*(X^n \times \ccMbar_{g,n+1})$ with the cap product of the fundamental class of $ \ccMbar_{g,n}$ with  $ I_{g,n+1}(\beta)$.  

However, this axiom seems to have no elementary analogue in our setting.  We briefly discuss the issue.  Consider the diagram
\[
\xymatrix{
\bT_{g,n+1}(X,\beta) \ar[r] \ar[d] & \bT_{g,n}(X,\beta) \ar[d] \\
\scrF_{g,n+1}(\bP^{d-g+1},d) \times X^{n+1} \ar[r] \ar[d] & \scrF_{g,n}(\bP^{d-g+1},d) \times X^n \ar[d] \\
\ccMbar_{g,n+1}\times X^{n+1}  \ar[r] & \ccMbar_{g,n} \times X^n
}
\]
The top square is essentially a pullback square, to which the axioms of a counting theory apply, but we have formulated the invariants $I_{g,n}(\beta)$ after pushing forwards all the way to the bottom line.  Although the bottom square is entirely algebro-geometric, the vertical maps are not proper, and it seems hard to analyse this using our axiomatic framework. On the other hand, the square formed by the top and bottom lines is not a pullback square, because of configurations in which the point being forgotten lies on an unstable bubble component, which is then contracted by stabilisation. 

\begin{remark}\label{rmk:forgetful in K-theory}
When $g=0$ one can compactify $\scrF_{0,n}(\bP^d,d)$ to  the compact complex orbifold $\ccMbar_{0,n}(\bP^d,d)$.  Starting from here, and the fact that the $K$-theory fundamental class  is preserved by pushforward under a map with rationally connected fibres,  one can obtain the forgetful axiom for genus zero invariants in complex $K$-theory; this imitates Lee's proof from \cite{Lee-QKFoundations}, in the special case $g=0$, but for more general symplectic manifolds. However,  we do not know a Morava analogue.
\end{remark}

\subsection{Tautological modules\label{sec:tautological}}

Suppose $(X,J)$ is a smooth quasi-projective algebraic variety. There is a collection of `tautological rings' associated to the moduli spaces $\ccMbar_{g,n}(X,\beta,J)$ of stable maps, which generalise the tautological rings $R^*(\ccMbar_{g,n}) \subset H^*(\ccMbar_{g,n})$  of the moduli spaces of curves, see  \cite{Mustata-Mustata} and \cite{Bae}.  The definition of the product in the ring is somewhat non-trivial because of the singularities of the spaces $\ccMbar_{g,n}(X,\beta,J)$, and proceeds via Fulton's `operational Chern-classes'.  Since the operational Chern class exists in algebraic cobordism, one can define the tautological ring in any complex-oriented theory $\bE$.  

The tautological rings have natural analogues for any counting theory on $\Glo$.  Suppose now $(X,\omega)$ is a compact symplectic manifold with compatible almost complex structure $J$, and fix a curve class $\beta$. Fix a global chart $(G,\scrT,\beta,E,s)$ for a moduli space $\ccMbar_{g,n}(\beta,J)$, with thickening $\scrT = \scrT_{g,n}(\beta)$ and virtual class $[\bT_{g,n}(\beta)]$.

We may consider the thickening $ \scrT_{g,n}(\beta) $ as a non-proper stably transversely almost complex global chart, equipped with the trivial vector bundle. As such, we have a forgetful map $\bT \to  \scrT_{g,n}(\beta) $, which induces a pullback
\begin{equation}
	E^*(\scrT_{g,n}(\beta) )  \to E^*(\bT_{g,n}(\beta)).  
\end{equation}
We also have a forgetful map $\scrT_{g,n}(\beta)$ to $ \scrF_{g,n}(\bP^{d-g},d)$, and from there to $\ccMbar_{g,n}(\bP^{d-g},d) $ and to $\ccMbar_{g,n}$.
\begin{defn}
	The \emph{tautological ring} $R^*(\bT_{g,n}(\beta))  $ of $ \bT_{g,n}(\beta)$ is the subring generated by
	\begin{enumerate}
		\item  the pullback of the first Chern class of the $i$-th cotangent line bundle $T_i^* \in E^*(\ccMbar_{g,n})$,
		\item the classes of the divisors in $ \scrT_{g,n}(\beta)$  associated to splitting $(\beta,g,n)$ as $(\beta_1 + \beta_2, g_1 + g_2, n_1 + n_2 + 2)$, and choosing an integer $1 \leq i \leq n$.
	\end{enumerate}
\end{defn}
\begin{rem}
	To clarify the construction of the second type of classes, observe that $\scrF_{g,n}(\bP^{d-g},d) $ inherits (from $\ccMbar_{g,n}(\bP^{d-g},d) $) divisors associated to splitting $(\beta,g,n)$ as $(\beta_1 + \beta_2, g_1 + g_2, n_1 + n_2 + 2)$, and choosing an integer $1 \leq i \leq n$. We  formulate the divisors in $\scrT_{g,n}(\beta) $ instead because the splitting by homology class is finer, in general, than the splitting by degree (area).
\end{rem}

In practice, we shall need to consider `tautological modules' which lie in homology rather than cohomology. Because the moduli spaces of stable maps are compact, we have a natural isomorphism
\begin{equation}
	E^*_G(\scrT | Z) = E^{lf}_*(\bT) \cong  E_*(\bT).
\end{equation}
\begin{defn}
	The \emph{tautological module} $R_*(\bT) \subset  E_*(\bT)$ is the submodule
	\[
	R^{g,n}_*(\beta;\bE) :=  R^*(\bT_{g,n}(\beta))   \cap [\bT_{g,n}(\beta)]  \, \subset E^{lf}_*(\bT).
	\]
\end{defn}
We will say that classes in the tautological ring, respectively module, which are not in the span of the unit, respectively virtual class, have `positive (co)dimension'.

\begin{lemma} An equivalence of distinguished global charts $\bT \to \bT'$ for $\ccMbar_{g,n}(\beta)$ yields an isomorphism of tautological modules. \end{lemma} 

\begin{proof} By definition the equivalence respects the virtual classes $[\bT]$ and $[\bT']$, and from Corollary \ref{cor:charts_summary} is compatible with evaluation and stabilisation maps.  
	In fact, these equivalences commute with the natural submersions from the thickenings to
	$\scrF_{g,h,d}$. Since the cotangent line bundles $T^*_i$ as well as the divisors in $\scrT_{g,n}(\beta)$
	are pullbacks of the corresponding divisors in $\scrF_{g,h,d}$,
	we get a corresponding isomorphism of tautological modules.
\end{proof}

\begin{rem} Working with pullbacks from the spaces $\scrF_{g,h,d}$ is analogous, in our setting, to working with pullbacks of gluing diagrams of the Artin stacks of `stable modular curves with degree' discussed in \cite{Lee-QKFoundations}. \end{rem}

\begin{defn} \label{defn:corrected}
	A \emph{corrected virtual class} for a moduli space $\ccMbar_{g,n}(\beta)$ presented by a global chart $(G,\scrT,E,s)$ is a class 
	\begin{equation} \label{eqn:corrected}
		[\bT_{g,n}(\beta)]^{corr} = [\bT_{g,n}(\beta)] + \alpha \in E_*^{lf}(\bT)
	\end{equation}
	where $\alpha$ belongs to the tautological module and has positive codimension. 
\end{defn}

\section{Quantum generalised cohomology}
\label{sec:quant-gener-cohom}

Let $\bE$ denote a counting theory on the category $\Glo$ of global charts.  In this section, we define an associative algebra $Q\bE^*(X; \Lambda_{\bE_*})$ which we call quantum $\bE$-cohomology.  The proof of associativity involves correcting the `naive' virtual classes for moduli spaces by corrections arising from arbitrary codimension strata of singular stable maps, which appear in normal crossing degenerations of moduli spaces when the domain curve degenerates (splitting). The need for such corrections in principle is a fundamental insight due to Givental \cite{Givental}. In the case of ordinary $K$-theory, he formulated these as corrections to the Poincar\'e pairing, but to work with more general theories $\bE$ governed by more complicated formal groups it seems necessary to correct the fundamental classes themselves.

\subsection{Morphisms associated to correspondences and their compositions}

The construction of operations in Gromov-Witten theory starts with a global chart equipped with evaluation maps
\[
\scrT \stackrel{e_{\pm}}{\longrightarrow} X_1 \times X_2.
\]
This defines a map 
\[
\Phi_{[\bT]}: E^*(X_1) \to E^*(X_2), \quad \alpha \mapsto D_{X_2} \circ (e_+)_* (e_-^*\alpha \cap [\bT])
\]
where $D_{X_2}: E_*(X_2) \to E^{\dim(X)-*}(X_2)$ is the duality isomorphism. In order to formulate properties of the resulting operations, we have to be able to geometrically understand the composition of such maps using the product global chart.

Given two global charts $e_{\pm}: \scrT_e \to X_1\times X_2$ and $f_{\pm}: \scrT_f \to X_2\times X_3$ as above, consider the product chart $\bT_e \times \bT_f$ equipped with the evaluation maps
\begin{equation}
  \begin{tikzcd}
    &\scrT_e \times \scrT_f \arrow[dl,"e_-"] \arrow[d,"e_+ \times f_-"] \ar[dr,"f_-"] &  \\
   X_1   & X_2 \times X_2 &  X_3.
  \end{tikzcd}
\end{equation}
Using the middle map to pullback the dual of the fundamental class of the diagonal on the $X_2$, we obtain a class
\begin{equation} \label{eq:pullback_fundamental_class}
   ((e_+ \times f_-)^*  \Delta_{e_+,f_-}) \cap ([\bT_e]\times[\bT_f])
\end{equation}
which agrees with the pushforward of the fundamental class of the fibre product. 

\begin{lemma} \label{lem:composition_operations} 
	The composition $\Phi_{[\bT_f]} \circ \Phi_{[\bT_e]}$ is equal to the map
	\[
	\alpha \mapsto  D_{X_3} \circ (f_+)_* (e_-^*(\alpha)  \cap (  (e_+ \times f_-)^*  \Delta_{e_+,f_-}) \cap  [\bT_e]\times[\bT_f])).
	\]
\end{lemma}

\begin{proof} This is standard from Fourier-Mukai theory if the fundamental classes are ordinary and not virtual. The extension to global charts follows from  the symmetric monoidal hypothesis on $\bE$, the equality $[\bT_e \times \bT_f] = [\bT_1] \otimes [\bT_2]$, and the pullback axiom for the restriction to the diagonal.  \end{proof}

\subsection{A correlator}

For a class $\beta \in H_2(X;\bZ)$ we have a thickening $\scrT_{0,3}(\beta)$ with evaluation maps 
\[
\xymatrix{
	&  \scrT_{0,3}(\beta) \ar[rd]_{ev_{\infty}} \ar[ld]^{ev_{0,1}} & \\ X\times X & & X
}
\]
Let 
\[
\kappa: E^*(X) \otimes_{E^*(pt)} E^*(X) \to E^*(X\times X)
\] be the natural cross-product map, and write $D_X: E_*(X) \stackrel{\sim}{\longrightarrow} E^{2n-*}(X)$ for the duality isomorphism coming from the canonical deformation class of almost complex structures on $X$. 
We obtain a map 
\[
\mu_\beta: E^*(X) \otimes_{E^*(pt)} E^*(X) \to E^*(X)
\]
by sending 
\[
a \otimes b \ \mapsto \ D_X\,\left((ev_{\infty})_* [\bT_{0,3}(\beta)] \cap ev_{0,1}^*(\kappa(a\otimes b)) \right).
\]
If $\beta=0$ then all stable maps are constant, the virtual class $[\bT_{0,3}(\beta)] \in E_*(X^3)$ is the class of the small diagonal, and $\mu_{\beta=0}(a\otimes b) = a \smile b$ recovers the usual cup product on $E^*(X)$, compare to \cite[Ch 2, Theorem 2.9]{Rudyak}.

Let $\Lambda_{\bE_*}$ denote the one-variable Novikov ring with formal variable $q$ over the coefficient ring $\bE_* = E^*(pt)$, i.e. the $\bE_*$-module whose elements are  series
\[
\Lambda_{\bE_*} = \left\{ \sum_i a_i q^{t_i} \, | \, a_i \in \bE_*, t_i \in \bR, \lim_i t_i = + \infty \right\}.
\]
The quantum product on $E^*(X;\Lambda_{\bE_*}) = E^*(X) \otimes_{\bE_*}\Lambda_{\bE_*}$  is an $\bE_*$-linear product which `corrects' the (non-associative) naive product
\begin{equation} \label{eq:naive_product}
  (a,b) \mapsto \sum_{\beta} \mu_{\beta}(a\otimes b) \, q^{\omega(\beta)}.  
\end{equation}
We will explain how this correction arises in this section.

\subsection{From the mesosphere\label{Sec:mesosphere}}

The construction of an associative quantum product on $H^*(X;\bE)$ follows the following general template, which in particular applies to the classical cases of ordinary cohomology and $K$-theory. Consider the projection
\[
\st: \scrT_{0,4}(\beta) \to \ccMbar_{0,4} 
\]
which forgets the map and stabilises the domain. The fact that $ \ccMbar_{0,4} $ is connected implies that the expression
\begin{equation} 
	[\st^{-1}(\lambda)] \cap [\bT_{0,4}(\beta)] 
\end{equation}
is independent of $\lambda  \in \bP^1 \backslash \{0,1,\infty\}$.  Following the standard proof of associativity, one specialises  to the (non-generic) case where $\lambda \in \{0,1,\infty\} \subset \bP^1 =  \ccMbar_{0,4}$ represents one of the three reducible points. In this case, the fibre $\st^{-1}(\lambda)$  is a strict normal crossing divisor $D_{\lambda} \subset \scrT_{0,4}(\beta)$ whose top irreducible components are indexed by splittings $\beta = \beta_1+\beta_2$. These divisors intersect so that the codimension $m$ strata are closures of the loci where the corresponding (thickened) stable map has domain which contains a chain of $m\geq 0$ rational curves between the stable components which are contracted to the node in $\bP^1\vee\bP^1$ under the stabilisation map. The axioms of a counting theory express this as
\begin{equation}  \label{eqn:class_of_fibre_of_stab}
    L(\{ \Delta \cap ([\bT_{0,3}(\beta_1)]\times [\bT_{0,3}(\beta_2)]) \}_{\beta_1 + \beta_2 = \beta}),
\end{equation}
where $\Delta$ is the pullback of the appropriate diagonal under evaluation.

If this expression agrees with the sum of the expression for the fundamental class of the fibre product (given by Equation \eqref{eq:pullback_fundamental_class}), then associativity holds. But comparing the two expressions shows that this can only be true whenever $L$ is the additive formal group (or if the higher order terms of the formal group law vanish by some coincidence). The higher order terms are supported on higher order strata, as illustrated in Figure \ref{fig:higher_depth_strata}.

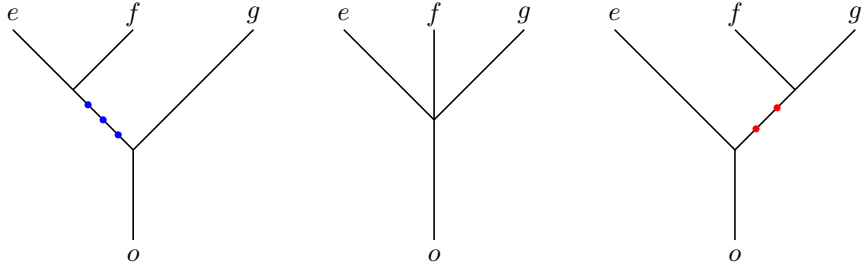
\begin{figure}[ht]
	\begin{center}
		\begin{tikzpicture}[scale=0.4]
			
			\draw[semithick] (0,0) to (-2,2);
			\draw[semithick] (0,0) to (2,2);
			\draw[semithick] (0,0) to (2,-2);
			\draw[semithick] (2,-2) to (6,2);
			\draw[semithick] (2,-2) to (2,-5);
			\draw (-2,2.5) node {$e$};
			\draw (2,2.5) node {$f$};
			\draw (6,2.5) node {$g$};
			\draw (2,-5.5) node {$o$};
			\draw[fill,blue] (0.5,-0.5) circle (0.1);
			\draw[fill,blue] (1,-1) circle (0.1);
			\draw[fill,blue] (1.5,-1.5) circle (0.1);

			\draw[semithick] (12,-1) to (9,2);
			\draw[semithick] (12,-1) to (12,2);
			\draw[semithick] (12,-1) to (15,2);
			\draw[semithick] (12,-1) to (12,-5);
			\draw (9,2.5) node {$e$};
			\draw (12,2.5) node {$f$};
			\draw (15,2.5) node {$g$};
			\draw (12,-5.5) node {$o$};
			
			\draw[semithick] (24,0) to (26,2);
			\draw[semithick] (24,0) to (22,2);
			\draw[semithick] (24,0) to (22,-2);
			\draw[semithick] (22,-2) to (18,2);
			\draw[semithick] (22,-2) to (22,-5);
			
			\draw (18,2.5) node {$e$};
			\draw (22,2.5) node {$f$};
			\draw (26,2.5) node {$g$};
			\draw (22,-5.5) node {$o$};
			\draw[fill,red] (22.7,-1.3) circle (0.1);
			\draw[fill,red] (23.4,-0.6) circle (0.1);

		\end{tikzpicture}
	\end{center}
	
	\caption{Configurations of higher depth in $\overline{\mathcal{M}}_{0,4}(\beta)$; the (red/blue) dots indicate chains of rational curves separating the irreducible components of the stabilisation, whilst the central configuration is schematic for the fibre over general $\lambda \in \overline{\calM}_{0,4}$.}
        \label{fig:higher_depth_strata}
\end{figure}

To obtain an associative product, we then seek `corrected' virtual classes 
\[
[\bT_{0,4}(\beta)]^{corr} \in E^*(\ccMbar_{0,4}(\beta)) \qquad \mathrm{and} \qquad [\bT_{0,3}(\beta_i)]^{corr} \in E^*(\ccMbar_{0,3}(\beta_i))
\]
with the following fundamental property:
\begin{equation} \label{eq:factorisation_corrected_class}
    \parbox{36em}{the pullback of $[\bT_{0,4}(\beta)]^{corr}$ to the $(\beta_1,\beta_2)$-component of the domain (normalization) of a normal crossing boundary fibre \emph{factorises} in the sense that it has the form $([\bT_{0,3}(\beta_1)]^{corr} \times [\bT_{0,3}(\beta_2)]^{corr}) \cap  [\Delta_X]$. }
\end{equation}
In order to also achieve commutativity as well, we shall also require that $[\bT_{0,4}(\beta)]^{corr}$ is invariant under permutations of the first 3 marked points.

We will write $D_{\beta_1,\beta_2}$ for the divisor in $\ccMbar_{0,4}(\beta)$ associated to a splitting $\beta = \beta_1 + \beta_2$.

\begin{Example} 
	If $\bE = H\bQ$ is rational singular cohomology then the uncorrected class $[\bT_{0,4}(\beta)]$ satisfies both requirements. This is because although  the map from the disjoint union of components of the normal crossing fibre is only componentwise-injective, the homology class defined by its image is given by the sum of the (virtual analogues of the) classes $D_{\beta_1,\beta_2} = [\ccMbar_{0,3}(\beta_1) \times_{ev} \ccMbar_{0,3}(\beta_2)]$. 
\end{Example}

\begin{Example} 
	If $\bE = KU$ is complex topological $K$-theory, then the class $\sum [D_{\beta_1,\beta_2}]$ satisfies the second requirement -- its pullback to a component by definition factorises as the product of the virtual classes of the two factors of the fibre product -- but it has no reason to be invariant under the involution exchanging marked points.
\end{Example}

\begin{lemma}
	The class $L_{\bE}(\{[D_{\beta_1,\beta_2}]\}_{\beta_1+\beta_2=\beta}) \in E^*(\ccMbar_{0,4}(\beta))$ is $\sigma_{13}$-invariant.
\end{lemma}

\begin{proof}
	This is immediate from Equation \eqref{eqn:class_of_fibre_of_stab}.
\end{proof}

\subsection{Some universal algebra}

The axioms of a counting theory imply that a (transverse, stably) almost complex codimension one submanifold $D \subset M$ of a (transverse, stably) almost complex $G$-manifold has an associated class $[D] = \tilde{c} \in E^*(M)$ which depends only on the associated complex line bundle. Write $\tilde{c}$ for the first Chern class in the theory $\bE$.

The corrected fundamental classes $[\bT_{0,3}(\beta_1)]^{corr}$  alluded to in the previous section will be obtained by multiplying the naive fundamental class by a power series in the Chern classes associated to boundary divisors corresponding to sphere bubbling at the output marked point. In order to understand the properties which such a power series needs to satisfy, recall that the geometry underlying Equation \eqref{eqn:class_of_fibre_of_stab} is that the general fibre of the map $\st: \ccMbar_{0,4}(\beta) \to \ccMbar_{0,4} = \bP^1$ represents $\st^*\tilde{c}(\mathcal{O}_{\bP^1}(1))$, and 
\[
\mathcal{O}(\st^{-1}(\lambda)) = \bigotimes_{\beta_1+\beta_2=\beta} \mathcal{O}(D_{\beta_1,\beta_2}).
\]
In particular, if one restricts the LHS to any given component $D_{\beta_1,\beta_2}$ one gets the trivial bundle, and hence
\[
\bigotimes_{\beta_1+\beta_2=\beta} \mathcal{O}(D_{\beta_1,\beta_2}) |_{D_{\beta_1,\beta_2}} \cong \mathcal{O}.
\]
Ordering the irreducible components / divisors as $D_0, D_1,\ldots$ one sees
\[
\mathcal{O}(D_i)|_{D_i} = \left ( \bigotimes_{j\neq i} \mathcal{O}(D_j)^{-1} \right)\big|_{D_i}.
\]
Applying $\tilde{c}$ to this identity and recalling that $\tilde{c}(\otimes_j \mathcal{L}_j) = L_{\bE}(\tilde{c}(\mathcal{L}_j))$ for any collection of line bundles, we are led to introduce the following quotient ring. Take \emph{finitely many ordered} variables $D_0, D_1, D_2, \ldots, D_N$, and define the ring $R$ to be the quotient 
\begin{equation} \label{eqn:crazy_ring}
	\bE_*\llbracket D_0,\ldots , D_N \rrbracket / \langle D_j^2 = D_j \cdot L_{\bE}(-_{\bE}D_0, -_{\bE}D_1,\ldots, -_{\bE}D_{j-1},-_{\bE}D_{j+1},\ldots,-_{\bE} D_N),  1 \leq j \leq N \rangle
\end{equation}
Here $L_{\bE}(x_1,\ldots,x_k)$ denotes the power series defining $x_1 +_{\bE} \cdots +_{\bE} x_k$ and the $-_{\bE}$ indicates the inverse with respect to the formal group law (i.e. $L_{\bE}(x, -_{\bE}x) = 0$). For notation, we will write
\[
r_j := D_j^2 - D_j \cdot L_{\bE}(-_{\bE}D_0, -_{\bE}D_1,\ldots, -_{\bE}D_{j-1}, -_{\bE}D_{j+1},\ldots, -_{\bE}D_N)
\]
so the ring $R = \bE_*\llbracket D_0,\ldots , D_N\rrbracket / \langle r_j, \, 1\leq j \leq N\rangle$. The power series that we need to correct the virtual fundamental class arises as follows:

\begin{lemma} \label{lem:crazy_algebra}
	There is a power series $f(D_0,D_1,\ldots, D_N) = 1+\cdots$ for which $L_{\bE}(D_0, D_1,\ldots,D_N)$ is given by
	\[
	D_0 \cdot f(0,\ldots) + D_1 \cdot f(D_0, 0,\ldots) + D_2 \cdot f(D_0, D_1,0,\ldots) + \cdots + D_N \cdot f(D_0,\ldots, D_{N-1}, 0)
	\]
	in the quotient $R$.
\end{lemma}

\begin{proof}
	We will first construct power series $f_0=1, f_1,\ldots, f_N$ for which
	\label{eqn:partial_result}
	\begin{equation}
		L_{\bE}(D_0, D_1,\ldots, D_N) = 1 \cdot D_0 + f_1(D_0)  \cdot  D_1 + f_2(D_0,D_1)  \cdot D_2 + \cdots + f_N(D_0,\ldots,D_{N-1})  \cdot  D_N.
	\end{equation}
	Consider the following algorithm: given any monomial appearing in $L_{\bE}(D_0, D_1,\ldots, D_N)$,  if $j$ is the largest index such that $D_j$ occurs in that monomial, we have two options:
	\begin{itemize}
		\item $D_j$ occurs with multiplicity one, so the monomial is $D_j \cdot \phi$ where $\phi$ is a monomial in $D_0,\ldots, D_{j-1}$;
		\item $D_j$ occurs with multiplicity $>1$ so the monomial is $D_j^{k\geq 2} \cdot \phi$ with $\phi$ a monomial in $D_0,\ldots, D_{j-1}$.
	\end{itemize}
	In the first case, we do nothing, and the term $\phi$ is put in $f_j$.  In the second case, we use the relation for $D_j^2$ to decrease the power with which $D_j$ occurs, more precisely we replace $D_j^k \phi$ by $D_j \cdot (L_{\bE}(-_{\bE}D_0, -_{\bE}D_1,\ldots, -_{\bE}D_{j-1}, -_{\bE}D_{j+1},\ldots, -_{\bE}D_N))^{k-1} \cdot \phi$.  Note that in this expression, every term is divisible by exactly one power of $D_j$. Some of the terms only involve $D_i$ with $i<j$, and they are put into $f_j$, whilst all others involve some monomial $D_\ell$ with $\ell > j$.   These are `pushed down' in the induction and treated later.  We then iterate, inductively with increasing $j$.  Thus, at the first iteration, we consider only terms $D_0^i$, and the coefficient of $D_0$ defines $f_0$ and all terms $D_0^i$ with $i>1$ are pushed down into the next stage; then turn to monomials whose highest index is some power of $D_1$ (some of which come from the original expression, and some from the re-expression of terms $D_0^i$ with $i>1$ in the previous stage). At the final stage, when re-expressing terms of the form $D_N^j \phi$ with $j>1$ we get terms $D_N \cdot   L_{\bE}(-_{\bE}D_0, -_{\bE}D_1,\ldots, -_{\bE}D_{N-1})^{j-1} \cdot \phi$, and the whole expression $L_{\bE}(-_{\bE}D_0, -_{\bE}D_1,\ldots, -_{\bE}D_{N-1})^{j-1} \cdot \phi$ is then incorporated into $f_N$, so the algorithm indeed terminates.
	
	This constructs a sequence of power series $f_j$ satisfying \eqref{eqn:partial_result}. Note the algorithm involves making no choice. Moreover, if 
	\[
	R_N := \bE_*\llbracket D_0,D_1,\ldots, D_N \rrbracket / \langle D_j^2 = D_j \cdot L_{\bE}(-_{\bE}D_0,\ldots, -_{\bE}D_{j-1}, -_{\bE}D_{j+1},  \ldots, -_{\bE}D_N), \ 0 \leq j \leq N\rangle  
	\]
	is the ring introduced above on $N+1$ variables, 
	then there is a natural quotient map $q: R_N \to R_{N-1}$ on setting $D_N = 0$, for which
	\[
	q(L_{\bE}(D_0,\ldots,D_N)) = L_{\bE}(D_0,\ldots,D_{N-1}).
	\]   In the algorithm re-expressing $L_{\bE}(D_0,\ldots,D_N)$ in $R_N$, if a monomial contains $D_N^k$ for some $k>1$ then all the terms occuring when it is replaced are divisible by $D_N$ (since the basic relation in \eqref{eqn:crazy_ring} has that property), so the algorithm is `compatible' with the homomorphism $q$.  Put differently, if one views the algorithm as producing functions $f_j = f_j^{(N)}$ that depend on the number of variables, then $f_j^{(N-1)} = f_j^{(N)}|_{D_N = 0}$. It follows that the $f_j$ are specialisations of a single power series $f$.  
\end{proof}

\begin{remark} Lemma \ref{lem:crazy_algebra} constructs a \emph{distinguished} power series in the sense that the algorithm makes no choice beyond the total ordering of the variables $D_i$.  For the additive formal group, 
	\[
	R_j = D_j^2 - D_j((-D_0) + \cdots + (\widehat{-D_j}) + \cdots + (-D_N)) \ = \ D_j(D_0+\cdots+D_N).
	\]
	Therefore all elements in the ideal generated by the $R_j$ belong to the principal ideal $\langle D_0+\cdots+D_N\rangle$; but clearly an element $f(D_0,\ldots,D_{N-1})\cdot D_N$ does not belong to this ideal.  This shows that the power series $f$ is unique in that case. We conjecture that the solution is unique over the universal formal group.
\end{remark}

\begin{lemma} \label{lem:resolve_partial}
	Suppose that one changes the total order on the variables $\{D_j\}$ by swapping $D_i =: D_{i+1}'$ and $D_{i+1} =: D_i'$. Let $f$ and $f'$ be the power series produced by the algorithm constructed in Lemma \ref{lem:crazy_algebra}.  Then $f = f'$ in the quotient ring $R / (D_i \cdot D_{i+1} = 0)$.
\end{lemma}

\begin{proof}
	Let $D'' := D_i + D_{i+1}$ and consider the ordered variables 
	\[
	D_0 < D_1 < \cdots < D_{i-1} < D'' < D_{i+1} < \cdots < D_N.
	\]  Then both $f$ and $f'$ agree with the output of the algorithm applied to this set of variables, because in the quotient ring $R / (D_i \cdot D_{i+1} = 0)$ the relation $D_i \cdot D_{i+1} = 0$ ensures that $L_{\bE}(D_i, D_{i+1}) = D_i + D_{i+1}$.
\end{proof}

\subsection{Associativity}

Consider $\ccMbar_{0,3}(\beta)$.  For each $\delta \in H_2(X;\bZ)$, there is an associated divisor 
\[
\hat{D}_\delta \subset \ccMbar_{0,3}(\beta)
\]
of curves which (in the top stratum) have a $\delta$-bubble at the output marked point and main component in class $\beta-\delta$.  We enumerate the $\hat{D}_{\delta}$ so that $\hat{D}_i < \hat{D}_j$ if $\omega(\delta_i) > \omega(\delta_j)$ again extended arbitrarily to a total order. We then define the corrected virtual class
\begin{equation} \label{eqn:correct_3}
	[\bT_{0,3}(\beta)]^{corr} :=  [\bT_{0,3}(\beta)]\cdot f(\hat{D}_0,\hat{D}_1,\ldots)
\end{equation}
where $f$ is the power series produced in Lemma \ref{lem:crazy_algebra}. Similarly, we define a correction to the virtual class of $[\bT_{0,4}(\beta)]$ by multiplying by $f(E_0,E_1,\ldots)$ where the $E_i$ enumerate divisors corresponding to bubbling at the output marked point. To prove that these expressions are well-defined, we need:
\begin{lemma} \label{lem:power-series-well-defined} The power series $f$ appearing in Lemma \ref{lem:crazy_algebra} does not depend on the choices made in resolving the partial order to a total order, so depends only on $(X,\omega,J)$ and $\beta$.  
\end{lemma}
\begin{proof}
	Note that if $\omega(\delta) = \omega(\delta')$ then Gromov compactness shows the corresponding divisors $\hat{D}_{\delta}$ and $\hat{D}_{\delta'}$ of broken curves in class $\beta$ are \emph{disjoint}. The result then follows from Lemma \ref{lem:resolve_partial}.
\end{proof}

\begin{remark} \label{remark:dummy_variable} If one adds another variable $D_{\delta''}$  (anywhere in the partial order) and the moduli spaces of $J$-curves in class $\delta''$ is empty, then the resulting corrected virtual class does not change. This follows from the fact that the $f_j$ produced by the algorithm of Lemma \ref{lem:crazy_algebra} specialise correctly under setting variables to zero, i.e. are induced by a single power series $f$. \end{remark}

\begin{lemma} \label{lem:corrected_works} The restriction of the `naive' virtual class $[\bT_{0,4}(\beta)]$ to the boundary stratum associated to a splitting $\beta = \beta_1 + \beta_2$ agrees with 
  \begin{equation} \label{eqn:easier_factorization}
	([\bT_{0,3}(\beta_1)]^{corr} \times [\bT_{0,3}(\beta_2)]) \cap  [\Delta_X].
\end{equation}
\end{lemma}

\begin{proof}
	The argument is summarised in the schematic of Figure \ref{fig:schematic}, where black crosses are the output of an operation with grey dots as inputs, and  blue dots denote locations of contracted rational tails.  The left picture indicates that the divisors of interest in $\ccMbar_{0,4}(\beta)$, when understanding the restriction of its naive virtual class to a normal crossing fibre of the stabilisation map, are those related to bubbling at the central node; the central picture uses these to define a corrected product, with the correction `concentrated at the output'; the right image is the composition of two such operations (this counts configurations which are more involved than those on the left, and amounts to correcting the virtual class of $\ccMbar_{0,4}(\beta)$ by the divisors living at the output). On both left and right, there is no correction at the inputs exchanged by $\sigma_{13}$, which ensures the first of the two required properties.
	
	Elaborating, fix attention on a particular component $D_{\beta_1,\beta_2} \subset \ccMbar_{0,4}(\beta)$ corresponding to the given splitting; say this corresponds to some $D_i$ in our enumeration. Each of the other divisors $D_j$ with $j<i$  meets $D_i$ (if at all) in a divisor whose top open stratum corresponds to broken curves with a single contracted rational curve at the glued marked point, the three components having classes $\beta_1-\delta_j, \delta_j, \beta_2$ for some $\delta_j$ of area $> \omega(\delta_i)$.  Thus, $D_j$ pulls back to the fibre product 
	\[
	\ccMbar_{0,3}(\beta_1) \times_{\Delta} \ccMbar_{0,3}(\beta_2)
	\]
	exactly as $(\hat{D}_j \times [\ccMbar_{0,3}(\beta_2)]) \cap \Delta$. Because the ordering is by area, the restriction of $\hat{D}_j$ to $\hat{D}_i$ with $i<j$ is simply trivial (the intersection is empty), which is the geometric analogue of the fact that the series $f_j$ constructed in Lemma \ref{lem:crazy_algebra} are specialisations of a single function $f$. This yields the required property for \eqref{eqn:easier_factorization}.
\end{proof}

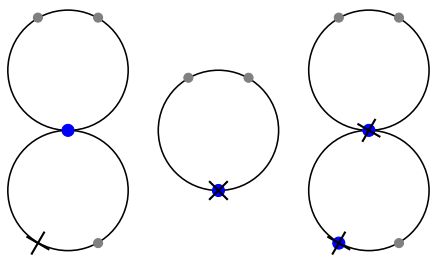
\begin{figure}[ht]
	\begin{center}
		\begin{tikzpicture}[scale=0.4]
			
			\draw[semithick] (-5,2) circle (2);
			\draw[semithick] (-5,-2) circle (2);
			\draw[blue,fill] (-5,0) circle (0.2);
			
			\draw[gray,fill] (-6,3.75) circle (0.15);
			\draw[gray,fill] (-4,3.75) circle (0.15);
			\draw[gray,fill] (-4,-3.75) circle (0.15);
			\path (-6,-3.75) pic[black,thick,rotate=15] {cross=5pt};
			
			\draw[semithick] (+5,2) circle (2);
			\draw[semithick] (+5,-2) circle (2);
			\draw[blue,fill] (+5,0) circle (0.2);
			\path (5,0) pic[black,thick,rotate=15] {cross=5pt};
			
			\draw[blue,fill] (4,-3.75) circle (0.2);
			\draw[gray,fill] (6,3.75) circle (0.15);
			\draw[gray,fill] (4,3.75) circle (0.15);
			\draw[gray,fill] (6,-3.75) circle (0.15);
			\path (4,-3.75) pic[black,thick,rotate=15] {cross=5pt};

			\draw[semithick] (0,0) circle (2);
			\draw[blue,fill] (0,-2) circle (0.2);
			
			\path (0,-2) pic[black,thick] {cross=5pt};
			\draw[gray,fill] (-1,1.75) circle (0.15);
			\draw[gray,fill] (+1,1.75) circle (0.15);

		\end{tikzpicture}
	\end{center}
	\caption{Schematic for corrections to virtual classes\label{fig:schematic}}
\end{figure}

\begin{theorem} \label{thm:associative}
       The terms 
	\[
	(a\ast b)_{\beta} := D((ev_\infty)_* (ev_{0,1}^* (a\otimes b) \cdot [\bT_{0,3}(\beta)]^{corr}))
	\]
	give rise to an associative product on $E^*(X;\Lambda_{\bE_*})$, which is the $\Lambda_{\bE_*}$-linear extension of $(a,b) \mapsto \sum_{\beta} (a\ast b)_{\beta} q^{\omega(\beta)}$.
\end{theorem}

\begin{proof}
	The product $a\ast b$ is graded commutative, because the corrected virtual class is by hypothesis invariant under permuting the first and second marked points. Lemma \ref{lem:corrected_works} implies that the compositions $(a \ast b)\ast c$ and $a \ast (b \ast c)$ can both be identified with a count on $\ccMbar_{0,4}(\beta)$ of curves with a fixed cross-ratio,  after its virtual class has been corrected by the divisors corresponding to bubbling at the output. 
\end{proof}

\begin{lemma} \label{lem:invariance_change_J}
	The quantum product is independent of the choice of $J$.
\end{lemma}

\begin{proof}

  We already have cobordism invariance of the virtual class for the moduli space of curves in any fixed class $\beta$, so it suffices to prove well-definition of the corrected virtual class.
  Take a path $J_t$ all taming $\omega$. We can consider all classes $\beta$ which contain a $J_t$-holomorphic representative for any $t$, and then partially / totally order all possible classes $D_{\beta_1,\beta_2}$ of broken curve configurations in class $\beta$ by the value $\omega(\beta_1)$ which is independent of $t$. The algorithm of Lemma \ref{lem:crazy_algebra} and Definition \ref{defn:corrected} produces a corrected virtual class which will agree, on specialisation to the end-points $t_0, t_1$ of the interval of complex structures, to those associated to $J_{t_0}$ and $J_{t_1}$, by Remark \ref{remark:dummy_variable}. \end{proof}

One can also formulate an invariance statement with respect to changing the symplectic form as follows: associated to a strictly convex cone $\Gamma$ in $H^2(M,\bZ)$ is a completion $R_{\Gamma} \equiv \bE[[H_2(M,\bZ)]]_{\Gamma}$ of the group ring on second-cohomology classes, consisting of series $\sum_{\beta \in H_2(M,\bZ)} a_{\beta} q^{\beta}$ satisfying the property that, for each element $\gamma \in \Gamma$ and each integer $N$, the set of exponents $\beta$ with for which $a_\beta \neq 0$ and $\langle \beta, \gamma \rangle \leq N$ is bounded. Given a symplectic form representing a cohomology class in $\Gamma $, the associated quantum operations are naturally defined with respect to the ring $R_{\Gamma} $:
\begin{prop} 
	The quantum product with coefficient in $R_\Gamma$ depends only on the deformation class of the symmpletic class within the cone $\Gamma$.  
\end{prop}

\begin{proof}   By deforming a path $(\omega_t, J_t)$ of compatible pairs into a concatenation of short segments which are constant in one or the other variable, it is sufficient to show invariance under small deformation of $\omega$ for fixed $J$, since the case of varying $J$ is already covered in Lemma \ref{lem:invariance_change_J}. As we vary $\omega$ from $\omega_0$ to $\omega_1$ with $J$ fixed,   if there are classes $\beta_0$ and $\beta_1$ for which the ordering by area changes, $\omega_0(\beta_0) < \omega_0(\beta_1)$ but $\omega_1(\beta_0) > \omega_1(\beta_1)$, then since we work with fixed $J$ we immediately see that the divisors of broken $J$-curves $D_{\beta_0}$ and $D_{\beta_1}$ in the space of curves in class $\beta$ are disjoint. This reduces us to the case covered by Lemma \ref{lem:resolve_partial}. 
\end{proof}
The above result implies, for example, the invariance of the quantum product under deformations of symplectic forms which tame a given almost complex structure, yielding an invariant of complex manifolds which admit a K\"ahler structure that is independent of the K\"ahler class.

\begin{Example} Take $\bE = KU$. The symplectic small quantum product $\star$ on $K(X)$ is defined by 
	\[
	e \star f = \sum_d (e\star f)_d \, q^{\omega(d)} 
	\]
	where
	\[
	(e\star f)_d = \ev_{3*} (\ev_1^*(e) \otimes \ev_2^*(f) \otimes \sum_{m\geq 0} (-1)^m [D\scrT^m])
	\]
	where $D\scrT^m$ is the codimension $m$ locus in $\ccMbar_{0,3}(d)$ of curves with a length $m$ rational tail at the output marked point. The group $\pi_0\Symp(X)$ acts on $QK^*(X)$ by ring automorphisms; this action was not previously known even in the case of $X$ being smooth algebraic.
\end{Example}

\begin{remark} \label{rmk:another_option} One could define a second (associative) quantum product on $E^*(X;\Lambda_{\bE_*})$, which corrects the classes $[\bT_{0,3}(\beta)]^{naive}$ by contributions from contracted rational tails at the two inputs of the product operation rather than from those at the single output.  Breaking symmetry to choose between input and output corrections seems to be a choice one is free to make universally. \end{remark}

\section{Splitting and operations}\label{Sec:operadic}

The axioms of a counting theory in Section \ref{sec:counting-theories} are well-adapted to deal with strict normal crossing divisors, but less well adapted to general normal crossing divisors (because of the difference in explicitness between Proposition \ref{prop:formal group} and Proposition \ref{prop:general_divisor}). For this reason, we shall give two formulations of the splitting axiom: an implicit general one in the form of Proposition \ref{prop:splitting} below, and an explicit one in the form of Theorem \ref{thm:operad}, asserting the existence of an algebra structure over the truncation of the Deligne-Mumford operad given by moduli spaces of curves with at least one input and one output. In the special cases of ordinary cohomology and $K$-theory, no restriction is required, and one can give an explicit general formula. 

\subsection{Divisors of reducible stable curves}

We fix a genus $g$ and number of marked points $n$, and a non-trivial decomposition 
\[
(g,n) = (g_1,n_1) + (g_2,n_2).
\]  
For each partition $\sigma$ of $\{1,\ldots,n\}$ into complementary subsets of sizes $n_1$ and $n_2$
there is a `gluing' map of compact complex orbifolds
\begin{equation} \label{eqn gluing map}
	\iota_{\sigma} := \iota: \ccMbar_{g_1,n_1+1} \times \ccMbar_{g_2,1+n_2} \to \ccMbar_{g,n}
\end{equation}
which is a finite quotient onto a normal crossing divisor $D^{\sigma}_{g_1,g_2} \subset \ccMbar_{g,n}$.  For notational simplicity we will often suppress the partition $\sigma$ and/or implicitly assume that $\sigma = \{1,\ldots,n_1\} \sqcup \{n_1+1,\ldots,n_1+n_2=n\}$.

\begin{Example}\label{ex:inj}
	The map $\iota$ need not be generically injective: when $n_1=0=n_2$ and $g_1=g_2$ the map factors through the second symmetric product $\Sym^2(\ccMbar_{g_1,1})$. 
\end{Example}

\begin{Example}\label{ex:not_snc}
	The map $\iota: \ccMbar_{g,1} \times \ccMbar_{g,2} \to \ccMbar_{2g,1}$ is generically injective.  However, there is a divisor $\ccMbar_{g,1} \cong D_{red} \subset \ccMbar_{g,2}$ of curves with a $\bP^1$-bubble carrying both marked points, and the map $\iota$ restricted to $\ccMbar_{g,1} \times D_{red} \cong \ccMbar_{g,1} \times \ccMbar_{g,1}$ factors through the second symmetric product of $\ccMbar_{g,1}$. The image of $\iota$ is normal crossing but not strict normal crossing.
\end{Example}

\subsection{Reducible framed and thickened curves}

\begin{lemma} \label{lem:submerse_to_domains} Let $\scrF:=\scrF_{g,n}(\bP^r,d)$ denote any Zariski open subset of the moduli stack of degree $d$ stable maps to projective space $\bP^r$  of curves contained in the locus of curves $u$ where $H^1(u^*O(1)) = 0$. Then $\scrF$ is smooth and the divisor of singular stable curves meets $\scrF$ in a normal crossing divisor.
\end{lemma}

\begin{proof}
	The Euler sequence shows that $H^1(u^*T\bP^r) = 0$, which ensures that $\scrF$ is smooth. The stronger cohomological condition $H^1(u^*O(1)) = 0$ implies that locally near the curve $u: C \to \bP^r$, the moduli space $\scrF$ is smooth over the versal deformation space of the domain $C$, because cohomology vanishing and global generation are both Zariski open conditions.  This means the forgetful morphism from $\scrF$  to the stack $\frak{M}_{g,n}$ of all prestable curves is a smooth morphism in the sense of  \cite[\href{https://stacks.math.columbia.edu/tag/075U}{Definition 075U}]{stacks-project}  (so locally analytically the map to the versal deformation space, which is given by the product of the deformation spaces over the nodes of $C$, is a submersion). The stabilisation map $\st: \frak{M}_{g,n} \to \ccMbar_{g,n}$ is dominant (essentially by definition)  and flat \cite[Proposition 3]{Behrend}. It is also locally toroidal (`log smooth', i.e. locally equivalent to a toric morphism in the sense of \cite[Definition 3.3.3]{CLS}) because it is an iterated prestable curve fibration. 
	A map of smooth schemes which is flat and (log) smooth pulls back a normal crossing divisor  to something with normal crossings, see e.g.  \cite[\href{https://stacks.math.columbia.edu/tag/0CBQ}{Lemma 0CBQ}]{stacks-project} and \cite{Abramovich-Karu}. The result now follows from the fact that the boundary of Deligne-Mumford space is normal crossing, and that $\scrF$ is a smooth scheme.
\end{proof}

\begin{lemma}
	Define $\scrF_{g_1,g_2}^{\,\square}$ by the pullback diagram
	\[
	\xymatrix{
		\scrF_{g_1,g_2}^{\,\square}  \ar[r]_-{\iota_\scrF} \ar[d] & \scrF_{g,n}(\bP^{d-g},d)\ar[d] \\
		\ccMbar_{g_1,n_1+1} \times \ccMbar_{g_2,1+n_2}  \ar[r]_-{\iota_{\sigma}}  &  \ccMbar_{g,n} 
	}
	\]
	Then $\iota_\scrF$ is a closed immersion; $\iota_\scrF (\scrF_{g_1,g_2}^{\,\square}) \subset \scrF_{g,n}(\bP^{d-g},d)$ is a divisor with normal crossings; and $\scrF_{g_1,g_2}^{\,\square}$ is its normalisation and is smooth.
\end{lemma}

\begin{proof} Note that the pullback of a normal crossings divisor is normal crossings. \end{proof}

The previous diagram induces another pullback diagram
\[
\xymatrix{
	\scrT_{g_1,g_2}^{\,\square} \ar[r] \ar[d] & \scrT_{g,n}(\beta)  \ar[d] \\
	\scrF_{g_1,g_2}^{\,\square} \ar[r]_-{\iota_\scrF}   &   \scrF_{g,n}(\bP^{d-g},d)
}
\]
(now in the category of global charts, so we also pull back the obstruction space and section in the top row). Since the map $\scrT_{g,n} \to \scrF_{g,n}$ is a submersion, the image of the  top arrow again has `normal crossing singularities' in an obvious sense, though this now takes place in the category of $G$-transverse almost complex manifolds, and $\scrT_{g_1,g_2}^{\,\square} $ plays the role of the normalisation of that image, and is a union of smooth manifolds.

For each $r \geq 0$,
let $\scrF_{g_1,g_2}^{\,\square,r} \subset \scrF_{g_1,g_2}^{\,\square}$ be the subspace consisting of those curves in which the stabilization map
to $\ccMbar_{g,n}$ contracts a chain of $r$
irreducible components to the distinguished node.
Let $\scrT_{g_1,g_2}^{\,\square,r} \subset \scrT_{g_1,g_2}^{\,\square}$ be the preimage of $\scrF_{g_1,g_2}^{\,\square,r}$ under the natural map.

\begin{lemma} \label{Lem:splitting_divisor_in_thickening}	
	For $r\geq 0$, the maps $\iota_\scrF|_{\scrF_{g_1,g_2}^{\,\square,r}}$ and $\iota_\scrT|_{\scrT_{g_1,g_2}^{\,\square,r}}$ are $(r+1):1$
	covering maps over their image in $\scrF_{g,n}(\C P^d,d)$ and $\scrT_{g,n}(\C P^d,d)$ respectively.
\end{lemma}

\begin{proof}
	Consider an element of $\scrT_{g,n}(\beta)$ which lies over the divisor $D^{\sigma}_{g_1,g_2}$. The domain $\Sigma$ contains a chain of $r \geq 0$ rational curves which are contracted to the distinguished node by the stabilisation map (these rational components, being unstable, contain no other marked point). If a stable curve in $D^{\sigma}_{g_1,g_2} \subset \ccMbar_{g,n}$ is described as the image of another curve under the gluing map $\iota$ \eqref{eqn gluing map},
	then it has a distinguished node corresponding to the gluing region.
	The key point is that this chain of rational curves contains $r+1$ nodes that separate $\Sigma$ into a genus $g_1$ and $g_2$ nodal curve respectively.
\end{proof}

Combining the previous discussions, we have a diagram
\[
\xymatrix{
	\bT_{g_1,g_2}^{\,\square} \ar[r]_-{\psi_\scrT} \ar[d]_{\pi_0} & \bT_{g,n}(\beta)  \ar[d]_{\pi_1} \\
	\scrF_{g_1,g_2}^{\,\square} \times X^n \ar[r]_-{\psi_\scrF} \ar[d]_{p_0}  &   \scrF_{g,n}(\bP^{d-g},d) \times X^n \ar[d]_{p_1} \\
	\ccMbar_{g_1,n_1+1} \times \ccMbar_{g_2,1+n_2} \times X^n \ar[r]_-{\psi_{\sigma}}  &  \ccMbar_{g,n}  \times X^n
}
\]
where the first line indicates that we have a pullback in the category of global charts, defined by a $G$-map on thickenings which is compatible with obstruction bundles.  We can write $\scrT_{g_1,g_2}^{\,\square}$ as a finite disjoint union
\begin{equation} \label{eqn:component_strata}
	\scrT_{g_1,g_2}^{\,\square} = \bigcup_{\beta_1+\beta_2 = \beta} \, \scrT_{g_1,g_2}^{\,\square} (\beta_1,\beta_2)
\end{equation}
according to the decomposition of the homology classes, where finiteness follows from Gromov compactness. 

\subsection{Calculus of split curves}

We next identify the virtual class for a smooth connected component $\bT_{g_1,g_2}^{\,\square}(\beta_1,\beta_2)$  of $\bT_{g_1,g_2}^{\,\square}$ in terms of products of classes of simpler moduli spaces.

Let $\psi_{\scrT,\sigma}$ denote the natural map
\[
\psi_{\scrT,\sigma}: \scrT_{g_1,n_1+1}(\beta_1) \times \scrT_{g_2,1+n_2}(\beta_2) \to \scrT_{g,n}(\beta),
\]
whose domain is the thickening for a global chart presenting the product moduli space
\[
\ccMbar_{g_1,n_1+1}(\beta_1) \times \ccMbar_{g_2,1+n_2}(\beta_2).
\]

Passing to $E_*$-homology, we note that the monoidal structure induces a map
\[
E_*(\scrT_{g_1,n_1+1}(\beta_1)) \otimes_{E_*(pt)} E_*(\scrT_{g_2,1+n_2}(\beta_2)) \longrightarrow E_*(\scrT_{g_1,n_1+1}(\beta_1) \times \scrT_{g_2,1+n_2}(\beta_2)),
\]
and that the diagonal $\Delta \subset X\times X$ defines a class in $E^*(X\times X)$.  Let $\Delta_{\sigma} \subset E^*(\scrT_{g_1,n_1+1}(\beta_1) \times \scrT_{g_2,1+n_2}(\beta_2))$ denote the pullback of  $\Delta_X \subset X\times X$ under the evaluation to $X^{n_1+2+n_2} \to X^2$ which picks out the two factors indexed by the glued marked points.

\begin{prop}\label{prop:square_as_product}
	For the relevant component of \eqref{eqn:component_strata} we have 
	\[
	[\bT_{g_1,g_2}^{\,\square}] = (\psi_{\scrT,\sigma})_*([\bT_{g_1,n_1+1}(\beta_1)]  \otimes [\bT_{g_2,1+n_2}(\beta_2)]) \cap \Delta_{\sigma} \quad \in E_*(\bT_{g,n}(\beta)).
	\]
\end{prop}

\begin{proof}
	By the equivalence of the two charts constructed in Section \ref{sec:split_curves} 
	we have that $\bT^{\square}_{g_1,g_2}$
	is isomorphic to a global Kuranishi chart $\bT_{\bL}$
	which is the pullback of the product
	chart $\bT_{g_1,n_1+1} \times \bT_{g_2,n_2+2}$
	of the diagonal $\Delta_X \subset X \times X$
	under the evaluation map corresponding to a choice of marked points on each curve.
\end{proof}

\subsection{Implicit splitting axiom}

The Gromov-Witten correlator for the class $\beta$ is by definition
\[
I_{g,n}(\beta) := (p_1\circ \pi_1)_* [\bT_{g,n}(\beta)] \in E_*(\ccMbar_{g,n} \times X^n).
\] 
Write $D_{\sigma} \in E^*(\ccMbar_{g,n} \times X^n)$ for the class dual to the divisor which is the image of $\psi_{\sigma}$.  If the map $\psi_\scrF$ defines a strict normal crossing divisor, or the counting theory $\bE$ is either $H\bQ$ or $KU$, 
we can apply Axioms \ref{pullback axiom} and \ref{normal crossing axiom} of a counting theory respectively to the upper and lower squares in this diagram. In the upper part, we obtain
\begin{equation} \label{eqn:upstairs_1}
	(\psi_{\scrT})_*[\bT_{g_1,g_2}^{\square}(\beta_1,\beta_2)] = [\bT_{g,n}(\beta)] \cap \pi_1^* (\psi_{\scrF})_*\left({ \sum_{\beta_1 + \beta_2 = \beta }} [\scrF^{\square}_{g_1,g_2}(\beta_1,\beta_2)]\right),
\end{equation}
while the lower part yields
\begin{equation} \label{eqn:upstairs_2}
	p_1^*(D_{\sigma}) = L\,(\{(\psi_\scrF)_*[\scrF^{\square}_{g_1,g_2}(\beta_1,\beta_2)]\}_{\beta_1 + \beta_2 = \beta}).
\end{equation}
Here the formal group law $L(\bullet)$ is applied to the collection of (first Chern classes of) irreducible divisors indicated, indexed by the splittings $\beta = \beta_1 + \beta_2$ compatible with the splitting of genus and marked points. In the case where $\psi_{\scrF}$ defines a non-strict normal crossing divisor, the analogue of \eqref{eqn:upstairs_2} is (even more) implicit, since it relies on Proposition \ref{prop:general_divisor} as indicated in Remark \ref{rem:axiom_3_variation}. We obtain a preliminary version of the splitting axiom: as above fix a decomposition $g=g_1+g_2$ and associated partition $\sigma$ of the $n=n_1+n_2$ marked points defining a divisor $D_{\sigma} \subset \ccMbar_{g,n}$. 

\begin{proposition}\label{prop:splitting}
	The cap-product $[\bT_{g,n}(\beta)] \cap (\st\times ev)^* (D_{\sigma} \times X^n)$ is determined, via the formal group $L$ and the cone complex to the boundary divisor,  by the tautological modules associated to the spaces with virtual classes $[\bT_{g_1,n_1+1}(\beta_1)]$ and $ [\bT_{g_2,1+n_2}(\beta_2)]$ indexed by decompositions $\beta = \beta_1+\beta_2 \in H_2(X;\bZ)$. 
\end{proposition}

\begin{proof} This follows on combining Proposition \ref{prop:square_as_product} with Equations \eqref{eqn:upstairs_1} and \eqref{eqn:upstairs_2}.
	Recall that, for any fixed $J$, only finitely many decompositions are realised by stable maps, so the formal group law is only applied to finite collections of irreducible divisors.  The codimension $r$ strata in  the nc-divisor $\iota_{\scrF}(\scrF_{g_1,g_2}^{\,\square})$ exactly correspond to stable maps in which a length $r$ chain of rational curves was contracted to the distinguished node. 
	Expanding out the formal group,  \eqref{eqn:upstairs_2} gives a formula for $[\bT_{g,n}(\beta)] \cap \pi^*D_{\sigma}$ as a weighted sum over virtual classes indexed by
	\begin{itemize}
		\item  tuples  $(\beta_1,\underline{\gamma}, \beta_2)$ with $\beta_1 + \sum \gamma_j + \beta_2 = \beta$, where the $\gamma_i$ index rational curves contracted at the distinguished node, and in particular involves all invariants of the form $[\bT_{0,2}(X,\gamma_i)]$ (capped against suitable multidiagonals); and
		\item elements of the Picard group of the moduli spaces of stable maps in these classes, coming from the restriction to some stratum $D_J$ of the various line bundles $\mathcal{O}(D_i)$ and their powers (which appear from the formal group expansion).
	\end{itemize}
	These are by definition cap-products against appropriate multi-diagonals of elements of the tautological modules for  $\beta_1$ and $\beta_2$.
\end{proof}

\begin{example} \label{ex:cotangent _line} The normal (orbifold line) bundle to the image of 
	\[
	\iota_{\sigma}: \ccMbar_{g_1,n_1+1} \times \ccMbar_{g_2,1+n_2} \to \ccMbar_{g,n_1+n_2}
	\]
	is given by the tensor product of the line bundles $T_{\pm}^*$ defined by the cotangent line classes at the two marked points $p_{\pm}$ glued to give the distinguished node at a curve in the target, so
	\[
	\bigotimes_{\beta_1+\beta_2=\beta} \mathcal{O}(\scrT_{g_1,g_2}^{\square}(\beta_1,\beta_2)) \, \cong T_{+}^* \otimes T_{-}^*.
	\]
	Therefore the tautological integrals arising from expanding the formal group will include terms which replace the virtual class of a stratum with its cap-product with powers of cotangent line classes. 
	\end{example}

\begin{remark} \label{rmk:dont_push}
	For the additive or multiplicative formal group, the expansion of $L(\bullet)$ contains only one term on any given stratum of $(\st \times ev)^{-1}(D_{\sigma})$, and one can push forward \eqref{eqn:upstairs_2}  to $E_*(\ccMbar_{g,n} \times X^n)$ to obtain the `usual' splitting theorems.  In general there is no natural way to describe the terms in the sum after pushforward, since they involve tautological integrals which are not cap products of  fundamental classes of strata with classes pulled back from the base. 
\end{remark}

Remark \ref{rmk:dont_push} implies that Proposition \ref{prop:splitting}  does not yield an identity in $E_*(X^n \times \ccMbar_{g,n})$, and does not give a recursive way of computing  the genus $g$ Gromov-Witten class for $\beta$ in terms of lower degree and genus invariants except when $\bE = H^*(-,\bQ)$ or $\bE = KU^*(-)$. Rather, it is recursive only if one knows the entire tautological modules, in the sense of Section \ref{sec:tautological}, for the moduli spaces at lower genus and degree.

\subsection{Operadic corrections\label{Sec:operad}}

Consider moduli spaces of curves with $1+1$ marked points, one being distinguished as an input and one as an output. The virtual class $[\bT_{g,1+1}(\beta)]$ defines an endomorphism 
\[
\Phi_{g,\beta}: E^*(X) \to E^*(X), \qquad \lambda \mapsto D \left( (ev_{out})_*(ev_{in}^* \lambda \cap [\bT_{g,1+1}(\beta)]) \right)
\] 
by pull-push ($D$ denotes duality $E_*(X) \to E^*(X)$ as usual). Just as the naive virtual classes for spaces $[\bT_{0,3}(\beta)]$ do not lead to an associative quantum product, the naive classes at higher genus do not give $E^*(X)$ the structure of an algebra over the Deligne-Mumford operad, the simplest failure of which is that gluing of moduli spaces does not match compositions of the maps $\Phi_{g,\beta}$.  Since we are now considering operations, divisors of split curves have strict normal crossings (since components carry ordered input and output marked points), which makes the geometry underlying Proposition \ref{prop:splitting} sufficiently explicit to correct the higher genus virtual classes to restore the operadic structure. 

Recall that if $D_{g_1,g_2} \subset \ccMbar_{g,n}$ is the divisor associated to gluing
\[
\ccMbar_{g_1,n_1+1} \times \ccMbar_{g_2,1+n_2} \to \ccMbar_{g,n}
\]
then the normal bundle $\mathcal{O}(D_{g_1,g_2}) = T_{in}^* \otimes T_{out}^*$ is the tensor product of the cotangent line bundles at the distinguished marked points.  Follow the same recipe as in Lemma \ref{lem:crazy_algebra}, we have:

\begin{lemma} \label{lem:crazier_algebra}
	Let $R'$ denote the quotient ring 
	\begin{equation} 
		R'\,  := \, \bE_*\llbracket S,T, D_0,D_1,\ldots \rrbracket / \langle D_i^2 = D_i \cdot L_{\bE}(S,T, -_{\bE}D_0, -_{\bE}D_1,\ldots, -_{\bE}D_{i-1}, -_{\bE}\widehat{D}_i, -_{\bE}D_{i+1},\ldots) \rangle
	\end{equation}
	Then there is a power series $F(D_0,\ldots)$, with coefficients in $\bE_*[[S,T]]$, so that
	\[
	L_{\bE}(D_0,\ldots,D_N) = D_0 \cdot F(0) + D_1 \cdot F(D_0, 0,\ldots) + \cdots \ \quad \in R'
	\]
	and $F(0) = 1$. \qed
\end{lemma}

\begin{defn} A \emph{graded} virtual class for $\ccMbar_{g,n+1}(\beta)$ with respect to a global chart $(G,\scrT,E,s)$  is a virtual class in $E^*(\scrT_{g,n+1}(\beta))\llbracket w_{1}, \ldots, w_n,  z \rrbracket$ with the property that the coefficient of each monomial belongs to the tautological module of the thickening.
\end{defn}

We think of the formal variables as attached to the input and output marked points, in an obvious way. Given an element
\begin{equation}
	\lambda =  \sum_{\alpha \in \bN^n} \lambda_\alpha w^\alpha \in E^*(\ccMbar_{g,n+1} \times  X \times \ldots \times X)\llbracket w_1, \ldots, w_n \rrbracket 
\end{equation}
and a graded virtual class $\sum_{\alpha} [M_{\alpha i}] w^\alpha z^i$ for $\bT_{g,n+1}(\beta)$, we obtain an element of $E^*(X)\llbracket z \rrbracket$ as a sum
\begin{equation}
	\sum_{i \geq 0} Dev_* (ev^*\lambda \cap [M_{\alpha i}]) z^i.
\end{equation}
This uses  a residue-type pairing of the input cohomology class formal variable $w_i$ against the input variable $z_i$ (i.e. view $z_{i} = w_i^{-1}$, and extract constant terms of the product).  Assuming convergence, extending linearly yields a Gromov-Witten operation
\begin{equation}
	E^*(\ccMbar_{g,n+1}) \otimes \left( E^*(X)\llbracket z \rrbracket \right)^{\otimes n} \to E^*(X)\llbracket z \rrbracket.  
\end{equation}
In the examples below, the linear extension will be well-defined (all coefficients finite) because of finiteness of contributions coming from Gromov compactness.

Returning to Lemma \ref{lem:crazier_algebra}, view $T = T^*_{out}$ be the cotangent line at the input to the second component of genus $g_2$ whilst $S = T^*_{in}$ is that at the output of the first component of genus $g_1$, and the $D_j$ are indexing divisors of broken curves according to decompositions $\beta = \beta_0 + \beta_1$ ordered intrinsically by area of $\beta_0$. Now expand 
\[
F = \sum_{i\geq 0} F_i(T^*_{out},D_0,\ldots) (T^*_{in})^i
\]
in powers of $S = T^*_{in}$. (Breaking symmetry between $S$ and $T$ is analogous to the symmetry breaking by choosing outputs rather than inputs in Remark \ref{rmk:another_option}.)  

\begin{defn}
	The corrected virtual class of a moduli space of higher genus maps is the series
	\begin{equation} \label{eqn:correct_in_higher_genus}
		[\bT_{g,n+1}(\beta)]^{corr} := \sum_{i \geq 0} [\bT_{g,1+1}(\beta)] \cdot \, F_i(T_{out}^*, \hat{D}_0, \hat{D}_1,\ldots) \cdot z^i \cdot \prod_{j = 1}^{n} \frac{1}{1-T^*_{j,in} \cdot w_i}  ,
	\end{equation}
	considered as an element of $E_*(\bT) \llbracket w_1, \ldots, w_n, z \rrbracket$
\end{defn}

\begin{theorem}\label{thm:operad} The corrected graded virtual classes given by Equation \eqref{eqn:correct_in_higher_genus} equip $E^*(X)\llbracket z \rrbracket$ with the structure of an algebra over the $E$-homology of the Deligne-Mumford operad $\ccMbar_{g,n+1} $ (with $0 < n$).
\end{theorem}

\begin{proof}[Sketch]
	Using the module structure of cohomology over homology, it suffices to prove the result for the fundamental class of the moduli spaces. The \emph{naive} virtual class $[\bT_{g,n+1}(\beta)]$ 
	has the property that if $\iota: D_{(g_1,n_1)\circ_k (g_2,n_2)}(\beta_1,\beta_2) \to \bT_{g,n+1}(\beta)$ is the inclusion of one irreducible stratum in the boundary divisor over the divisor $D_{(g_1,n_1) \circ_k (g_2,n_2)} \subset \ccMbar_{g,n+1}$ of split curves for which gluing takes place at the $k$th marked point, then
	\begin{equation} \label{eqn:factorise_at_higher_genus}
		[\bT_{g,n +1}(\beta)] \cap D_{(g_1,n_1)\circ_k (g_2,n_2)}(\beta_1,\beta_2) = [\bT_{g_1,n_1+1}(\beta_0)]^{out-corr} \bullet [\bT_{g_2,n_2+1}(\beta_1)]^{in-corr}
	\end{equation}
	factorises into corrections of the corresponding virtual classes which involve data intrinsically defined on the smaller moduli spaces and which are `concentrated' at the output respectively input, and where $\bullet$ is the previously introduced residue pairing.  More precisely, 
	\[
	[\bT_{g,n+1}] \cap \st^*D_{(g_1,n_1)\circ_k (g_2,n_2)} = [\bT_{g,n+1}] \cap L_{\bE}(\{D_{(g_1,n_1)\circ_k (g_2,n_2)}(\beta_0,\beta_1)\})
	\]
	by Proposition \ref{prop:formal group} and Example \ref{ex:cotangent _line},
	and so Lemma \ref{lem:crazier_algebra} writes this formal group expression as 
	\[
	\sum F_i(T^*_{out},D_0, D_1,\ldots) (T^*_{in})^i
	\]
	where the $D_i$ label the divisors $D_{(g_1,n_1)\circ_k (g_2,n_2)}(\beta_0,\beta_1)$ ordered by $\omega(\beta_0)$. We can therefore set
	\[
	[\bT_{g_1,n_1+1}(\beta_0)]^{out-corr} = \sum_{i\geq 0} [\bT_{g_1,n+1}(\beta_0)] F_i(T^*_{out}, \hat{D}_0, \hat{D}_1,\ldots) \cdot z^i
	\]
	and 
	\[
	[\bT_{g_2,1+1}(\beta_1)]^{in_k-corr} = \sum_{j \geq 0} [\bT_{g_2,1+1}(\beta_1)] (T^*_{in})^j w^j_k = [\bT_{g_2,1+1}(\beta_1)]\cdot \frac{1}{1-T^*_{in} w_k}
	\]
	to ensure that Equation \eqref{eqn:factorise_at_higher_genus} is satisfied, where the $\hat{D}_i$ are the divisors of broken curves with a bubble at the output marked point, ordered intrinsically by $\omega$-area of the bubble.   Universality of the construction shows that  \eqref{eqn:factorise_at_higher_genus} is sufficient to imply compatibility with composition when defining the corrected virtual class as in \eqref{eqn:correct_in_higher_genus} by incorporating all input and output corrections at once. 
\end{proof}

\subsection{Genus reduction}

The image of the gluing map $\ccMbar_{g-1,n+2} \to \ccMbar_{g,n}$ defines a divisor which self-intersects rather than being strict normal crossings. There is again a diagram of pullback squares
\[
\xymatrix{
	\scrT_{g-1,n+2}^{\,\square} \ar[r]_-{\iota_\scrT} \ar[d] & \scrT_{g,n}(\beta)  \ar[d] \\
	\scrF_{g-1,n+2}^{\,\square} \times X^n \ar[r]_-{\iota_\scrF} \ar[d] & \scrF_{g,n}(\bP^{d-g},d) \times X^n\ar[d] \\
	\ccMbar_{g-1,n+2} \times X^n \ar[r]_-{\iota_{\sigma}}  &  \ccMbar_{g,n} \times X^n
}
\]
and now the maps $\iota_{\scrF}$ and $\iota_{\scrT}$ have smooth irreducible domains but images in which $r$ sheets self-intersect along the locus of maps where a chain of $r-1$ rational curves are contracted under stabilisation at the distinguished node. The `genus reduction axiom' for the virtual fundamental class by necessity appeals to the Equation \eqref{eqn:nonstrict_nc_divisor}, which is not explicit cf. Remark \ref{rem:axiom_3_variation}. This again expresses
\[
[\bT_{g,n}(\beta)] \cap \st^*D_{\sigma}
\]
for $D_{\sigma} \in E^*(X^n \times \ccMbar_{g,n})$ dual to the image of $\iota_{\sigma}$ above  in terms of tautological classes capped against $[\bT_{g+1,n-2}(\beta')]\cap\Delta_{\sigma}$, where $\Delta_{\sigma}$ is the diagonal at the two marked points, and the $\beta'$ range over classes with $\beta' + \underline{\gamma} = \beta$ for a rational chain $\underline{\gamma}$ at the marked point. The coefficients of the simpler invariants will now depend on the formal group law $L(\bullet)$ of the theory and on the cone complex of  $D_{\sigma}$. The cases of $H\bQ$ and $KU$ are simpler, stemming from Examples \ref{ex:HQ_easier} and \ref{ex:divisor_Ktheory}, and the `usual' genus reduction axiom (as in the algebraic case) holds for these theories.

\appendix

\section{Gluing and comparisons of charts}
\label{sec:gluing-comp-charts}

	\subsection{Parameterized Gluing Theorem}\label{Subsec:gluing}
	
	Let $(X,J)$ be a smooth almost complex manifold and let $\beta \in H_2(X)$.
	Let $g, h \in \bN$ and let $\ccMbar \lra{} \ccMbar_{g,h}$ be a smooth \'{e}tale map (in the orbifold sense)
	from a smooth manifold and define
	a family of curves $\ccCbar \lra{} \ccMbar$ via the pullback square:
	\[\begin{tikzcd}
		\ccCbar & {\ccCbar_{g,h}} \\
		\ccMbar & {\ccMbar_{g,h}}.
		\arrow[from=1-2, to=2-2]
		\arrow[from=1-1, to=1-2]
		\arrow[from=1-1, to=2-1]
		\arrow[from=2-1, to=2-2]
	\end{tikzcd}\]
	We let $\ccCbar^o \subset \ccCbar$ be the region where the fiber is smooth (i.e. the complement of the nodes)
	and let $Y_{\ccCbar} := \Omega^{0,1}_{\ccCbar^o/\ccMbar} \otimes_\bC TX$ 
	be the bundle over $\ccCbar^o \times X$ whose fiber over $(c,x)$
	is the space of linear anti-holomorphic maps from the vertical tangent space of $\ccCbar^o$ to $(TX,J)$.
	Let $U$ be a smooth manifold, $W$ a real vector space and let
	\begin{equation}
		\lambda : W \lra{} C^\infty_c(U \times \ccCbar^o \times X, P^*Y_{\ccCbar})
	\end{equation}
	be a linear map where $P : U \times \ccCbar^o \times X \lra{} \ccCbar^o \times X$ is the natural projection map.
	For each $\nu \in \ccMbar$, let $\ccCbar|_\nu$ be the fiber of $\ccCbar$ over $\nu$.
	Let
	\begin{equation}
		\ccMbar^{\reg}(X) := \left\{
		\begin{array}{l|l}
			s = (a,\nu) \in U \times \ccMbar & u \ \textnormal{is smooth and} \ u_*(\ccCbar|_\nu) = \beta \\ 
			u : \ccCbar|_\nu \lra{} X & \overline{\partial}(u(\cdot)) + \lambda(w)(a,\cdot,u(\cdot)) = 0 \\
			w \in W & u \ \textnormal{is regular} \\
		\end{array} 
		\right\}.
	\end{equation}
	Let $\Pi : \ccMbar^{\reg}(X) \lra{} U \times \ccMbar$
	be the natural projection map sending $(s,u,w)$ to $s$. The following result can be proved using the comparison with slice charts discussed in Section \ref{sec:comp-glob-kuran} below, but it can also be directly extracted from Pardon's account of gluing \cite[Appendix C]{Pardon2016algebraic}.
	
	\begin{theorem} \label{theorem parameterized gluing}
		For each $(s,u,w) \in \ccMbar^{\reg}(X)$, $s = (a,\nu)$, and any small neighborhoods
		$K \subset \Pi^{-1}(\Pi(s,u,w))$ of $(s,u,w)$ and
		$Q \subset U \times \ccMbar$ of its image under $\Pi$, there is a continuous map $g$ fitting into the diagram:
				\[\begin{tikzcd}
			{Q \times K} & {\ccMbar^{\reg}(X)} \\
			Q & {U \times \ccMbar}
			\arrow["\Pi", from=1-2, to=2-2]
			\arrow["g", from=1-1, to=1-2]
			\arrow["j"', hook, from=2-1, to=2-2]
			\arrow["proj"', from=1-1, to=2-1]
		\end{tikzcd}\]
		where $j$ is the natural inclusion map and $g$ is
		an open embedding.
		The restriction of $g$ to each fiber of the projection map $Q \times K \lra{} Q$ is smooth and it is equivariant under the natural action on $K$ given by the stabilizer group of the image of $\nu$ in $\ccMbar_{g,h}$ if $K$ is equivariant under this group.
		The choices involved in constructing such a product neighborhood $g$ ensure that if $g'$ was another such neighborhood then $g$ and $g'$ are $C^1_{loc}$-compatible as in \cite[Definition 4.27]{AMS-Hamiltonian}. \qed
	\end{theorem}
	
	The fibre $ \Pi^{-1}(\Pi(s,u,w))$ is a moduli space of curves with a fixed domain, so is a smooth manifold by the regularity assumption on the maps $u$ in $\ccMbar^{reg}(X)$ and classical elliptic regularity; taking $K$ to be a ball, one sees in particular that $\ccMbar^{\reg}(X)$ is a topological manifold.
	
	We wish to generalize the above theorem to allow for families of curves
	that do not necessarily submerse onto $\ccMbar_{g,h}$.
	Let $P : \ccMbar' \to \ccMbar_{g,h}$ be a smooth map from a smooth manifold and define
	a family of curves $\ccCbar' \lra{} \ccMbar'$ via the pullback square:
	\[\begin{tikzcd}
		{\ccCbar'} & {\ccCbar_{g,h}} \\
		{\ccMbar'} & {\ccMbar_{g,h}}
		\arrow[from=1-2, to=2-2]
		\arrow[from=1-1, to=1-2]
		\arrow[from=1-1, to=2-1]
		\arrow[from=2-1, to=2-2]
	\end{tikzcd}\]
	Let ${\ccCbar'}^o \subset \ccCbar$ the complement of all nodes
	and let $Y_{\ccCbar'} := \Omega^{0,1}_{{\ccCbar'}^o/\ccMbar'} \otimes_\bC TX$.
	Let 
	\begin{equation}
		\lambda' : W' \lra{} C^\infty_c({\ccCbar'}^o \times X, Y_{\ccCbar'})
	\end{equation}
	For each $\nu \in \ccMbar'$, let $\ccCbar'|_\nu$ be the fiber of $\ccCbar'$ over $\nu$.
	Let
	\begin{equation}
		{\ccMbar'}^{\reg}(X) := \left\{
		\begin{array}{l|l}
			\nu \in \ccMbar' & u \ \textnormal{is smooth and} \ u_*(\ccCbar'|_\nu) = \beta \\ 
			u : \ccCbar'|_\nu \lra{} X & \overline{\partial}(u(\cdot)) + \lambda'(w)(\cdot,u(\cdot)) = 0 \\
			w \in W & u \ \textnormal{is regular} \\
		\end{array} 
		\right\}.
	\end{equation}
	Let $\Pi' : \ccMbar^{\reg}(X) \lra{} \ccMbar'$
	be the natural projection map sending $(\nu,u,w)$ to $\nu$. 
	
	\begin{corollary} \label{cor curvemap}
		For each $(\nu,u,w) \in {\ccMbar'}^{\reg}(X)$, and any small neighborhoods
		$K \subset \Pi^{-1}(\Pi(\nu,u,w))$ of $(\nu,u,w)$ and
		$Q \subset \ccMbar'$ there is a continuous map $g$ fitting into the diagram:
				\[\begin{tikzcd}
			{Q \times K} & {{\ccMbar'}^{\reg}(X)} \\
			Q & {\ccMbar'}
			\arrow["\Pi", from=1-2, to=2-2]
			\arrow["g", from=1-1, to=1-2]
			\arrow["j"', hook, from=2-1, to=2-2]
			\arrow["proj"', from=1-1, to=2-1]
		\end{tikzcd}\]
		where $j$ is the natural inclusion map and $g$ is
		an open embedding.
		The restriction of $g$ to each fiber of the projection map $Q \times K \lra{} Q$ is smooth and it is equivariant under the natural action on $K$ given by the stabilizer group of the image of $\nu$ in $\ccMbar_{g,h}$ if $K$ is equivariant under this group.
		The choices involved in constructing such a product neighborhood $g$ ensure that if $g'$ was another such neighborhood then $g$ and $g'$ are $C^1_{loc}$-compatible as in \cite[Definition 4.27]{AMS-Hamiltonian}.
	\end{corollary}
	\begin{proof}
		Choose a complete metric on $\ccMbar_{g,h}$ and let $\exp : T\ccMbar_{g,h} \to \ccMbar_{g,h}$
		be the exponential map.
		Define
		\begin{equation}
			T := \exp \circ DP : T\ccMbar' \to \ccMbar_{g,h}.
		\end{equation}
		Define $\ccMbar'' := T\ccMbar'$ and
		let $\ccCbar'' \to \ccMbar''$ be the pullback of $\ccCbar_{g,h}$ via $T$.
		Since a submersion is locally a projection map,
		 Theorem \ref{theorem parameterized gluing}
		applies to $\ccCbar'' \to \ccMbar''$.
		Since $\ccCbar' \to \ccMbar'$ is the restriction of $\ccCbar'' \to \ccMbar''$
		to the zero section of $\ccMbar'' = T\ccMbar'$,
		the corollary follows since 
		a $C^1_{loc}$ structure restricted to a smooth submanifold of the base also admits a $C^1_{loc}$ structure.	
	\end{proof}

	\begin{cor} \label{cor algfam}
		Let $\ccCbar' \to \ccMbar'$ be a flat algebraic family of curves over a smooth base.
		Then the conclusion of Corollary \ref{cor curvemap} holds for this family of curves.
	\end{cor}
	\begin{proof}
		Since the statement of Corollary \ref{cor curvemap} is local in nature,
		it is sufficient to find for each $\nu \in \ccMbar'$, a neighborhood $U$ of $\nu$
		so that $\ccCbar'|_U \to U$ is the pullback of $\ccCbar_{g,h} \to \ccMbar_{g,h}$
		via a smooth map $f : U \to \ccMbar_{g,h}$.
		For the same reason, we can assume that the fibers of $\ccCbar'|_U \to U$  are connected.
		
		Such a map is constructed as follows:
		Choose local sections $s_1,\cdots,s_h$ disjoint from the nodes on a neighborhood $U$ of $\nu$
		so that for each $\nu' \in U$, we have that $\Sigma_{\nu'} := (\ccCbar'|_{\nu'},s_1(\nu'),\cdots,s_h(\nu'))$
		is an automorphism free marked curve.
		Then $f$ is the map sending $\nu'$ to $\Sigma_{\nu'} \in \ccMbar_{g,h}$.
	\end{proof}

	\subsection{Local Slice Charts} \label{subsec:FOOOlocalslices}
	
	We wish to compare our construction of a global Kuranishi chart
	with that of local `slice' Kuranishi charts constructed by Fukaya, Oh, Ohta, and Ono \cite{FO3}. 
	In this subsection we will recall such a construction.  This section and the next are not required for the results of this paper, but may be helpful to readers already familiar with \cite{FO3} and Kuranishi atlases.

	Let $(X,\omega)$ be a closed symplectic manifold, let
	$\beta \in H_2(X;\Z)$ and let $g,h \in \bN$.
	Let $J$ be an $\omega$-tame almost complex structure.
	Let $D \subset X$ be a smooth codimension $2$ submanifold with boundary whose closure is a compact submanifold with boundary and let $r \in \bN$.
	
	\begin{defn} \label{defn sliceforcurve}
		A \emph{$(D,r)$-slice curve}
		is a smooth map $u : \Sigma \lra{} X$ from a nodal curve $\Sigma$
		with $h + r$ marked points $p_1,\cdots,p_{h+r}$
		satisfying the following properties:
		\begin{enumerate}
			\item $u$ is transverse to $D$ meaning that all the nodes of $\Sigma$ map to $X - D$ and $u$ restricted to the smooth locus of $\Sigma$ is transverse to $D$,
			\item $u^{-1}(D) = \{p_{h+1},\cdots,p_r\}$ and
			\item the marked nodal curve $(\Sigma,p_1,\cdots,p_{h+r})$ has trivial automorphism group.
		\end{enumerate}
	\end{defn}
	
			We let $S_r$ denote the permutation group on $r$ elements, and let 
 $S_r$ act by permuting the last $r$ marked points $p_{h+1},\cdots,p_{h+r}$.
	We let $\ccMbar'_{g,h+r} \subset \ccMbar_{g,h+r}$ be the subspace of automorphism free curves
	and $\ccC'_{g,h+r}$ the corresponding universal curve and $\ccC^{'o}_{g,h+r}$ the complement of its nodes.
	\begin{defn} \label{bundleY2}
		Let $Y_{g,h+r} := \Omega^{0,1}_{\ccC^{'o}_{g,h+r}/\ccMbar'_{g,h+r}} \otimes_\C TX$
		be the $S_r$ vector bundle over $\ccC^{'o}_{g,h+r} \times X$ whose fiber over a point $(c,x) \in \ccC^{'o}_{g,h+r} \times X$ is the space of anti-holomorphic maps from $T_c(\ccC^{'o}_{g,h+r}  |_{ \Pi(c)})$ to $T_x X$, where $\Pi$ is the natural projection  map from $ \ccC^{'o}_{g,h+r}$ to $ 			\ccMbar'_{g,h+r}$.
	\end{defn}
	Choose a finite dimensional approximation scheme $(W_{\mu},\lambda_{\mu})_{\mu \in \bN}$
	for $C^\infty_c(Y_{g,h+r})$. 
	
	For the next definition, we use the fact that the domain of a $(D,r)$-slice curve has no automorphisms and so we can canonically identify the smooth locus of its domain with a fibre of $\Pi$.
	
	\begin{defn} \label{defn Fukayathickening}
		We define
		$\scrT_{g,h} = \scrT_{g,h}(\beta,D,r,\mu)$
		to be the space of triples $(\phi,u,e)$ where $\phi \in \ccMbar'_{g,h+r}$,  $u : \ccC'_{g,h+r}|_\phi \lra{} X$ is a $(D,r)$-slice curve
		and $e \in W_\mu$,  satisfying the following equation:
		\begin{equation} 
			\overline{\partial}_J u|_{\ccC^{'o}_{g,h+r}|_\phi} + (\lambda_\mu(e)) \circ \Gamma_u =0
		\end{equation}
		where
		\begin{equation}
			\Gamma_u : \ccC^{'o}_{g,h+r}|_\phi \to \ccC^{'o}_{g,h+r} \times X, \quad \Gamma_u(\sigma) := (\sigma,u(\sigma))
		\end{equation}				
		is the graph map.
		The symmetric group $S_r$ on $r$ elements acts in a natural way on $\scrT_{g,h}$
		by permuting the last $r$ marked points.
		We call $(S_r,\scrT_{g,h},\scrT_{g,h} \times W_\mu,0)$ a \emph{$(D,r)$ slice chart}.
	\end{defn}

	\begin{prop} \label{prop Kslice}
		For $\mu$ large enough, we have that $(S_r,\scrT_{g,h},\scrT_{g,h} \times W_\mu,0)$
		is a topological Kuranishi chart for an open subset of the moduli space $\ccMbar_{g,h}(X,J,\beta) $.
	\end{prop}
	\begin{proof}
		This is a consequence of Theorem \ref{theorem parameterized gluing}.
	\end{proof}

	\subsection{Comparing Global Kuranishi Charts with Slice Charts}
	\label{sec:comp-glob-kuran}
	
	We will now compare appropriate open subsets of our global Kuranishi chart construction from Subsection \ref{Subsec:easier}
	and the slice charts from Section \ref{subsec:FOOOlocalslices}.
	In order to do this, we first need to describe
	the slice chart construction from Section \ref{subsec:FOOOlocalslices} in terms of global
	Kuranishi data (Definition \ref{defnnoncomplexthickening}).
	
	In addition to the choices made in the previous section, let us fix an integer $d \in \bN$. 
	
	We will now construct a  Kuranishi chart
	of curves which intersect $D$ transversally.
	This will be a certain codimension $0$ submanifold
	of the global Kuranishi chart constructed
	in Section \ref{Subsec:easier}.

	Before we do this, we will first construct the open subset of the global Kuranishi data $(\scrT,s)$
	from Example \ref{examplegl} that will be related to our slice chart above.
	We let $G, \scrF, \scrC, \scrF^{(2)}, \Delta, (V_\mu,\lambda_\mu)_{\mu \in \bN}, \scrT, s$
	be as in Example \ref{examplegl}. Let $\mu \gg 1$ be large.
	We let $B_\scrT \subset\scrT^\pre(\beta,\lambda_\mu)$
	be the subspace of tuples $(\phi,u,e)$
	where $u : \Sigma \to X$
	is transverse to $D$ and for which there are
	are marked points $p_1,\cdots,p_{h+r}$ on $\Sigma$
	where $p_1,\cdots,p_h$ are the marked points coming from $\scrC|_\phi$
	and where the additional $r$ marked points
	make $u$ into a $(D,r)$-slice curve as in Definition \ref{defn sliceforcurve}.
	We let $\scrT_D \subset \scrT$ be the preimage of $B_\scrT$ in $\scrT$ under the natural projection map $\scrT \to \scrT^{\pre}$.
	We let $s_D := s|_{\scrT_D}$.
	Then $(\scrT_D,s_D)$ is  Kuranishi chart data describing the moduli space of $J$-holomorphic
	curves that are transverse to $D$ and which are $(D,r)$-slice
	when adding an additional $r$ marked points.

	Let us now describe the the slice chart from Section \ref{subsec:FOOOlocalslices}
	above in terms of  Kuranishi chart data.
	Let $\check{G}$ be the permutation group on $r$ elements.
	We let $\check{\scrF} := \ccMbar_{g,h+r}$
	be the moduli space of genus $g$ curves
	with $h+r$ marked points
	and let $\check{\scrC} \to \check{\scrF}$ the correpsonding universal curve.
	We let $\check{\scrF}^{(2)} := \check{G} \times \check{\scrF}$
	and we let $\check{\scrF}^{(2)} := (\check{\scrF}^{(2)},i,p)$ be the corresponding microbundle
	where $p : \check{\scrF}^{(2)} \to \check{\scrF}$ is the natural projection map and $i : \check{\scrF} \to \check{\scrF}^{(2)}$
	sends $\phi$ to $id \times \phi$.
	Choose a finite dimensional approximation scheme $(\check{\scrT}_\mu,\check{\lambda}_\mu)_{\mu \in \bN}$
	for $\check{\scrC}$.
	Fix $\mu \in \bN$ large.
	Let $B_{\check{\scrT}} \subset\scrT^\pre(\beta,\check{\lambda}_\mu)$
	be the subspace of $(D,r)$-slice curves as in Definition \ref{defn sliceforcurve}.
	We now let $\check{\scrT} \:= \check{G} \times B_{\check{\scrT}}$.
	We let $\check{s} : \check{\scrT} \to \check{\scrF}^{(2)}$
	be the natural projection map to $\check{G} \times \check{\scrF} = \check{\scrF}^{(2)}$. By construction: 

	\begin{lemma}
		The  Kuranishi chart associated to $(\check{\scrT},\check{s})$
		is equal to a $(D,r)$-slice chart as in Definition \ref{prop Kslice}. \qed
	\end{lemma}

	The following proposition tells us that the global Kuranishi chart
	for our moduli space of curves constructed in Section \ref{Subsec:easier}
	is locally equivalent up to stabilization, germ equivalence and group enlargement
	to the $(D,r)$-slice chart from Definition \ref{prop Kslice}.

	\begin{prop}
		$(G,\scrT,s)$ and $(\check{G},\check{\scrT},\check{s})$
		are related by a sequence of
		group enlargements 
		and general stabilizations (Definition \ref{defngenearlstablization}) and their inverses.
		Hence their associated  Kuranishi charts are isomorphic.
	\end{prop}
	\begin{proof}
		This is done in three stages.

		We let $G_1 := G \times \check{G}$.
		We let $\scrF_1 := \scrF_{g,h+r,d}$
		and $\scrC_1 \to \scrF_1$ the corresponding universal curve.
		We let $\scrF^{(2)}_1 := \check{G} \times \scrF^{(2)}_{g,h+r,d}$ (see Section \ref{Subsec:moduli}).
		We think of this as a smooth microbundle $\scrF^{(2)}_1 = (\scrF^{(2)}_1,i,p)$ over $\scrF_1$
		where $i$ sends $\phi$ to $(id,\phi)$ and $p$ is the natural projection map to $\scrF^{(2)}_1$.
		The group $G_1$ acts on these moduli spaces as follows: $G$ acts in the natural way and
		$\check{G}$ permutes the last $r$ marked points.
		We let $(V_{\mu,1},\lambda_{\mu,1})_{\mu \in \bN}$ be the pullback of $(V_\mu,\lambda_\mu)_{\mu \in \bN}$
		to $\scrC_1$ (Definition \ref{defnpulbackoffdscheme}).
		We let $B_{\scrT_1} \subset \scrT^\pre(\beta,\lambda_{\mu,1})$
		be the subspace of $(D,r)$-slice curves.
		We let $P' \to B_{\scrT_1}$ be the principal $G$-bundle over $B_{\scrT_1}$
		whose fiber over $(\phi,u)$ is the space of unitary bases of $H^0(L_u^{\otimes k})$.
		We let $P_1 := \check{G} \times P'$ be the corresponding $G_1$-bundle over $B_{\scrT_1}$.
		We let $\scrT_1 := (P_1)_\bC$.
		Then elements of $\scrT_1$ correspond to tuples $(\check{g},\phi,u,e,F)$
		where $\check{g} \in \check{G}$, $(\phi,u) \in\scrT^\pre(\beta,\lambda_{\mu,1})$,
		$e \in V_{\mu,1}$ and $F$ is a basis of $H^0(L_u^{\otimes k})$.
		We let $s_1 : \scrT_1 \to \scrF^{(2)}_1$ send $(\check{g},\phi,u,e,F)$
		to $(id,(\phi,\phi_F))$.
		Then $(\scrT_1,s_1)$ is  Kuranishi data for $(G_1,\scrF^{(2)}_1,\lambda_{\mu,1})$.
		All the curves involved are $(D,r)$-slice curves,
		we have that the  Kuranishi chart associated to $(\scrT_1,s_1)$
		is a group enlargement of the one associated to $(\scrT_D,s_D)$.
		(This group enlargement is not an enlargement obtained from Lemma \ref{lemmagroupenlargement}, i.e. it is not induced by a principal bundle enlargement.)

		Now let $(V_{\mu,2},\lambda_{\mu,2})_{\mu \in \bN}$ be the direct sum of
		$(V_{\mu,2},\lambda_{\mu,2})_{\mu \in \bN}$
		and the pullback of $(\check{\scrT}_\mu,\check{\lambda}_\mu)_{\mu \in \bN}$
		under the natural projection map $\scrF_1 \to \check{\scrF}$ given by sending a curve to its domain.
		We now let $P_2$ be the pullback of $P_1$ to $\scrT^\pre(\beta,\lambda_{\mu,2})$.
		Let $\scrT_2 := (P_2)_\bC$ which is the pullback of $\scrT_1$.
		Let $s_2$ be the composition of the projection map to $\scrT_1$ with $s_1$.
		Then $(\scrT_2,s_2)$ is	an approximation stabilization of $(\scrT_1,s_1)$
		as in Definition \ref{defnapproxenlargement}.

		Let $(V_{\mu,3},\lambda_{\mu,3})_{\mu \in \bN}$ be the pullback of $(\check{\scrT}_\mu,\check{\lambda}_\mu)_{\mu \in \bN}$
		under the natural projection map $\scrF_1 \to \check{\scrF}$ given by sending a curve to its domain.
		We let $B_{\scrT_3} \subset\scrT^\pre(\beta,\lambda_{\mu,3})$
		be the subspace of $(D,r)$-slice curves.
		We let $P_3 \to B_{\scrT_3}$ be the $\check{G}$-equivariant principal $G$-bundle
		whose fiber over $(\phi,u)$ is a unitary basis of $H^0(L_u^{\otimes k})$.
		Let $\scrT_3 = (P_3)_\bC$
		and let $s_1 : \scrT_1 \to \scrF^{(2)}_1$ send $(\check{g},\phi,u,e,F)$
		to $(id,(\phi,\phi_F))$.
		Then $(\scrT_2,s_2)$ is	an approximation stabilization of $(\scrT_3,s_3)$
		as in Definition \ref{defnapproxenlargement}.	In addition,  it is a group enlargement of $(\check{\scrT},\check{s})$.
	\end{proof}

\bibliographystyle{habbrv}
\bibliography{mybib}

\end{document}